\documentclass[11pt,letterpaper,twoside,reqno]{amsart}
\usepackage{amsmath,amsfonts,amsbsy,amsgen,amscd,mathrsfs,amssymb,amsthm,mathtools,tensor}
\usepackage{siunitx}
\usepackage{algorithm,algorithmic}
\usepackage[shortlabels]{enumitem}

\usepackage{authblk}
\usepackage[foot]{amsaddr}

\usepackage[toc, page]{appendix}

\usepackage[margin=1in]{geometry}
\usepackage[dvipsnames]{xcolor}
\definecolor{MyColor}{HTML}{0047AB}
\usepackage{fancyhdr}

\usepackage{hyperref}
\hypersetup{
colorlinks=true, 
linktoc=all,     
linkcolor=blue,  
}

\usepackage{etoolbox}
\makeatletter
\renewcommand{\@secnumfont}{\bfseries}
\makeatother
\patchcmd{\section}{\scshape}{\bfseries}{}{}
\patchcmd{\section}{\normalfont}{\normalfont\color{MyColor}}{}{}
\patchcmd{\subsection}{\normalfont}{\normalfont\color{MyColor}}{}{}
\makeatletter
\def\subsubsection{\@startsection{subsubsection}{3}%
\z@{.5\linespacing\@plus.7\linespacing}{-.5em}%
{\normalfont\bfseries}}
\makeatother

\usepackage{graphicx}
\usepackage{float}
\usepackage{caption}
\usepackage{subcaption}
\usepackage{tikz}

\usepackage{booktabs} 
\usepackage{multirow}
\usepackage{tabu} 

\usepackage[normalem]{ulem} 

\usepackage[color=yellow]{todonotes}

\newtheorem{theorem}{Theorem}[section]
\newtheorem{lemma}[theorem]{Lemma}
\newtheorem{definition}[theorem]{Definition}

\newtheorem{proposition}[theorem]{Proposition}

\numberwithin{equation}{section}

\newtheorem{remark}[theorem]{Remark}
\newtheorem{example}{Example}[section]

\DeclareMathOperator*{\argmin}{arg\,min}

\makeatletter
\@tfor\next:=abcdefghijklmnopqrstuvwxyzABCDEFGHIJKLMNOPQRSTUVWXYZ\do{
\def\command@factory#1{
\expandafter\def\csname b#1\endcsname{\mathbf{#1}}
\expandafter\def\csname fk#1\endcsname{\mathfrak{#1}}
\expandafter\def\csname bb#1\endcsname{\mathbb{#1}}
\expandafter\def\csname cl#1\endcsname{\mathcal{#1}}
\expandafter\def\csname bcl#1\endcsname{\mathbfcal{#1}}
}
\expandafter\command@factory\next
}

\makeatletter
\@namedef{subjclassname@2020}{\textup{2020} Mathematics Subject Classification}
\makeatother

\begin{document}

\title[Gradient flows for regularised MDPs in  Polish spaces]{A Fisher--Rao gradient flow for entropy-regularised Markov decision processes in  Polish spaces}

\author{Bekzhan Kerimkulov$^1$}
\email{b.kerimkulov@ed.ac.uk}
\author{James-Michael Leahy$^2$}
\email{j.leahy@imperial.ac.uk}
\author{David Siska$^1$}
\email{d.siska@ed.ac.uk}
\author{Lukasz Szpruch$^1$}
\email{l.szpruch@ed.ac.uk}
\author{Yufei Zhang$^2$}
\email{yufei.zhang@imperial.ac.uk}

\address{$^1$ School of Mathematics, University of Edinburgh, United Kingdom}
\address{$^2$ Department of Mathematics, Imperial College London, United Kingdom}




\maketitle
\begin{abstract}

We study the global convergence of a Fisher--Rao policy gradient flow for infinite-horizon entropy-regularised Markov decision processes with Polish state and action spaces. 
The flow is  a continuous-time analogue of a  policy mirror descent method.
We establish the global well-posedness of the gradient flow and demonstrate its exponential convergence to the optimal policy. Moreover, we prove the flow is stable with respect to gradient evaluation, offering insights into the performance of a natural policy gradient flow with log-linear policy parameterisation. To overcome challenges stemming from the lack of the convexity of the objective function and the discontinuity arising from the entropy regulariser, we leverage the performance difference lemma  and the duality relationship between the gradient and mirror descent flows.
Our analysis  provides a theoretical foundation for developing various 
discrete policy gradient algorithms.
\end{abstract}

{\hypersetup{linkcolor=MyColor}
\setcounter{tocdepth}{2}
\tableofcontents}

\section{Introduction}
\label{sec:introduction}

Policy gradient algorithms are a cornerstone of recent successes of reinforcement learning (RL) \cite{sutton2018reinforcement}.
Gradient-based algorithms like policy gradient methods \cite{sutton1999policy}, actor-critic methods \cite{haarnoja2018soft}, and mirror descent-based methods \cite{kakade2001natural, tomar2020mirror}  have proven highly effective. 
They provide adaptability, especially when paired with powerful function approximations, to handle  Markov Decision Processes (MDPs)  with continuous state and action spaces \cite{doya2000reinforcement, van2012reinforcement, manna2022learning}.  
Despite the practical success,  a mathematical theory that provides guarantees for the convergence of these algorithms has been elusive, partially due to the pervasive nonconvexity inherent in the objective function.
Indeed, existing theoretical works are largely confined to MDPs with a finite number of states and actions, namely the tabular setting. A convergence analysis of policy optimisation algorithms with function approximation for MDPs with general state and action spaces remains a challenging open problem.  

Analysing gradient-based algorithms for continuous state and action spaces presents unique technical challenges compared to their discrete counterparts. 
For tabular MDPs, algorithm convergence is often achieved by ensuring that the algorithm has a positive probability of uniformly visiting all states and actions, with the convergence guarantees frequently depending on the cardinality of the state and action spaces; see Section~\ref{sec:intro_lit_rev}  for more details.
These arguments and error bounds do not extend to continuous state and action spaces.
To deal with continuous state and action spaces,  new analytical approaches are required to investigate the optimization objective landscape and develop convergent algorithms. 
Incorporating function approximation into the algorithm further introduces function approximation error and complicates the analysis.

This work sheds light on this challenging problem in the context of entropy-regularised MDPs, which have gained considerable attention due to their impressive empirical performance and favorable theoretical properties \cite{haarnoja2017reinforcement,  geist2019theory, agarwal2020optimality, mei2020global, cayci2021linear, lan2023policy, zhan2023policy}.

We introduce a continuous-time 
{policy mirror descent algorithm} for solving entropy-regularised MDPs with state and action spaces being  Polish spaces. This gradient flow naturally extends the Fisher--Rao flow for probability measures \cite{gallouet2017jko, liu2023polyak} to the present setting with conditional probability measures. 
We prove that the flow is globally well-posed and exhibits exponential convergence to the optimal policy. Furthermore, we establish the flow's stability under inexact gradient evaluation, and further   deduce performance guarantees for a natural policy gradient flow with log-linear parameterised policies.
Our analysis provides a theoretical foundation for deriving discrete policy gradient algorithms
(see Section~\ref{sec:discrete}
and Appendix~\ref{sec:numerical})
and offers insights into the optimal convergence rate achievable by discrete-time policy gradient methods in general settings (see Remarks~\ref{rmk:convergence} and~\ref{rmk:discrete_mirror}).

\subsection{Overview of main results}
\label{sec:intro_overview}
In this section, we provide a road map of the key ideas and contributions of this work without introducing needless technicalities. The precise assumptions and statements of the results can be found in Section~\ref{sec:main_result}.

\subsubsection*{Regularised Markov decision process}
Consider an infinite horizon Markov decision model $(S,A,P,c,\gamma)$,  
where the state space $S$ and action space $A$ are Polish spaces with possibly   infinite cardinality, $P\in \clP(S|S\times A)$ is the transition probability kernel, $c$ is a bounded cost function, and $\gamma\in [0,1)$ is the discount factor. 
Let $\mu \in \mathcal \clP(A)$ denote a reference probability measure and  $\tau>0$ denote a regularisation parameter.   
For each  stochastic policy $\pi \in \clP(A|S)$ and $s\in S$, define the regularised value function by
\begin{equation}
\label{eq:value_introduction}
V^{\pi}_{\tau}(s)= \bbE_{s}^{\pi}\left[\sum_{n=0}^\infty\gamma^n \Big(c(s_n,a_n) + \tau \operatorname{KL}(\pi(\cdot|s_n)|\mu)\Big)\right]\in \bbR\cup \{\infty\}\,,
\end{equation}
where $\operatorname{KL}(\pi(\cdot|s)|\mu) 
$ is the    Kullback--Leibler (KL) divergence  of $\pi(\cdot|s)$ with respect to $\mu$, defined as  
$
\operatorname{KL}(\pi(\cdot|s)|\mu)
\coloneqq \int_{A} \ln\frac{\mathrm{d} \pi(\cdot|s)}{\mathrm{d}\mu}(a) \pi(d a|s)  
$ if $\pi(\cdot |s)$ is absolutely continuous with respect to $\mu$,
and infinity otherwise. 
For a fixed $\rho \in \clP(S)$,   we consider the following minimisation problem:
\begin{equation}
\label{eq:value_rho_introduction}
\min_{\pi\in \mathcal P(A|S)}
V^{\pi}_{\tau}(\rho),
\quad 
\textnormal{with $
V_\tau^\pi(\rho) \coloneqq  \int_S V_\tau^\pi(s)\rho(ds)
$}.
\end{equation}

This paper proposes and analyzes a
gradient flow 
in  $\clP(A|S)$ for the 
minimisation problem
\eqref{eq:value_rho_introduction}. 
By the Bellman principle,~\eqref{eq:value_rho_introduction} admits a unique optimal policy 
$\pi^*_{\tau}\in \Pi_{\mu}=\{\pi \in \clP(A|S): \ln \frac{\rm d \pi}{\rm d\mu}\in B_b(S\times A)\}$,  which is independent of $\rho$     
and satisfies 
\[
\pi^*_{\tau}(da|s) = \exp\left(-\frac{1}{\tau }(Q^{\pi^*_{\tau}}_{\tau}(s,a)-V^{\pi^*_{\tau}}_{\tau}(s))\right)\mu(da)\,,
\quad \forall s\in S\,,
\]
where for all $\pi\in \Pi_{\mu}$, 
$Q^\pi$ is  the  associated $Q$-function  defined by
\[
Q^{\pi}_\tau(s,a)= c(s,a) + \gamma \int_S V_{\tau}^{\pi}(s')P(ds'|s,a)\,.
\]
Note that  $\pi^*_\tau$ exists for merely measurable transition kernels and cost functions, thanks to the KL divergence in \eqref{eq:value_introduction}; see Remark~\ref{rmk:KL_role} for a more detailed discussion on its role in our analysis.

\subsubsection*{Flat derivative and  Fisher--Rao flow}
The minimisation problem \eqref{eq:value_rho_introduction} is contrained  over all transition probability kernels,  which form a convex subset of the space of bounded signed kernels $b\clM(A|S)$.
Hence, the gradient of $\pi\mapsto V^\pi_\tau(\rho)$ must be chosen carefully so that the updates remain in the set  $\clP(A|S)$. 
This requires identify a proper dual space of $\clP(A|S)$.

To this end, for a given $\nu \in \clP(S)$, define a duality pairing $\langle \cdot, \cdot\rangle_{\nu}: B_b(S\times A)\times b\clM(A|S)\rightarrow \mathbb{R}$  by
\begin{equation}
\label{eq:dual_pair}
\langle Z, m\rangle_{\nu} = \frac{1}{1-\gamma}\int_{S}\int_{A} Z(s,a)m(da|s)\nu(ds)\,,\quad (Z, m)\in B_b(S\times A)\times b\clM(A|S).  
\end{equation}
The flat derivative of $V^{\cdot}_{\tau}(\rho)$ relative to the duality pairing $\langle \cdot, \cdot\rangle_{\nu}$ is a map  
$\frac{\delta V^{\cdot}_{\tau}(\rho)}{\delta \pi}\big|_{\nu}:  \Pi_{\mu}\to B_b(S\times A)$ such that for every $\pi,\pi'\in \Pi_\mu$,
\begin{equation*}
\lim_{\varepsilon\searrow 0}\frac{V^{(1-\varepsilon)\pi+\varepsilon\pi'}_\tau(\rho) - V^\pi_\tau(\rho)}{\varepsilon} = \left\langle \frac{\delta V^{\pi}_{\tau}(\rho)}{\delta \pi}\bigg|_{\nu}, \pi' - \pi  \right\rangle_\nu\,\,\,
\text{and}\,\,\,\left\langle \frac{\delta V^{\pi}_{\tau}(\rho)}{\delta \pi}\bigg|_{\nu}, \pi  \right \rangle_\nu = 0.    
\end{equation*}
One can show that (see Appendix \ref{sec:policy_gradient}),  the flat derivative is given by
\begin{equation}
\label{eq:delta_V_delta_pi-intro}
\frac{\delta V^{\pi}_{\tau}(\rho)}{\delta \pi}\bigg|_{\nu}(s,a)= \left(Q^{\pi}_\tau (s,a) + \tau \ln\frac{\mathrm{d}  \pi}{\mathrm{d}  \mu} (s,a) -V^{\pi}_\tau(s)\right)\frac{
\mathrm{d}
d^{\pi}_{\rho}}{\mathrm{d} \nu}(s)\,,
\end{equation}
where $d^{\pi}_{\rho}\in \clP(S)$ is the occupancy measure associated with $\pi$.
The flat derivative \eqref{eq:delta_V_delta_pi-intro} 
represents 
the  first-order  variation of the value function with respect to perturbations within
the policy space $\Pi_\mu$ (see Appendix~\ref{sec:policy_gradient} for a more detailed comparison with the classical policy gradient formula established for tabular MDPs).
It   generalises the notation of the flat derivative applied to probability measures (see e.g., \cite{liu2023polyak}) to encompass probability transition kernels.

Using the flat derivative,  we propose the following  gradient flow for \eqref{eq:value_rho_introduction}:
\begin{equation}
\label{eq:gradient_flow_introduction}
\begin{aligned}
\partial_t \pi_t(da|s) &= -\frac{\delta V^{\pi_t}_{\tau}}{\delta \pi}(s,a)\pi_t(da|s),\quad t>0\,.
\end{aligned}
\end{equation}
where for each $\pi$,  $\frac{\delta V^{\pi}_{\tau}}{\delta \pi}$ denotes  
the flat derivative \eqref{eq:delta_V_delta_pi-intro} relative to the pairing  
$\langle \cdot, \cdot\rangle_{d^{\pi}_{\rho}}$:
\begin{equation}
\label{eq:delta_V_delta_pi-intro-state-dep-pairing}
\frac{\delta V^{\pi}_{\tau}}{\delta \pi}\coloneqq \frac{\delta V^{\pi}_{\tau}(\rho)}{\delta \pi}\bigg|_{d^{\pi}_{\rho}}= Q^{\pi}_\tau  + \tau \ln\frac{\mathrm{d}  \pi}{\mathrm{d}  \mu}  -V^{\pi}_\tau\,.
\end{equation}

We refer to \eqref{eq:gradient_flow_introduction} as the Fisher--Rao flow, since in the special case when $S=\emptyset$, it reduces to the Fisher--Rao gradient flow for probability measures (see, e.g., \cite{liu2023polyak}).
Multiplying the first variation $\frac{\delta V^{\pi_t}_{\tau}}{\delta \pi}(s,a)$ by $ \pi_t$ alongside $\left\langle \frac{\delta V^{\pi_t}_{\tau}(\rho)}{\delta \pi}\big|_{d^{\pi_t}_{\rho}}, \pi_t  \right \rangle_{d^{\pi_t}_{\rho}} = 0$  ensures the flow remains within $\clP(A|S)$. This multiplication can also be viewed as pre-conditioning the flat derivative with  
the Hessian of state-integrated negative entropy.
In Appendix \ref{sec:natural_gradient_flows}, we show this entropy function induces a Riemannian metric on $\Pi_{\mu}$ and~\eqref{eq:gradient_flow_introduction} corresponds to a Riemannian gradient flow 
with respect to this metric, see~\eqref{eq:natural_gradient_pi_flow}.
This Riemannian interpretation of the Fisher--Rao flow 
is not required in our subsequent analysis, but
has later been used in~\cite{muller2024fisher}
for optimizing over state-action occupation measures for tabular MDPs.

\subsubsection*{Connection to mirror descent}
We now highlight   that the flow~\eqref{eq:gradient_flow_introduction} 
coincides with the  continuous-time limit of a mirror descent algorithm. This observation not only offers an alternative motivation for the flow \eqref{eq:gradient_flow_introduction}, but is also essential for its convergence analysis.

Specifically, let $\lambda>0$, $\pi_0\in \Pi_{\mu}$ and $n\in \bbN$, and consider  the  mirror descent update   given by 
\begin{align}
\label{eq:mirror_des_direct_intro}
\begin{split}
\pi^{n+1} &= \underset{\pi \in \clP(A|S)}{\operatorname{arg\,min}}\left[ \left\langle \frac{\delta V^{\pi^n}_\tau}{\delta \pi}, \pi-\pi^n \right\rangle_{d_{\rho}^{\pi^n}}  + \frac{\lambda}{1-\gamma}\int_{S}\operatorname{KL}(\pi(\cdot|s)|\pi^n(\cdot|s))d^{\pi^n}_{\rho}(ds)\right]\\
&=  \underset{\pi \in \clP(A|S)}{\operatorname{arg\,min}}\int_{S}\left[\int_{A}\frac{\delta V^{\pi^n}_\tau}{\delta \pi}(s,a)(\pi-\pi^n)(da|s) + \lambda \operatorname{KL}(\pi(\cdot|s)|\pi^n(\cdot|s))\right]d^{\pi^n}_{\rho}(ds)
\,.
\end{split}
\end{align}
The $\operatorname{KL}$-divergence term is an adaptive Bregman divergence induced by a state-integrated negative entropy, where the integration measure $d^{\pi^n}_{\rho}$ is chosen adaptively based on the current iterate $\pi^n$;
see \eqref{def:state_integrated_entropy} in Appendix \ref{sec:natural_gradient_flows}.
The minimum above can be achieved by a pointwise in $s\in S$ optimization:
\begin{equation}\label{eq:pointwise_min}
\pi^{n+1}(\cdot|s)= \underset{m \in \mathcal P(A)}{\operatorname{arg\,min}}\left[ \int_A\frac{\delta V^{\pi^n}_\tau}{\delta \pi}(s,a)(m(da) - \pi^n(da|s)) + \lambda \operatorname{KL}(m|\pi^n(\cdot|s))\right]\,,
\end{equation}
which aligns with the policy mirror descent algorithm studied in \cite{lan2022policy} (i.e., Algorithm 1) and \cite{xiao2022convergence}. By~\cite[Lemma 1.4.3]{dupuis1997weak}, the pointwise minimization in \eqref{eq:pointwise_min} is achieved by
\begin{equation}\label{eq:Dupuis_update}
\frac{\mathrm d \pi^{n+1}}{\mathrm d \pi^n}(s,a) = \frac{\exp{\left(-\frac1\lambda\frac{\delta V^{\pi^n}_\tau}{\delta \pi}(s,a)\right)}}{\int_A \exp{\left(-\frac1\lambda \frac{\delta V^{\pi^n}_\tau}{\delta \pi}(s,a')\right)\pi^n(da'|s)}}\,.
\end{equation}
Taking the logarithm of \eqref{eq:Dupuis_update} and rearranging the terms yield
\begin{equation}
\label{eq:explict_euler_for_mirror_intro}
\lambda \left(\ln \frac{\mathrm d \pi^{n+1}}{\mathrm d \mu}(s,a) - \ln \frac{\mathrm d \pi^{n}}{\mathrm d \mu}(s,a)\right) = -\frac{\delta V^{\pi^n}_\tau}{\delta \pi}(s,a)
- \lambda \ln \int_A e^{-\frac1\lambda \frac{\delta V^{\pi^n}_\tau}{\delta \pi}(s,a')}\pi^n(da'|s)\,.
\end{equation}
Interpolating in the time variable and letting $\lambda \to \infty$, we obtain
\begin{equation*}
\partial_t \ln \frac{\mathrm d \pi_t}{\mathrm d \mu}(s,a) = -\left(\frac{\delta V^{\pi_t}_\tau}{\delta \pi} (s,a) -\int_A \frac{\delta V^{\pi_t}_\tau}{\delta \pi}(s,a')\pi_t(da'|s)\right)=-\frac{\delta V^{\pi_t}_\tau}{\delta \pi} (s,a)\,.
\end{equation*}
This along with \eqref{eq:delta_V_delta_pi-intro-state-dep-pairing} implies   $Z_t: = \ln \frac{\mathrm d \pi_t}{\mathrm d \mu}$ 
satisfies  
\begin{equation}\label{eq:mirror_descent_intro}
\partial_t Z_t (s,a) 
=-(Q^{\pi_t}_{\tau}(s,a)-V^{\pi_t}_{\tau}(s)+\tau Z_t(s,a))\,, \quad \pi_t(da|s) :=\boldsymbol{\pi}(Z_t)(da|s) \,,
\quad t>0\,,
\end{equation}
where the map $\boldsymbol{\pi}: B_b(S\times A)\rightarrow \Pi_{\mu}$ is defined by
\begin{equation}\label{eq:bold_pi_intro}
\boldsymbol{\pi}(Z)(da|s)= \frac{e^{Z(s,a)}}{\int_A e^{Z(s,a')}\mu(da')}\,\mu(da)\,.
\end{equation}

The flow \eqref{eq:mirror_descent_intro} is a continuous-time mirror descent flow where $\boldsymbol{\pi}: B_b(S\times A)\rightarrow \Pi_{\mu}$ is a  mirror map from   the dual space $B_b(S\times A)$ to the primal space $\Pi_{\mu}$.
It can be viewed as a Riemannian gradient flow in  $B_b(S\times A)$ with respect to the metric induced by the Hessian of the convex conjugate of a (state-integrated) negative entropy (see \eqref{eq:natural_gradient_mirror_Z}   in Appendix \ref{sec:natural_gradient_flows}), which is an  analog of the Fisher-information matrix over the infinite-dimensional policy class  $\{\boldsymbol{\pi}(Z)\mid Z\in B_b(S\times A)\}$. This observation generalises the well-known connection between mirror descent and natural gradient flows observed in the finite-dimensional setting (see, e.g., \cite{raskutti2015information}).

To show the flow $(\pi_t)_{t\ge 0}$ given by \eqref{eq:mirror_descent_intro} indeed satisfies the Fisher--Rao flow \eqref{eq:gradient_flow_introduction},
consider the following generalisation  of \eqref{eq:mirror_descent_intro}: 
\begin{equation}\label{eq:mirror_descent_intro_baseline}
\partial_t Z_t (s,a) = -\frac{\delta V^{\pi_t}_\tau}{\delta \pi}(s,a)+f_t(s)\,, \quad \pi_t(da|s) :=\boldsymbol{\pi}(Z_t)(s,a)\,,\quad t>0\,,
\end{equation}
with some given  $f_t\in B_b(S)$. Applying the chain rule to $t\mapsto \pi_t(da|s)=\boldsymbol{\pi}(Z_t)$ yields
\begin{equation}\label{eq:grad_flow_density_intro}
\frac{1}{\frac{\mathrm d \pi_t}{\mathrm d \mu}}\partial_t \frac{\mathrm d \pi_t}{\mathrm d \mu} = -\frac{\delta V^{\pi_t}_\tau}{\delta \pi}\,,
\end{equation}
which is the formulation  of~\eqref{eq:gradient_flow_introduction} in terms of density functions (Lemma~\ref{lemma:dpi_t}). Interestingly,  the primal flow \eqref{eq:gradient_flow_introduction} is invariant under different choices of the baseline $f$ in \eqref{eq:mirror_descent_intro_baseline}.

Here we highlight that both the Fisher--Rao flow~\eqref{eq:gradient_flow_introduction} and the mirror flow~\eqref{eq:mirror_descent_intro}
provide a foundation for deriving various discrete policy
gradient algorithms through appropriate timestepping schemes.
In fact, 
an explicit Euler discretization of~\eqref{eq:gradient_flow_introduction} with an orthogonal projection yields the proximal gradient method in~\cite{liu2023projected}, an implicit Euler discretization of~\eqref{eq:gradient_flow_introduction} yields the  proximal point method, and an explicit Euler discretization of 
\eqref{eq:mirror_descent_intro} yields  
the mirror descent method~\cite{lan2022policy, xiao2022convergence, zhan2023policy}.
See Section~\ref{sec:discrete} for more details on the mirror descent method and Appendix~\ref{sec:numerical} for new policy gradient methods derived from higher-order discretizations of~\eqref{eq:mirror_descent_intro}.

\subsubsection*{Our contributions.}

This paper rigorously analyses the Fisher--Rao flow \eqref{eq:gradient_flow_introduction} by leveraging its connection with its dual formulation \eqref{eq:mirror_descent_intro}.

\begin{itemize}
\item We show that if~\eqref{eq:mirror_descent_intro} has a solution and $\pi_t(da|s) \propto e^{Z_t(s,a)}\mu(da)$, then $V^{\pi_t}_{\tau}(\rho)$ is differentiable in $t$ and $\partial_t V^{\pi_t}_\tau(s) \leq 0$ for all $s\in S$ (Proposition~\ref{prop:value_monotone}).  Then by using the value function as a Lyapunov function, we prove the the global well-posedness of~\eqref{eq:mirror_descent_intro} (Theorem~\ref{thm:wp_MD}). By further leveraging the correspondence between  \eqref{eq:gradient_flow_introduction}  and \eqref{eq:mirror_descent_intro} (Lemma~\ref{lemma:dpi_t}), we establish that for suitable initial conditions, the Fisher--Rao flow \eqref{eq:gradient_flow_introduction}
is globally well-posed.
\item We prove that
the value function along~\eqref{eq:gradient_flow_introduction} converges exponentially fast to the optimal value function, and the  policies also   converge exponentially fast to the optimal policy (Theorem \ref{thm:linear_convergence}).
The analysis leverages 
a performance difference lemma,
which quantifies 
the sub-optimality of an arbitrary policy $\pi$  by the KL-divergence between the policy and the optimal policy $\pi^*_\tau$ (Lemma~\ref{lem:performance_diff}).  
\item We prove the stability of the Fisher--Rao flow \eqref{eq:gradient_flow_introduction} when the gradient direction \eqref{eq:delta_V_delta_pi-intro-state-dep-pairing} is evaluated approximately. Specifically, we show that such an approximate Fisher--Rao flow exhibits similar exponential convergence with an additional term depending explicitly on the gradient evaluation error (Theorem \ref{thm:stability_mirror}). This stability result allows for establishing an exponential convergence of a natural policy gradient flow involving log-linear parametrised policies (Theorem \ref{thm:NPG_stability}); see below for more details.

\item 
We propose discrete policy gradient algorithms through appropriate time-stepping schemes for \eqref{eq:gradient_flow_introduction}. An explicit Euler discretization yields the mirror descent update \eqref{eq:mirror_des_direct_intro}, whose  convergence rate is analyzed in
Theorem \ref{thm:convergence_mirror_steps}. The performances of higher-order discretizations are examined numerically in Appendix~\ref{sec:numerical},  demonstrating that a carefully chosen higher-order discretization of the continuous-time flow may result in smaller error than the standard mirror descent update. 

\item 
We discuss the polynomial convergence of~\eqref{eq:gradient_flow_introduction} for unregularized MDPs (with $\tau = 0$ in~\eqref{eq:value_rho_introduction}) in Section~\ref{sec:unregularised}, and the additional challenges compared to the regularized cases.

\end{itemize}

To the best of our knowledge, this is the first work providing a  rigorous analysis of policy gradient and policy mirror descent methods for MDPs with general state and action spaces and general stochastic policies.

\subsubsection*{Heuristic derivation of exponential convergence.} To illustrate the key ideas behind the convergence analysis of \eqref{eq:gradient_flow_introduction}, we temporarily assume that~\textit{the flow \eqref{eq:gradient_flow_introduction} is well-posed, the map $t\mapsto \int_S \operatorname{KL}(\pi^*_\tau (\cdot|s)|\pi_t(\cdot|s))d^{\pi^*_\tau}_{\rho}(ds)$ is differentiable, and the orders of all integration and differentiation are interchangeable}. 
The analysis first leverages the value function as a Lyapunov function to establish the monotonicity of the flow and then utilizes the state-integrated entropy as a Lyapunov function to demonstrate the exponential convergence.
Indeed,
by the chain rule,  for all $t>0$,
\begin{equation}\label{eq:FR_flow_decreases_intro}
\frac{d}{dt}V^{\pi_t}_{\tau}(\rho)= \left\langle\frac{\delta V^{\pi_t}_{\tau}}{\delta \pi}, \partial_t\pi_t\right\rangle_{d_{\rho}^{\pi_t}}=- \left\langle \left|\frac{\delta V^{\pi_t}_{\tau}}{\delta \pi}\right|^2, \pi_t\right\rangle_{d_{\rho}^{\pi_t}}\le 0\,,
\end{equation}
which shows that   $t\mapsto V^{\pi_t}_{\tau}(\rho)$ is decreasing.
To address the non-convexity of $\pi\mapsto V^{\pi}_{\tau}(\rho)$ (see, e.g., \cite{mei2020global}), 
we prove the following performance difference lemma, extending the result for the tabular case in ~\cite{kakade2002approximately}:
for any two policies $\pi$ and $\pi'$ (Lemma~\ref{lem:performance_diff}),
\begin{equation}\label{eq:per_diff_intro}
V^{\pi}_\tau(\rho)-V^{\pi'}_\tau(\rho)= \bigg\langle \frac{\delta V_{\tau}^{\pi'}}{\delta \pi}(\rho)\bigg |_{d^{\pi}_{\rho}},\pi-\pi'\bigg\rangle_{d^{\pi}_{\rho}}  +    \frac{\tau}{1-\gamma} \int_{S}\operatorname{KL}(\pi(\cdot | s)|\pi'(\cdot | s))d^{\pi}_\rho(ds)\,.
\end{equation}  
Note that \eqref{eq:per_diff_intro} integrates the state variable  with respect to $d^{\pi}_{\rho}$ rather than $d^{\pi'}_{\rho}$, and hence   \eqref{eq:per_diff_intro} does not imply  $\pi\mapsto V^{\pi}_\tau(\rho)$ is convex, except in the special case with $S=\emptyset$, i.e., the bandit setting.
However, it does reveal that the $\tau$-dependent  KL-term serves as a type of strong convex regularisation, allowing for obtaining the exponential convergence of \eqref{eq:gradient_flow_introduction}.
Indeed, differentiating  the map $t\mapsto y_t\coloneqq \int_S \operatorname{KL}(\pi^*_\tau (\cdot|s)|\pi_t(\cdot|s))d^{\pi^*_\tau}_{\rho}(ds)$ yields
\begin{align}\label{eq:dKL_sketch}
\begin{split}
\dot{y}_t&=
\partial_t \left(\int_S \int_A\left(\ln \frac{\mathrm{d} \pi^*_\tau}{\mathrm{d} \mu }(a|s)-\ln \frac{\mathrm{d} \pi_t}{\mathrm{d} \mu }(a|s)\right)\pi^*_\tau(da|s) d^{\pi^*_\tau}_{\rho}(ds)\!\!\right)\\
&= -\int_S \int_A\frac{1}{\frac{\mathrm{d}  \pi_t }{{\mathrm{d} \mu}}(a|s)} 
\frac{\mathrm{d} \partial_t \pi_t }{\mathrm{d} \mu}(a|s)
\pi^*_\tau(da|s) d^{\pi^*_\tau}_{\rho}(ds) 
= (1-\gamma)\bigg\langle \frac{\delta V^{\pi_t}_\tau(\rho)}{\delta \pi}\Big|_{d^{\pi^\ast_\tau}_\rho},\pi^\ast_\tau - \pi_t\bigg\rangle_{d^{\pi^\ast}_\rho}\,, 
\end{split}
\end{align}
where we have used \eqref{eq:grad_flow_density_intro} and  $\left\langle \frac{\delta V^{\pi_t}_{\tau}(\rho)}{\delta \pi}\big|_{d^{\pi^\ast}_\rho}, \pi_t  \right \rangle_{d^{\pi^\ast}_\rho} = 0$.
Applying  \eqref{eq:per_diff_intro} shows that  for all $t>0$,
\begin{equation}\label{eq:dKL_dt}
\dot{y}_t=  -\tau y_t-(1-\gamma)(V^{\pi_t}_{\tau}(\rho)-V^{\pi^*_\tau}_{\tau}(\rho)) \,, 
\end{equation}
or equivalently 
\[
e^{\tau t}y_t =  y_0 -(1-\gamma) \int_0^t e^{\tau r}(V^{\pi_r}_{\tau}(\rho)-V^{\pi^*_\tau}_{\tau}(\rho))\, dr\,.
\]
Since $t\mapsto V^{\pi_t}_{\tau}(\rho)$ is decreasing (see~\eqref{eq:FR_flow_decreases_intro}) and $y_t\ge 0$, for all $t\geq 0$ we have 
\[
V^{\pi_t}_{\tau}(\rho)-V^{\pi^*_\tau}_{\tau}(\rho)\le  \frac{\tau }{(1-\gamma)(e^{ \tau t}-1)} \int_S \operatorname{KL}(\pi^*_\tau(\cdot|s)|\pi_0(\cdot|s) )d^{\pi^*_\tau}_{\rho}(ds)\,,
\]
which is the desired exponential convergence
of \eqref{eq:gradient_flow_introduction}.

However, making the above heuristic arguments rigorous involves several technical challenges, mainly due to the discontinuity of the KL-divergence $\clP(A)\ni \nu\mapsto  \operatorname{KL}(\nu|\mu)\in \mathbb{R}\cup\{\infty\}$ in \eqref{eq:value_rho_introduction}.   The well-posedness of \eqref{eq:gradient_flow_introduction} cannot be deduced from standard results for ordinary differential equations, since   \eqref{eq:gradient_flow_introduction} involves a discontinuous nonlinearity, and the solution lies in the incomplete 
subset $\Pi_{\mu}$ of $\clP(A|S)$. Moreover, the differentiability of $t\mapsto \int_S \operatorname{KL}(\pi^*_\tau (\cdot|s)|\pi_t(\cdot|s))d^{\pi^*_\tau}_{\rho}(ds)$ and the interchangeability of all necessary operations in \eqref{eq:FR_flow_decreases_intro} and \eqref{eq:dKL_sketch} remain unclear.

To address the above difficulties, we leverage the duality between \eqref{eq:gradient_flow_introduction} and \eqref{eq:mirror_descent_intro}, rather than focusing solely on the primal flow \eqref{eq:gradient_flow_introduction}. The rigorous analysis proceeds by first studying the dual variable $Z_t$ satisfying \eqref{eq:mirror_descent_intro}, and then reconstructing a policy $\pi_t= \boldsymbol{\pi}(Z_t)$ for the primal flow via the mirror map $\boldsymbol{\pi}$ in \eqref{eq:bold_pi_intro}. The main advantage of working with the dual variable $Z_t\in B_b(S\times A)$ is that directional derivatives and chain rules can be made rigorous without introducing variations directly with respect to measures. Policies are naturally constrained to be in $\Pi_{\mu}$, and we only vary in the dual space $B_b(S\times A)$, which is a Banach space.  

\subsubsection*{Natural policy gradient with log-linear policies.}

We highlight that the stability of \eqref{eq:gradient_flow_introduction}, as proven in Theorem \ref{thm:stability_mirror}, provides a systematic framework for analysing policy gradient methods with parameterised policies. Many of these algorithms can be seen as approximate Fisher--Rao flows with inexact gradient evaluations, and their performance can be quantified using Theorem \ref{thm:stability_mirror}. In the sequel, we illustrate the practical application of Theorem \ref{thm:stability_mirror} for log-linear parameterised policies.

Let $(\bbH, \|\cdot\|_{\bbH}) $ be a Hilbert space with the inner product $\langle\cdot, \cdot\rangle_{\bbH}$ (e.g., $\bbH=\mathbb{R}^N$ or $\bbH=\ell^2$) representing the parameter space, and let $g\in  B_b(S\times A; \bbH)$ be a fixed feature basis. Consider the following minimisation problem:
\begin{equation*}
\min_{\theta\in \bbH}
V^{\pi_\theta}_{\tau}(\rho)\,,
\end{equation*}
where we optimise \eqref{eq:value_introduction} over all log-linear parametrised policies $\pi_\theta$ given by: 
\begin{equation}\label{eq:linear_policy}
\pi_{\theta} =\boldsymbol{\pi}(\langle \theta, g(\cdot)\rangle_{\mathbb H})= \boldsymbol{\pi}(\langle \theta, g_{\pi_\theta}(\cdot)\rangle_{\mathbb H})\,,\,\,\text{where}\,\,\, g_{\pi_\theta}(s,a) := g(s,a) - \int_A g(s,a')\pi_\theta(da'|s)\,.
\end{equation}
By Proposition \ref{prop:differentiability_dV_df} and the chain rule established in Section \ref{sec:derivatives}: 
\begin{equation}\label{eq:V_gradient_theta}
\nabla_{\theta} V^{\pi_{\theta}}_\tau(\rho)= \frac{1}{1-\gamma}\int_{S}\int_{A} 
\left(Q^{\pi_{\theta}}_{\tau}(s,a)+\tau \ln \frac{\textrm{d}\pi_{\theta}}{\textrm{d} \mu}(a|s)\right)
g_{\pi_{\theta}}(s,a)\pi_{\theta}(da|s)d_{\rho}^{\pi_{\theta}}(ds)\,.
\end{equation}

We start by deriving a natural gradient flow with respect to  $\theta$. Let $\lambda>0$, and consider the following mirror descent  update:
\begin{equation}\label{eq:md_intro_theta_KL}
\theta^{n+1} \in \underset{\theta \in \bbH}{\operatorname{arg\,min}} \left(\langle \theta, \nabla_{\theta}V^{\pi_{\theta^n}}_\tau(\rho)\rangle_{\bbH} + \frac{\lambda}{1-\gamma}\int_{S}\operatorname{KL}(\pi_{\theta}(\cdot|s)|\pi_{\theta^n}(\cdot|s))d_{\rho}^{\pi_{\theta^n}}(ds)\right)\,,
\end{equation}
which is an analogue of \eqref{eq:mirror_des_direct_intro}
in    the   log-linear policy class.
Taylor expanding the $\operatorname{KL}$-term at $\theta=\theta^n$  yields 
\begin{align*}
\frac{1}{1-\gamma}\int_{S}\operatorname{KL}(\pi_{\theta}(\cdot|s)|\pi_{\theta^n}(\cdot|s))d_{\rho}^{\pi_{\theta^n}}(ds) 
&= \frac{1}{2}\langle \theta-\theta^n, \mathscr{F}(\theta^n) (\theta-\theta^n)\rangle_{\bbH} + O(\|\theta-\theta^n\|^3_{\bbH})\,,
\end{align*}
where $\mathscr{F}(\theta) \in \clL(\bbH) $ is the (positive semidefinite) Fisher information operator defined by  
\[
\mathscr{F}(\theta) \coloneqq  \int_{S}\int_{A}  \big(g_{\pi_{\theta}}(s,a) \otimes g_{\pi_{\theta}}(s,a) \big)\pi_{\theta}(da|s)d_{\rho}^{\pi_{\theta}}(ds)\,.
\]
For large $\lambda$, the iterate  $\theta^{n+1}$  will  remain close to $\theta^n$, and one may approximate \eqref{eq:md_intro_theta_KL} by a  more tractable update given by
\[
\theta^{n+1} =  \underset{\theta \in \bbH}{\operatorname{arg\,min}} \left(\langle \theta, \nabla_{\theta}V^{\pi_{\theta^n}}_\tau(\rho)\rangle_{\bbH} + \frac{\lambda}{2}\langle \theta-\theta^n, \mathscr{F}(\theta^n)(\theta-\theta^n)\rangle_{\bbH}\right)\,,
\]
which admits an explicit expression  
\[
\theta^{n+1} = \theta^n - \lambda^{-1}\mathscr{F}(\theta^n)^{-1}\nabla_{\theta}V^{\pi_{\theta^n}}_\tau(\rho)\,
\]
provided that $\mathscr{F}(\theta^n)$ is invertible. 
Letting  $\lambda \to \infty$ leads to the following flow: 
\begin{equation}\label{eq:NPG_flow_intro}
\partial_t {\theta}_t = -\mathscr{F}(\theta_t)^{-1}\nabla_{\theta}V^{\pi_{\theta_t}}_\tau(\rho)\,, \quad t>0\,.
\end{equation}
This flow normalises the direction of the vanilla gradient using the  inverse of the Fisher information operator $\mathscr{F}(\theta)$,
which represents the steepest descent  on the policy manifold $\{\pi_\theta\mid \theta\in \bbH\}$ endowed with the Fisher information metric, known as the natural policy gradient \cite{kakade2001natural, bhatnagar2009natural}. 
It can also be viewed as a parametric representation of the mirror descent flow \eqref{eq:mirror_descent_intro} (cf.~\eqref{eq:natural_gradient_mirror_Z} in Appendix \ref{sec:natural_gradient_flows}).

To avoid inverting $\mathscr{F}(\theta)$ in \eqref{eq:NPG_flow_intro}, we approximate the gradient $\nabla_{\theta}V^{\pi_{\theta_t}}_\tau(\rho)$ by the feature $g_{\pi_{\theta_t}}$, as in the actor-critic algorithm (e.g.~\cite{konda1999actor, haarnoja2018soft}). Here, we adopt the compatible function approximation \cite{sutton1999policy}, which employs the centered features $g_{\pi_{\theta}}$ to approximate a centered $Q$-function. In particular, for any $\theta\in \bbH$, consider the quadratic loss $\ell^{\pi_{\theta}}: \bbH \rightarrow \bbR$ defined by
\begin{equation}
\label{eq:loss_Q}
\begin{aligned}
\ell^{\pi_{\theta}}(w)&=\frac{1}{2} \int_{S}\int_A |
A^{\pi_{\theta}}_{\tau}(s,a)- \langle w,  g_{\pi_{\theta}}(s,a)\rangle_{\bbH}|^2 \pi_{\theta}(da|s)d_{\rho}^{\pi_{\theta}}(ds)\,, \\ 
\textnormal{where}&\quad  
A^{\pi_{\theta}}_{\tau}(s,a)\coloneqq Q^{\pi_{\theta}}_{\tau}(s,a)-\int_{A} Q^{\pi_{\theta}}_{\tau}(s,a')\pi_{\theta}(da'|s)\,.\footnotemark
\end{aligned}
\end{equation}
\footnotetext{By subtracting the baseline $\int_{A} Q^{\pi_{\theta}}_{\tau}(s,a')\pi_{\theta}(da'|s)$, $A^{\pi_{\theta}}_{\tau}$ is unbiased in the sense that \linebreak $\int_{A} A^{\pi_{\theta}}_{\tau}(s,a') \pi_{\theta}(da'|s)=0$.
This can serve as a variance reduction if $A^{\pi_{\theta}}_{\tau}$ needs to be estimated through samples \cite{sutton1999policy, bhatnagar2009natural,schulman2015high, cayci2021linear, zhang2020global}, for example, when the cost $c$ or the kernel $P$ of the MDP is not assumed to be known.}
If  $w^*(\theta)\in \operatorname{arg\,min}_{w\in \bbH}\ell^{\pi_{\theta}}(w)$, then   the first-order condition of \eqref{eq:loss_Q} produces   
\begin{equation}\label{eq:first_order_loss}
\mathscr{F}(\theta) w^* (\theta)=\int_{S}\int_{A} Q^{\pi_{\theta}}_{\tau}(s,a)g_{\pi_{\theta}}(s,a)\pi_{\theta}(da|s)d_{\rho}^{\pi_{\theta}}(ds) \,.   
\end{equation}
Substituting \eqref{eq:first_order_loss} into \eqref{eq:V_gradient_theta} and using 
$\pi_{\theta} = \langle\theta, g(\cdot)\rangle_{\bbH}$, we obtain
\[
\nabla_{\theta} V^{\pi_{\theta}}(\rho)=\frac{1}{1-\gamma}\mathscr{F}(\theta)\left(w^*(\theta) + \tau \theta  \right)\,,
\]
based on which,  the natural gradient flow \eqref{eq:NPG_flow_intro} can be equivalently written as:
\begin{equation}\label{eq:NPG}
\partial_t {\theta}_t = -(w^*(\theta_t)+ \tau \theta_t)\,,\quad t>0\,.
\end{equation}

To analyse \eqref{eq:NPG}, observe that the flow is, in fact,  an approximation of the Fisher--Rao flow \eqref{eq:gradient_flow_introduction}.
Indeed, let  $(\theta_t)_{t\ge 0}$ satisfy \eqref{eq:NPG}, the chain rule shows that the policies $(\pi_{\theta_t})_{t\ge 0}$  satisfies  
\begin{align*}
\begin{split}
\partial_t \pi_t (da|s)
&  =  - \left(  Q_t(s,a)  + \tau \ln\frac{\mathrm{d}  \pi_t}{\mathrm{d}  \mu}(s,a) -\int_A \left( Q_t(s,a')  + \tau \ln\frac{\mathrm{d}  \pi_t}{\mathrm{d}  \mu}(a'|s)\right)
\pi_t(da'|s)\right)\pi_t(da|s )
\\
& =
- \left( \frac{\delta V^{\pi_{t}}_{\tau}}{\delta \pi}(s,a)+ \mathcal{E}_t(s,a)   -\int_A \mathcal{E}_t(s,a') 
\pi_t(da'|s)\right)\pi_t(da|s )
\,,
\end{split}
\end{align*}
where 
\[
Q_t(s,a)=\langle w^*(\theta_t),  g(s,a)\rangle_{\bbH},
\quad 
\mathcal{E}_t(s,a)
=\langle w^*(\theta_t),  g(s,a)\rangle_{\bbH} -A^{\pi_{\theta_t}}_\tau (s,a)\,.
\]
This is  an approximate Fisher--Rao flow  with a perturbed gradient (cf.~\eqref{eq:gradient_flow_introduction}), which can be analysed based on 
Theorem \ref{thm:stability_mirror}. In particular, if  $w^*(\theta_t)$ achieves zero in \eqref{eq:loss_Q}, then the approximation error
$\mathcal{E}_t=0$ and \eqref{eq:NPG} is a parametric representation of the Fisher--Rao flow \eqref{eq:gradient_flow_introduction}. 

As the minimiser $w^*(\theta)$ of  \eqref{eq:loss_Q} may not exist or may not be sufficiently regular to guarantee the well-posedness of \eqref{eq:NPG}, in Section \ref{sec:npg_log_linear}, we analyse a slightly modified version of \eqref{eq:NPG} where we regularise the quadratic loss \eqref{eq:loss_Q} to ensure the well-posedness of the natural policy flow. Based on Theorem \ref{thm:stability_mirror},
we prove that the value functions of $(\pi_{\theta_t})_{t\ge 0}$ along the modified flow converge exponentially, up to a term quantifying the accumulated approximation error $(\mathcal{E}_t)_{t\ge 0}$. 

\subsection{Most related works}\label{sec:intro_lit_rev}

There is an enormous amount of research literature on RL and we cannot hope to do it justice here. For this reason, we focus on the subset of RL that we feel is most related to our  work.

\subsubsection*{MDPs with discrete state and action spaces}
Extensive research has been conducted on policy gradient algorithms in the discrete setting. For entropy-regularised MDPs, \cite{cen2022fast} analysed discrete-time versions of \eqref{eq:gradient_flow_introduction} and \eqref{eq:mirror_descent_intro}, demonstrating linear convergence for both value functions and policies. In the same setting,  \cite{cayci2021linear} studied the linear convergence of natural policy gradient with log-linear policies in the entropy-regularised setting, which is a discrete-time analog of~\eqref{eq:NPG}.
Building on these works, \cite{xiao2022convergence} and   \cite{KHODADADIAN2022105214} achieved linear convergence for unregularised MDPs with inexact policy evaluation by employing geometrically increasing step sizes.
Recently, \cite{lan2023policy} and  \cite{zhan2023policy} established the linear convergence of policy mirror descent for MDPs with arbitrary convex regularisers, with the latter achieving convergence rates independent of action space dimensions.
Furthermore, \cite{yuan2022linear} analysed log-linear parameterised policies with compatible $Q$-function approximations, providing estimates based on $L^2$-approximation errors instead of the $L^\infty$ estimates found in \cite{xiao2022convergence}.
These works typically require  additional assumptions on the MDP--specifically regarding the initial condition, and the structures of transition kernels and cost functions--to ensure persistent exploration for the Q-function approximation.

Note that these works primarily focus on discrete-time algorithms, and their analyses may not be directly applicable to   study   the continuous-time flows presented in \eqref{eq:gradient_flow_introduction} and \eqref{eq:mirror_descent_intro}. Moreover, most of the results (except those in \cite{zhan2023policy}) depend explicitly on the action space cardinality, and cannot be extended to continuous action spaces. 

Instead of optimizing over policies, one may choose to optimize over state-action occupation measures,
by leveraging the convexity of value function with respect to occupation measures. 
This approach has been employed in \cite{neu2017unified} for average cost problems and in \cite{zhang2020variational}
for discounted cost problems.
Compared to the discounted criterion, analyzing policy gradient methods for the average cost criterion is technically more challenging, as it requires maintaining suitable recurrence/ergodicity throughout the gradient methods.

\subsubsection*{MDPs with continuous state and action space}

Although MDPs with continuous state and action spaces are widely used in practical applications \cite{doya2000reinforcement, van2012reinforcement, manna2022learning}, the convergence analysis of policy gradient methods in this setting remains   less developed compared to its discrete counterparts.

Most existing works have focused on discrete-time linear quadratic regulator (LQR) problems with linear parameterised policies. By leveraging the linear-quadratic structure,  one can show that the objective function satisfies the Polyak--{\L}ojasiewicz (PL)  (also known as gradient dominance) condition \cite{polyak1963gradient, lojasiewicz1963topological, kurdyka1998gradients}. This condition implies that all stationary points are globally optimal, and based on this property, the policy gradient descent algorithm has been shown to converge linearly \cite{fazel2018global, bu2019lqr, hu2023toward}. These results have been extended to continuous-time systems \cite{sontag2022remarks,giegrich2024convergence} or to scenarios involving small nonlinear perturbations of LQR systems \cite{han2023policy}.

The insights from the LQR problem have been extended to encompass general MDPs with finite-dimensional parameterised policies  \cite{bhandari2019global,bhatt2019policy, bedi2021sample, bedi2022hidden, zhang2022convergence, fatkhullin2023stochastic}. These extensions typically assume that the objective function uniformly satisfies a PL condition across all feasible parameter choices, and the action space is a subset of a finite-dimensional Euclidean space. Consequently, these results are not directly applicable to our scenario with an infinite-dimensional action space. Furthermore, verifying the uniform PL condition in our context remains an open challenge. It is worth noting that even in the discrete setting, \cite{mei2021leveraging} has shown that the objective only satisfies a non-uniform PL condition depending on the current iterate.

For MDPs with general convex regularisers,
\cite{lan2022policy}  analyses policy mirror descent (i.e., discrete-time versions of \eqref{eq:gradient_flow_introduction} and \eqref{eq:mirror_descent_intro}), and achieves linear convergence  rates and stability results. However, their results are limited to action spaces in Euclidean spaces and do not extend to stochastic policies with continuous action spaces.

For a general MDP with policies being normal distributions with a fixed covariance matrix and mean given by a function from a reproducing kernel Hilbert space (RKHS)~\cite{paternain2020stochastic} show convergence of policy gradient to the optimal representative from the RKHS. 
For MDPs with mean-field softmax policies, 
which are infinite-dimensional policies parameterised by probability measures, \cite{agazzi2020global}  derives a Wasserstein gradient flow for the parameter measure. They prove that if this gradient flow converges to a stationary point with full support, the resulting softmax policy is optimal. However, they do not provide conditions under which the gradient flow converges.
The same setting is revisited in \cite{leahy2022convergence}, where they show that the value function decreases  along the Wasserstein gradient flow, and introducing strong entropic regularisation at the level of parameter measure   leads to exponential convergence.
Finally, \cite{zhang2021wasserstein} analyses a two-timescale actor-critic algorithm for unregularised MDPs. The policy is updated on a slower timescale using the Fisher--Rao flow, while the $Q$-function is updated via the Wasserstein gradient flow on a faster timescale. They establish an error bound of $O(1/T)$; however, they assume the differentiability of the function $t\mapsto  \int_S \operatorname{KL}(\pi^* (\cdot|s)|\pi_t(\cdot|s))d^{\pi^*}_{\rho}(ds)$ along the flow without providing a formal proof.
In this paper, we rigorously prove this crucial differentiability result in the setting of regularised MDPs. It is important to highlight that due to the presence of entropy regularisation, the analysis of the flow becomes more technically involved compared to the unregularised problem.

\subsubsection*{Optimisation over measure spaces}

In \cite{chizat2021convergence, aubin2022mirror}, sublinear convergence rates 
of mirror descent methods
are established for   minimising  (relatively) smooth and convex functions   over   spaces of measures.  A linear convergence rate is   proved 
in \cite{liu2023polyak}
under an additional Polyak--{\L}ojasiewicz condition for the continuous-time   Fisher--Rao flow.

However, it is important to highlight that the optimisation problem \eqref{eq:value_rho_introduction} presents unique challenges compared to the scenarios in \cite{chizat2021convergence,aubin2022mirror,liu2023polyak}. Here, we aim to optimise a non-convex objective over all 
probability transition kernels representing 
stochastic policies.
A  policy may induce a state  distribution that is a different from the optimal one (see e.g., 
the adaptive Bregman divergence in \eqref{eq:mirror_des_direct_intro}). Controlling this distribution shift throughout the flow necessitates a more intricate    analysis. 

\subsection{Outline of the paper}

In Section~\ref{sec:notation}, we state the main notation used throughout the paper. In Section~\ref{sec:main_result}, we introduce entropy-regularised MDPs on general state and action spaces, state the main results (i.e., Theorems~\ref{thm:linear_convergence},~\ref{thm:stability_mirror} and~\ref{thm:NPG_stability}), and discuss the unregularised setting. We prove some key properties of entropy-regularised MDPs in Section \ref{sec:prop_of_MDPs}. Specifically, in Section \ref{sec:basis_properties_MDP}, we establish bounds on and regularity of the key functions (i.e., the value and $Q$-functions and the log-density of policies). In Section \ref{sec:derivatives}, we define our notion of differentiability, prove a chain rule, and then prove the differentiability of key functions. We prove our main results in Section \ref{sec:proof_main_results}.
In Appendix~\ref{sec:Bellman}, we include a proof of the Bellman principle for entropy-regularised MDPs on general state and action spaces. 
In Appendix~\ref{sec:natural_gradient_flows}, we explain that all the gradient flows appearing in the paper are {\em natural} gradient flows with respect to suitable Riemannian metric.
The section also conveniently summarises some of the more technical results proved in Section~\ref{sec:derivatives}.

\subsection{Acknowledgments}
We are grateful for our encouraging discussion with colleagues, particularly A. Koppel and T. Drivas. JML is grateful for support from the US AFOSR Grant FA8655-21-1-7034.

\subsection{Notation}
\label{sec:notation}

Let $\mathbb R_+ = [0,\infty)$ and $\bbN_0=\bbN\cup\{0\}$. For given normed vector spaces $(X,\|\cdot\|_X)$ and $(Y,\|\cdot\|_Y)$,
we denote by $\clL(X,Y)$ the normed vector space of bounded linear operators $T: X\to Y$, equipped with the operator norm
$\|T\|_{\clL(X,Y)}=\sup_{\|x\|_X\le 1} \|Tx \|_Y=1$.  For simplicity, we write $\clL(X)=\clL(X,X)$.  For a given Hilbert space $(X,\langle \cdot,\cdot\rangle_X)$ and $x,y\in X$, we define the outer product $x\otimes y\in \clL(X)$ by $(x\otimes y) z=x\langle y,z\rangle_X$ for all $z\in X$.

Let $(E,d)$ denote a Polish space (i.e., a complete separable metric space). We always equip a Polish space with its Borel sigma-field $\mathcal{B}(E)$.  For a given measure $\mu$ in $E$, denote by $L^p(E,\mu)$, $p\in [1,\infty]$, the Lebesgue spaces of integrable functions. Denote by $B_b(E)$ the space of bounded measurable functions $f:E\rightarrow \mathbb{R}$ endowed with the supremum norm $\|f\|_{B_b(E)}=\sup_{x\in E}|f(x)|$. Denote by $\mathcal{M}(E)$ the Banach space of finite signed measures $\mu$ on $E$ endowed with the total variation norm $\|\mu\|_{\mathcal{M}(A)}=|\mu|(E)$, where $|\mu|$ is the total variation measure. 
Recall that if $\mu=f d\rho$, where $\rho \in \mathcal{M}_+(E)$ is a nonnegative measure and $f\in L^1(E,\rho)$, then $\|\mu\|_{\mathcal{M}(E)}=\|f\|_{L^1(E,\rho)}$.  Denote by $\mathcal{P}(E)\subset\mathcal{M}(E)$ the set of probability measures on $E$. For given $\mu,\mu'\in \mathcal{P}(E)$, we write   $\mu\ll \mu'$ if $\mu$ is absolutely continuous with respect to $\mu'$, and define the Kullback-Leibler (KL) divergence  of $\mu$ with respect to $\mu'$ (or  relative entropy of $\mu$ relative to $\mu'$)  by 
$\textnormal{KL}(\mu|\mu') = \int_{E}\ln\frac{\mathrm{d} \mu}{\mathrm{d}\mu'}(x)\mu(dx) $ if $\mu\ll \mu'$, and $\infty$ otherwise. 

Given Polish spaces $(E_1,d_1)$ and $(E_2,d_2)$, we introduce notation for measurable functions $k:E_1\rightarrow \mathcal{M}(E_2)$. 
Denote by $b\mathcal{M}(E_1|E_2)$ the Banach space of bounded signed kernels $k: E_2\rightarrow \mathcal{M}(E_1)$ endowed with the norm $\|k\|_{b\mathcal{M}(E_1|E_2)}=\sup_{x\in E_2}\|k(x)\|_{\mathcal{M}(E_1)}$; that is, $k(U|\cdot): E_2\rightarrow \mathbb{R}$ is measurable for all $U\in \mathcal{M}(E_1)$ and $k(\cdot|x)\in \mathcal{M}(E_1)$ for all $x\in E_2$ \cite{gerlach2015lattice}. For given $k_1,k_2\in b\clM(E_1|E_2)$, we write $k_1\ll k_2$ if there exists a measurable function $\eta :E_2\times E_1\to \mathbb{R}_+$ such that for all $x\in E_2$ and $B\in \mathcal{B}(E_1)$, $k_1(B|x)= \int_B \eta(x,y)k_2(dy|x)$. For a fixed positive measure $\mu\in \clM(E_1)$ and $k\in  b\clM(E_1|E_2)$, we write $\mu\ll k$ (resp.~$ k\ll \mu$) if the associated kernel $k_\mu \in b\clM(E_1|E_2)$ given by $k_\mu(dy|x)=\mu(dy)$ for all $x\in E_2$ satisfies   $ k_\mu\ll k$ (resp.~$  k\ll k_\mu$). We denote by $b\mathcal{M}_{\mu}(E_1|E_2)$ the space of kernels $k\in  b\clM(E_1|E_2)$ such that $k\ll \mu$. We denote by $\mathcal{P}(E_1|E_2)$ (resp.~$\mathcal{P}_\mu(E_1|E_2)$) the  collection  of $P\in b\mathcal{M}(E_1|E_2)$ (resp.~$P\in b\mathcal{M}_\mu(E_1|E_2)$) such that $P(\cdot|x)\in \mathcal{P}(E_1)$ for all $x\in E_2$. For given $\mu\in \clP(E_2)$ and $k\in \clP(E_1|E_2)$,  we define the semidirect product $\mu\otimes k\in \clP(E_2\times E_1) $ of $\mu$ with $k$ by $ (\mu\otimes k)(A\times B)=\int_A k(B|x)\mu(dx)$, for all $A\in \clB(E_2),B\in \clB(E_1) $.

\section{Problem formulation and statements of main results} 
\label{sec:main_result}

\subsection{Entropy-regularised MDPs}
\label{sec:main_results_entropy_reg_MDPs}

In this section, we formulate the entropy-regularised MDPs with continuous state and action spaces. 
Let $S$ and $A$ be Polish spaces,\footnote{
The condition that 
$S$ and $A$  are Polish spaces is essential for applying the Kolmogorov extension theorem, which 
allows for the construction of  the unique probability measure  $\mathbb P^\pi_\rho$ satisfying  \eqref{eq:markov_like} for each  policy  $\pi$;  see    \cite[Proposition 7.28]{bertsekas2004stochastic} and \cite[Section 2.2]{hernandez2012discrete}.
} $P\in \clP(S|S\times A)$, $c\in B_b(S\times A)$  and   $\gamma\in [0,1)$. 
The five-tuple  $(S,A,P,c,\gamma)$ determines an infinite horizon Markov decision model, where $S$ and $A$ represent the state and action spaces, respectively, $P$ represents the transition probability,  $c$ represents the cost function and $\gamma$ represents the discount factor.  
Let  $\Pi=\{\pi=\{\pi_n\}_{n\in \bbN_0}: \pi_n\in \clP(A|H_n)\}$ denote the set of (possibly non-Markovian) stochastic policies, where for each $n\in \bbN_0$, $H_n\coloneqq (S\times A)^{n}\times S$ is the space of admissible histories.

Let $(\Omega:=(S\times A)^{\bbN_0},\mathcal{F})$ denote the canonical sample space, where $\clF=\clB(\Omega)$ is the corresponding Borel sigma-algebra. Elements of $\Omega$ are of the form $(s_0,a_0,s_1,a_1,\ldots)$ with $s_n\in S$ and $a_n\in A$ denoting the projections and called the state and action variables, at time $n\in \mathbb{N}_0$, respectively. By \cite[Proposition 7.28]{bertsekas2004stochastic}, for any given initial distribution $\rho\in\clP(S)$ and policy $\pi\in \Pi$, there exists a unique product probability measure $\bbP^{\pi}_{\rho}$ on $(\Omega,\mathcal{F})$ with expectation denoted $\mathbb{E}^{\pi}_{\rho}$ such that for all $n\in \bbN_0$, $B\in \clB(S)$ and $C\in  \clB(A)$, $\bbP_{\rho}^{\pi}(s_0\in B)=\rho(B)$ and 
\begin{equation} \label{eq:markov_like} 
\bbP_{\rho}^{\pi}(a_n\in C|h_n)=\pi_n(C|h_n),
\quad 
\bbP_{\rho}^{\pi}(s_{n+1}\in B|h_n, a_n)=P(B|s_n,a_n)\,, 
\end{equation}
where $h_n=(s_0,a_0,\ldots, s_{n-1}, a_{n-1}, s_{n})\in H_n$. In particular, if $\pi$ is a  Markov stochastic policy (i.e., $\pi_n\in \clP(A|S)$ for all $n\in \bbN_0$), then $\{s_n\}_{n\in \bbN_0}$ is a Markov process with kernel $\{P_{\pi,n}\}_{n\in \bbN_0}\in \clP(S|S)$ given by
\[
P_{\pi,n}(ds'|s)=\int_{A}P(ds'|s,a)\pi_n(da|s),
\quad \forall s\in S, n\in \bbN_0\,.
\]
For $s\in S$, we denote $\mathbb{E}^{\pi}_{s}=\mathbb{E}^{\pi}_{\delta_s}$, where $\delta_s\in \mathcal{P}(S)$ denotes the Dirac measure at $s\in S$.

Let $\mu \in \mathcal \clP(A)$ denote a reference measure and $\tau\in (0,\infty)$ denote a regularisation parameter. For each $\pi=\{\pi_n\}_{n\in \bbN_0} \in \Pi$ and $s\in S$, define the following regularised value function:
\begin{equation}
\label{eq:V_pi_tau}
V^{\pi}_{\tau}(s)= \bbE_{s}^{\pi}\left[\sum_{n=0}^\infty\gamma^n \Big(c(s_n,a_n) + \tau \operatorname{KL}(\pi_n(\cdot|h_n)|\mu)\Big)\right]\in \bbR\cup \{\infty\}\,,
\end{equation}
which may be infinite if  $\pi_n\not \in \clP_\mu(A|S)$ for some $n\in \mathbb{N}_0$, or if $\bbE_{s}^{\pi}\left[ \sum_{n=0}^\infty\gamma^n   \operatorname{KL}(\pi_n(\cdot|h_n)|\mu)\right]$ diverges. Since $c$ is bounded and $H_n\ni h_n\mapsto \operatorname{KL}(\pi_n(\cdot|h_n)|\mu)\in [0,\infty]$ is non-negative and measurable, $V^{\pi}_\tau: S\rightarrow \bbR\cup \{\infty\}$ is a well-defined measurable function. We define the optimal value function $V^*_{\tau}: S\rightarrow \bbR\cup \{\infty\}$ by
\begin{equation}
\label{eq:optimal_value}
V^*_{\tau}(s)=\inf_{\pi \in \Pi}V^{\pi}_{\tau}(s),
\quad \forall s\in S\,,
\end{equation}
and refer to  $\pi^* \in \Pi$ as an optimal policy if  $V^{\pi^*}_{\tau}(s)=V^*_{\tau}(s)$, for all  $s\in S$. 

One can prove that $V^*_{\tau}$ satisfies a dynamic programming principle (DPP) as stated in  Theorem~\ref{thm:DPP}, which implies that $V^*_{\tau}\in B_b(S)$ and for all $s\in S$, 
\[
V^{\ast}_{\tau}(s)=-\tau\ln\int_{A}\exp\left(-
\frac{1}{\tau}Q^{\ast}_{\tau}(s,a)\right)\mu(da),
\]
where $Q^*_{\tau}\in B_b(S\times A)$ is defined by  
\[
Q^{*}_{\tau}(s,a)=c(s,a)+\gamma\int_S V_{\tau}^{*}(s')P(ds'|s,a)\,,
\quad \forall (s,a)\in S\times A\,.
\]
Moreover, there is an optimal policy $\pi^*_{\tau} \in \clP_{\mu}(A|S)$  given by
\begin{equation}
\label{eq:optimal_policy}
\pi^*_{\tau}(da|s) = \exp\left(-\frac{1}{\tau }(Q^{\ast}_{\tau}(s,a)-V^{\ast}_{\tau}(s))\right)\mu(da)\,,
\quad \forall s\in S.
\end{equation}

\begin{remark}
\label{rmk:KL_role}
The KL divergence in \eqref{eq:V_pi_tau} 
is essential for ensuring the existence of optimal policies and designing a convergent gradient flow for general state and action spaces.
Due to the KL divergence,~\eqref{eq:optimal_value}
satisfies the DPP with merely measurable transition kernels $P$ and cost functions $c$.
This contrasts with unregularized MDPs, which  require additional continuity conditions on $P$ and $c$ to  establish the DPP and construct an optimal policy, see~\cite[Assumption 4.2.1]{hernandez2012discrete}.

The KL divergence also ensures persistent exploration.
By~\eqref{eq:optimal_policy}, any policy equivalent to $\mu$ (specifically those  in Definition \ref{def:pi_mu})  
explores the optimal policy $\pi^*_\tau$ in the sense that the Radon--Nikodym derivative of a policy w.r.t.~the optimal policy has an upper and lower bound. 
This property is  absent in the unregularized problem, where the optimal action often is given by a Dirac distribution, making it challenging to design a gradient flow that maintains  exploration; 
see Example \ref{example:bandit}  and the discussion following it for details. 

The KL divergence further ensures that the value function $\pi\mapsto V^\pi_\tau (\rho)$  admits a strong convexity-like expansion given by the performance difference lemma (Lemma~\ref{lem:performance_diff}), allowing us to establish the exponential convergence of the proposed gradient flow.
It is worth noting that~\cite{neu2017unified} shows for  average-cost MDPs in the tabular setting, the regularized value function  is convex in the state-action occupancy measure. 
Although we expect a similar result   in the present setting, we do not explicitly leverage this property in the design and analysis of our algorithms.
Instead, we directly focus on optimizing over the policy spaces.
\end{remark}

The expression of 
$\pi^*_\tau$ in \eqref{eq:optimal_policy}
suggests that, without loss of generality, it suffices to minimise \eqref{eq:V_pi_tau} over the class of stationary Markov policies that are equivalent to the reference measure $\mu$.

\begin{definition}
\label{def:pi_mu}
Let $\Pi_{\mu}$ denote the class of policies $\pi =\{\pi_n\}_{n\in \mathbb{N}_0} \in  \Pi $ such that $\pi_n\in \clP_\mu(A|S)$ for all $n\in \mathbb{N}_0$, and for which there exists $f\in B_b(S\times A)$ such that $\pi_n (da|s) =\frac{\exp\left(f(s,a)\right)}{ \int_{A}  \exp\left(f(s,a)\right) \mu(da)} \mu(da)$ for all $s\in S$ and $n\in \mathbb{N}_0$. In the sequel, we identify $ \Pi_\mu$ with the set  $\{\boldsymbol{\pi}(f)\mid f\in B_b(S\times A)\}\subset \clP_{\mu}(A|S)$, where $ \boldsymbol{\pi} :B_b(S\times A) \to \clP_{\mu}(A|S)$ is defined by 
\begin{equation}
\label{eq:pi_f_mu}
\boldsymbol{\pi}(f)(da|s)= \frac{e^{f(s,a)}}{\int_A e^{f(s,a')}\mu(d a')}\mu(d a),
\quad \forall f\in B_b(S\times A) \,.    
\end{equation}
\end{definition}

For each $\pi \in \Pi_{\mu} $, we define the $Q$-function $Q^{\pi}_{\tau}\in B_b(S\times A)$ by 
\begin{equation}\label{def:Q_fn}
Q^{\pi}_{\tau}(s,a)=c(s,a)+\gamma\int_S V_{\tau}^{\pi}(s')P(ds'|s,a)\,.
\end{equation}
Then by the Bellman principle (see Lemma~\ref{lem:on_policy}), for all $\pi \in \Pi_{\mu} $ and $s\in S$,
\begin{equation}\label{eq:on_policy}
V^{\pi}_{\tau}(s)=\int_{A}\left(Q_\tau^\pi(s,a)+\tau \ln \frac{\mathrm{d} \pi}{\mathrm{d} \mu}(a|s)\right)\pi(da|s)\,.
\end{equation}
For each $\pi\in \clP(A|S)$, we define the occupancy kernel $d^{\pi}\in\clP(S|S)$ by
\begin{equation}\label{eq:occupancy_s}
d^{\pi}(ds'|s)=(1-\gamma)\sum_{n=0}^{\infty}\gamma^nP^n_{\pi}(ds'|s)\,,
\end{equation}
where $P^n_{\pi}$ is the $n$-times  product of the kernel $P_{\pi}$ with $P^0_{\pi}(ds'|s)\coloneqq \delta_s(ds')$ and the convergence is understood in $b\clM(S|S)$. For a given initial distribution $\rho\in \clP(S)$, we define
\[
V^{\pi}_{\tau}(\rho)=\int_{S}V^{\pi}_{\tau}(s) \rho(ds) \quad \textnormal{and} \quad 
d^{\pi}_{\rho}(ds)=\int_{S}d^{\pi}(ds|s')\rho(ds')\,.
\]

\subsection{Convergence of  the gradient flow }
\label{sec:MD_approach_main}

This section analyses the well-posedness and the convergence of the gradient flow \eqref{eq:gradient_flow_introduction} for minimising  \eqref{eq:V_pi_tau} over $\Pi_\mu$. We recall the gradient flow here for the reader's convenience: 
for given  $\pi_0\in \clP_\mu(A|S)$,   
consider the initial value problem:
\begin{equation}\label{eq:FR_flow}
\left\{
\begin{aligned}
&\partial_t {\pi_t}(da|s) =-\left(Q^{\pi_t}_\tau(s,a)+ \tau \ln\frac{\mathrm{d}  \pi_t}{\mathrm{d}\mu}(a|s) -V^{\pi_t}_\tau(s)\right)\pi_t(da|s)\,,\quad   t>0,
\\
&\pi\vert_{t=0}=\pi_0\,,
\end{aligned}
\right.
\end{equation}
which is an ordinary differential equation on the infinite-dimensional convex space $\mathcal P(A|S)$.

It is well-known that for given $\rho\in \clP(S)$, the map $ \Pi_\mu \ni \pi\mapsto  V^\pi_\tau(\rho)\in \mathbb{R} $ is, in general,  non-convex  (see e.g., \cite{agarwal2020optimality,mei2020global}).
Consequently, the convergence of \eqref{eq:FR_flow} cannot be inferred from standard convergence results for gradient flows in convex optimization problems. 
The following lemma  is crucial for addressing this non-convexity issue, 
whose proof is given in 
Section~\ref{sec:basis_properties_MDP}.

\begin{lemma}[Performance difference]
\label{lem:performance_diff}
For all $\rho \in \clP(S)$ and $\pi,\pi'\in \Pi_{\mu}$, 
\begin{align*}
&V^{\pi}_\tau(\rho)-V^{\pi'}_\tau(\rho) \\
&\quad = \frac{1}{1-\gamma}\int_S \bigg[\int_A\left(Q^{\pi'}_{\tau}(s,a)+\tau \ln \frac{\mathrm{d} \pi'}{\mathrm{d}\mu}(a|s)\right)(\pi-\pi')(da|s) + \tau    \operatorname{KL}(\pi(\cdot | s)|\pi'(\cdot | s)) \bigg]d^{\pi}_\rho(ds)\,.
\end{align*}
\end{lemma}

Lemma~\ref{lem:performance_diff} generalizes the performance difference lemma from \cite[Lemma 25]{mei2020global}, which was established for entropy-regularized MDPs in the tabular setting, to accommodate general state and action spaces.
The performance difference lemma was first introduced for (unregularized) tabular MDPs in \cite{kakade2002approximately} and has since become fundamental in the analysis of various policy gradient algorithms for tabular MDPs
(see e.g., \cite{agarwal2020optimality,xiao2022convergence,  lan2022policy, zhan2023policy}).
In our analysis, it acts as a substitute for strong convexity of value function. 
Indeed, 
recall that  $Q^{\pi'}_{\tau}-V_\tau^{\pi'} + \tau \ln \frac{\mathrm{d} \pi'}{\mathrm{d} \mu}$
is the first variation of $\pi\mapsto V^{\pi}_{\tau}(\rho)$ at $\pi'$ (see \eqref{eq:delta_V_delta_pi-intro}), and hence  Lemma \ref{lem:performance_diff}
allows for treating  the $\tau$-dependent  KL-term  as a  strong convex regularisation (if  the occupancy measure $d^\pi_\rho$ is ignored).

In the sequel, we rigorously prove the well-posedness and    convergence of \eqref{eq:FR_flow}. 
Analyzing   \eqref{eq:FR_flow}
for continuous state and action spaces
is   more complex than for tabular MDPs~\cite{muller2024essentially, muller2024fisher}.  
Specifically, in the present setting,~\eqref{eq:FR_flow} lies in the infinite-dimensional metric space $\mathcal P(A|S)$, which reduces to a finite-dimensional set in the tabular case.
Moreover, justifying the term $\ln\frac{\mathrm{d}\pi_t}{\mathrm{d}\mu} $ in~\eqref{eq:FR_flow} is necessary for general action spaces $A$,  since not all measures in $\mathcal P(A)$  are absolutely continuous with respect to $\mu$.
This issue does not arise for finite action spaces, where all measures are absolutely continuous with respect to the uniform distribution.

An essential tool for analyzing  \eqref{eq:FR_flow} is the following mirror descent dynamics, which is the dual flow of \eqref{eq:FR_flow} in $B_b(S\times A)$.  Specifically, for  given  $Z_0\in B_b(S\times A)$, 
consider the initial value problem:
\begin{equation}\label{eq:mirror_descent}
\left\{
\begin{aligned}
&\partial_t Z_t(s,a) =   - \left( Q^{\pi_t}_\tau(s,a)  + \tau Z_t(s,a)-V^{\pi_t}_\tau(s)\right), \quad  \pi_t (da|s) = \boldsymbol{\pi}(Z_t)(da|s)\,,
\quad t>0,
\\
&Z\vert_{t=0}=Z_0.
\end{aligned}
\right.
\end{equation}
where  $ \boldsymbol{\pi} :B_b(S\times A) \to  \clP_{\mu}(A|S)$ is  defined by \eqref{eq:pi_f_mu}. Equation \eqref{eq:mirror_descent} extends the (finite-dimensional) mirror descent ODE system in \cite{krichene2015accelerated}
to the infinite-dimensional space $\Pi_\mu$. At each $t>0$, it updates the dual variable $Z_t$ in the dual space $B_b(S\times A)$ along the gradient direction, and maps it back to the primal space $\clP_{\mu}(A|S)$ using the map $\boldsymbol{\pi}$. As we will show in Lemma \ref{lemma:dpi_t},  \eqref{eq:mirror_descent} produces the same flow for $(\pi_t)_{t\ge 0}$ as \eqref{eq:FR_flow}. 

Here, we emphasize two significant theoretical advantages of the dynamics \eqref{eq:mirror_descent} over \eqref{eq:FR_flow}: (i) \eqref{eq:mirror_descent} lies within the Banach space $B_b(S\times A)$, while the solution to \eqref{eq:FR_flow} lies in the incomplete metric space $\clP_\mu(A|S)$. (ii) For each $t>0$, the policy $\pi_t$ obtained from \eqref{eq:mirror_descent} naturally belongs to $\Pi_\mu$, whereas it is initially unclear whether the solution to \eqref{eq:FR_flow} is absolutely continuous with respect to $\mu$. These advantageous features of \eqref{eq:mirror_descent} make it a more convenient choice for establishing the well-posedness of the flow $(\pi_t)_{t\ge 0}$ and to ensure the required differentiability along the flow. 

The next lemma shows the equivalence between \eqref{eq:FR_flow} and \eqref{eq:mirror_descent}. The proof is given in Section \ref{sec:convergence_FR_flow}.

\begin{lemma}
\label{lemma:dpi_t}
Let $T>0$. 
\begin{enumerate}[(1)]
\item \label{item:from_MD_to_FR}
Let $Z\in  C^1([0,T); B_b(S\times A))$ be such that~\eqref{eq:mirror_descent} holds and define $\pi_t =\boldsymbol{\pi}(Z_t)$ for all $t\in [0,T) $.
Then $\pi\in  C^1([0,T); \Pi_\mu  )$ satisfies \eqref{eq:FR_flow} with $\pi_0=\boldsymbol{\pi}(Z_0)$.\footnotemark
\item \label{item:from_FR_to_MD}
Let $\pi\in  C^1([0,T); \Pi_\mu )$ be such that~\eqref{eq:FR_flow} holds and define $Z_t =\ln\frac{\mathrm{d}  \pi_t}{\mathrm{d}  \mu} $ for all $t\in[0,T)$. Then $Z\in  C^1([0,T); B_b(S\times A))$ satisfies \eqref{eq:mirror_descent}
with $Z_0=\ln\frac{\mathrm{d}  \pi_0}{\mathrm{d}  \mu} $.
\footnotetext{For any $I\subset  [0,\infty)$, we write  $\pi\in  C^1(I; \Pi_\mu )$ if $\pi\in C^1(I; b\clM(A|S))$ and $\pi_t\in \Pi_\mu$ for all $t\in I$.}
\end{enumerate}
\end{lemma}

\begin{remark}
The update of $(Z_t)_{t\ge 0}$ can be modified without altering the flow $(\pi_t)_{t\ge 0}$. Indeed, by \eqref{eq:dpi_dt},   for any $f\in C(\bbR_+; B_b(S))$, if $Z\in C^1(\mathbb{R}_+; B_b(S\times A))$ satisfies for all $t>0$, 
\begin{equation*} 
\partial_t Z_t(s,a)  =   - \left( Q^{\pi_t}_\tau(s,a) + \tau Z_t(s,a)+f_t(s)\right),
\quad     \pi_t (da|s)  =    \boldsymbol{\pi}(Z_t)(da|s)\,,
\end{equation*}
then $\pi\in  C^1(\bbR_+; \Pi_\mu  )$ satisfies \eqref{eq:FR_flow}. 
This provides additional flexibility in constructing a solution of \eqref{eq:FR_flow} by choosing an appropriate term $f$. In our case, we select $f=-V^{\pi_t}_\tau$ in \eqref{eq:mirror_descent}  as it facilitates deducing the uniqueness of \eqref{eq:FR_flow} from Lemma \ref{lemma:dpi_t} Item \ref{item:from_FR_to_MD}.
\end{remark}

The maps $Z\mapsto Q^{\boldsymbol{\pi}(Z)}_\tau$ and $Z\mapsto V^{\boldsymbol{\pi}(Z)}_\tau$ are merely locally Lipschitz continuous, due to the presence of the KL divergence in  $Q^{\pi}_\tau$ and $V^{\pi}_\tau$ (see Lemma \ref{lemma:Q_local_lipschitz}). Consequently,  \eqref{eq:mirror_descent}  involves a locally Lipschitz non-linearity, which prevents us from directly deducing the existence of a global solution from the Picard--Lindel{\"o}f theorem. To overcome this difficulty, we first establish that the value function decreases along the flow $(Z_t)_{t\ge 0}$, and hence is uniformly bounded in $t$.

\begin{proposition}
\label{prop:value_monotone}
Let $Z_0\in B_b(S\times A)$ and $T>0$. If $Z\in  C^1([0,T); B_b(S\times A))$
satisfies \eqref{eq:mirror_descent} and $\pi_t =\boldsymbol{\pi}(Z_t)$ for all $t\in [0,T)$, then $t\mapsto V^{\pi_t}_\tau $ is differentiable, and $\partial_t V^{\pi_t}_\tau (s)\le 0$ for all $s\in S$ and $t\in [0,T)$. 
\end{proposition}

The proof of Proposition \ref{prop:value_monotone} is given in Section \ref{sec:convergence_FR_flow}.
To accommodate general state and action spaces, 
a key technical step is to establish the Hadamard differentiability of 
$B_b(S\times A)\ni Z\mapsto V^{\boldsymbol{\pi}(Z)}_\tau\in \mathbb R$ (Proposition \ref{prop:differentiability_dV_df}).
We adopt Hadamard differentiability as it is the weakest form of differentiability that permits a chain rule.

Using Proposition \ref{prop:value_monotone},
we   prove that a solution of \eqref{eq:mirror_descent} does not explode in any finite time (Proposition \ref{prop:a_priori_bound_Z}), and subsequently establish the well-posedness of \eqref{eq:mirror_descent} (or equivalently \eqref{eq:FR_flow}), as stated in the following theorem.

\begin{theorem}\label{thm:wp_MD}
For each $Z_0\in B_b(S\times A)$, there exists a unique $Z\in  C^1(\mathbb{R}_+; B_b(S\times A))$ satisfying \eqref{eq:mirror_descent} for all $t\in \bbR_+  $. Consequently, for all $\pi_0\in \Pi_\mu$, there exists a unique  $\pi\in  C^1(\mathbb{R}_+; \Pi_\mu)$  satisfying \eqref{eq:FR_flow} for all  $t\in \bbR_+   $.
\end{theorem}

The proof of Theorem \ref{thm:wp_MD} is given in Section \ref{sec:convergence_FR_flow}.

Recall that Proposition \ref{prop:value_monotone} shows that the value function decreases along the flow  \eqref{eq:mirror_descent} (or equivalently \eqref{eq:FR_flow}).   
The following theorem proves an exponential convergence of the value function and the policy. 

\begin{theorem}
\label{thm:linear_convergence}
Let  $Z\in  C^1(\mathbb{R}_+; B_b(S\times A))$ satisfy \eqref{eq:mirror_descent}. Then for all $\rho\in \clP(S)$ and $t>0$,
\begin{equation*}
V^{\pi_{t}}_{\tau}(\rho)-V^{\pi^*_\tau}_{\tau}(\rho)\le \frac{\tau }{(1-\gamma)(e^{ \tau t}-1)} \int_S \operatorname{KL}(\pi^*_\tau(\cdot|s)|\pi_0(\cdot|s) )d^{\pi^*_\tau}_{\rho} (ds)\,,
\end{equation*}
and 
\begin{equation*}
\int_S\|\pi_t(\cdot|s) -\pi^*_\tau(\cdot|s)\|^2_{\clM(A)}d^{\pi^*_\tau}_{\rho} (ds)
\le 2 e^{-\tau t}\int_S 
\operatorname{KL}(\pi^*_\tau(\cdot|s)|\pi_0(\cdot|s) )d^{\pi^*_\tau}_{\rho} (ds)\,,
\end{equation*}
where  $\pi_t =\boldsymbol{\pi}(Z_t)$ for all $t\ge 0$ and $\pi^*_\tau$ is the optimal policy defined in \eqref{eq:optimal_policy}. The same convergence result also holds for the unique solution $\pi\in  C^1(\mathbb{R}_+; \Pi_\mu)$ of \eqref{eq:FR_flow}. 
\end{theorem}

\begin{remark}
\label{rmk:convergence}   
The exponential convergence in Theorem \ref{thm:linear_convergence} holds for all initial distributions, which implies that $(V^{\pi_t}_\tau (s))_{t\ge 0}$ converges to $V^{\pi_t}_\tau(s)$ for any $s\in S$.
This result provides valuable insights into the optimal theoretical convergence rate that a suitable discrete-time policy gradient method can achieve. 
It also highlights the gap between the convergence rates of existing discrete mirror descent updates and this optimal rate; see Section \ref{sec:discrete} for more details.

An interesting question is whether an analogous convergence result holds for the average-cost criterion (see e.g., \cite[Chapter 5]{hernandez2012discrete} and \cite{neu2017unified}). Addressing this would involve extending the current argument through a rigorous investigation of the geometry and regularity of the optimization landscape. Certain structural assumptions on MDPs may need to be imposed to ensure recurrence or ergodicity of  the controlled transition kernels along the flow.
\end{remark}

The proof of Theorem \ref{thm:linear_convergence} is given in Section \ref{sec:convergence_FR_flow}.
It relies on introducing a   Bregman divergence  on the dual space 
$B_b(S\times A) $ (see \eqref{eq:bregman_Dfg}),
and proving that along the  flow  \eqref{eq:mirror_descent}, 
this Bregman divergence is differentiable and coincides with the KL-divergence of the corresponding policies.

\subsection{Stability of the gradient flow}

In this section,  we state a stability result for the Fisher--Rao gradient flow \eqref{eq:FR_flow}, where the policies are updated with approximate $Q$-functions. This stability result will be utilised to quantify the performance of a policy gradient algorithm with log-linear parametrised policies. 

More precisely, consider policies $(\pi_t)_{t\ge 0}$ that satisfy the following approximate Fisher--Rao gradient flow: for $t\ge 0$,
\begin{align}
\label{eq:FR_flow_approx}
\left\{
\begin{aligned}
&\partial_t \pi_t(da|s )  =  - \left( Q_t(s,a)  + \tau \ln\frac{\mathrm{d}  \pi_t}{\mathrm{d}  \mu}(a|s)-\int_A \left(Q_t(s,a')  + \tau \ln\frac{\mathrm{d}  \pi_t}{\mathrm{d}  \mu}(a'|s)\right)
\pi_t(da'|s)\right)\pi_t(da|s )\,,
\\
&  \pi\vert_{t=0}=\pi_0\,.
\end{aligned}
\right.
\end{align}
where $Q: \mathbb{R}_+\to  B_b(S\times A))$ is a given function.
For each $t>0$, $Q_t$  approximates  $ Q^{\pi_t}_{\tau}$ and may be constructed depending on  $(\pi_r)_{r\in [0,t]}$; see the subsequent discussion for a concrete example. The dynamics \eqref{eq:FR_flow_approx} replaces $V^{\pi_t}_\tau $ in \eqref{eq:mirror_descent} by the average of  the approximate $Q$-function, which ensures that $\pi_t(\cdot|s)\in \clP(A) $  for all $t>0$ and $s\in S$.

Theorem \ref{thm:stability_mirror} quantifies the performance of $(\pi_t)_{t\ge 0}$ generated by \eqref{eq:FR_flow_approx}. The proof is given in Section \ref{sec:stability_FR} and leverages a mirror descent representation of \eqref{eq:FR_flow_approx} (see \eqref{eq:mirror_descent_inexact}).

\begin{theorem}
\label{thm:stability_mirror} 
Assume that $\pi\in C(\mathbb{R}_+; \Pi_\mu)$  satisfies   \eqref{eq:FR_flow_approx} with some $Q:\mathbb{R}_+\to  B_b(S\times A)$. Then  for all $\rho\in \clP(S)$ and  $t>0$, 
\begin{align}
\label{eq:stability_estimate_statement}
\begin{split}
\min_{r\in [0,t]}V^{\pi_{r}}_{\tau}(\rho)-V^{\pi^*_\tau}_{\tau}(\rho) 
& \le\frac{\tau   }{(1-\gamma)(e^{\tau t}-1)} \bigg(\int_S \operatorname{KL}(\pi^*_\tau(\cdot|s)|\pi_0(\cdot|s) )d^{\pi^*_\tau}_{\rho} (ds)  \\
&\quad +2 \kappa   \int_0^t e^{\tau r } \left\|Q^{\pi_r}_\tau   -Q_r\right\|_{L^1\big(S\times A,  \rho_{\rm ref}\otimes \frac{\pi_r+ \pi_{\rm ref}}{2}\big)} dr\bigg)\,,
\end{split}
\end{align}
where  $\pi^*_\tau$ is the optimal policy  defined in \eqref{eq:optimal_policy}, $\rho_{\rm ref} \in \clP(S)$ is a reference measure such that $\rho\ll \rho_{\rm ref}$, $\pi_{\rm ref}\in \clP(A|S)$ is a reference policy such that  $\mu\ll \pi_{\rm ref}$, and $\kappa\ge 1$ is the concentrability coefficient defined by
\begin{equation}
\label{eq:concentration_coefficient}
\kappa \coloneqq \left\|
\frac{\mathrm{d} d^{\pi^*_\tau  }_\rho    }{\mathrm{d} \rho_{\rm ref}  }
\right\|_{L^\infty(S , \rho_{\rm ref})}+  
\left\|\frac{\mathrm{d} d^{\pi^*_\tau  }_\rho  \otimes  
\pi^*_\tau }{\mathrm{d} \rho_{\rm ref}\otimes  \pi_{\rm{ref}} }
\right\|_{L^\infty(S\times A, \rho_{\rm ref}\otimes  \pi_{\textrm{ref}})}\,.
\end{equation}
The estimate \eqref{eq:stability_estimate_statement} holds with  $Q^{\pi_r}_\tau   -Q_r$ replaced by $Q^{\pi_r}_\tau   -Q_r+ F_r$ for any measurable $F:\mathbb{R}_+\to B_b(S)$. 
\end{theorem}
\begin{remark}
Theorem \ref{thm:stability_mirror} indicates that  the policies $(\pi_t)_{t\ge 0}$ 
in \eqref{eq:mirror_descent_inexact} converge to the optimal policy at the same exponential rate as in Theorem \ref{thm:linear_convergence}, subject to a policy evaluation error   $(Q_t-Q^{\pi_t}_{\tau})_{t\ge 0}$. 
The approximation error of $(Q_t)_{t\ge 0}$
is quantified with respect to a suitable reference measure $\rho_{\rm ref}$
and a reference policy $\pi_{\rm ref}$.
The similarity between this reference  policy and the optimal policy is quantified by the concentrability coefficient $\kappa$, which is commonly used in the reinforcement learning literature (see, e.g., \cite{liu2019neural, agarwal2020optimality, mei2020global, wang2019neural, cayci2021linear, zhang2021wasserstein}).

Theorem \ref{thm:stability_mirror}   provides a unified approach for analyzing policy gradient methods with inexact Q-functions $(Q_t)_{t\ge 0}$ estimated from data. It    reduces the task  to quantifying  the Q-function error 
$ (Q^{\pi_t}_\tau-Q_t)_{t\ge 0}  $.
To bound this error, precise parameterizations and estimation procedures for the Q-function must be specified, and additional structural assumptions on the MDP may be necessary to ensure persistent exploration. 
See  \cite{agarwal2020optimality,   cen2022fast, xiao2022convergence, zhan2023policy}
for more details in the tabular setting.

\end{remark}

\subsection{Convergence of natural gradient flow for log-linear policies}\label{sec:npg_log_linear}
In what follows, we will apply Theorem \ref{thm:stability_mirror} to quantify the accuracy of policies  $(\pi_t)_{t\ge 0}$ arising from a continuous-time natural policy gradient algorithm (see e.g., \cite{kakade2001natural, cayci2021linear}). The algorithm considers log-linear parametrised policies  and updates the policy parameter by incorporating the local geometry of the parameter space. More precisely, let $(\bbH, \|\cdot\|_{\bbH}) $ be a Hilbert space with the inner product $\langle\cdot, \cdot\rangle_{\bbH}$ representing the parameter space (e.g., $\bbH=\mathbb{R}^N$ or $\bbH=\ell^2$) and let $g\in  B_b(S\times A; \bbH)$ be a fixed feature basis. Consider minimising the value function \eqref{eq:V_pi_tau} over the following parametrised policies: 
\[
\left\{\pi_{\theta} = \boldsymbol{\pi}(\langle\theta, g(\cdot)\rangle_{\bbH})\mid \theta\in \bbH \right\}\subset \Pi_\mu\,,
\]
where $ \boldsymbol{\pi} :B_b(S\times A)\to\Pi_\mu $ is  defined by \eqref{eq:pi_f_mu}.
Instead of updating $\pi_\theta$ along the natural policy gradient flow \eqref{eq:NPG} derived in Section \ref{sec:introduction},  
for convenience, we consider the following approximate natural policy gradient flow: 
\begin{equation}
\label{eq:npg_clipping}
\partial_t {\theta}_t = -(w_t(\theta_t) + \tau \theta_t), \quad t>0; \quad \theta\vert_{t=0}=\bar{\theta}\,, 
\end{equation}
where $w_t$ is  the minimiser of the following  regularised loss 
\begin{equation}\label{eq:regularised_loss}
w_t(\theta_t) = \underset{\|w\|_{\bbH}\le R_t}{\operatorname{arg\,min}}\left(\int_{S}\int_A |A^{\pi_{\theta_t}}_{\tau}(s,a)  - \langle w, g_{\pi_{\theta_t}}(s,a)\rangle_{\bbH}|^2 \pi_{\theta_t}(da|s)d_{\rho}^{\pi_{\theta_t}}(ds)+\lambda_t \|w\|^2_{\bbH}\right)\,,
\end{equation}
$g_{\pi_{\theta_t}}$ is the centered feature defined in \eqref{eq:linear_policy}, 
and $R_t>0$ and $ \lambda_t>0$ are given parameters. The regularisation term $\lambda_t \|w\|^2_{\bbH}$ ensures that $w_t$ exists and depends locally Lipschitz continuously on $\theta_t$. The gradient clipping with radius $R_t>0$ yields an a priori  bound of $w_t$, which allows for proving the well-posedness of \eqref{eq:npg_clipping}. In practice, one can set $\lambda_t\to 0$ and $R_t\to \infty$ as $t\to \infty$ for the  consistency between  \eqref{eq:loss_Q} and \eqref{eq:regularised_loss}.

The following theorem establishes the well-posedness of \eqref{eq:npg_clipping} and further quantifies the performance of $(\pi_{\theta_t})_{t\ge 0}$. 
The proof is given in Section \ref{sec:NPG_proof}. 
The well-posedness of~\eqref{eq:npg_clipping} relies on analysing the regularity of $w_t$ in $t$ and $\theta_t$. The performance guarantee follows directly from  
Theorem~\ref{thm:stability_mirror} and the observation that  $(\pi_{\theta_t})_{t\ge 0}$ satisfies 
the approximate Fisher--Rao flow \eqref{eq:FR_flow_approx} with  $Q_t(s,a) =\langle w_t,  g(s,a)\rangle_{\bbH}$ for all $t>0$ (see \eqref{eq:npg_Fisher_Rao}). 

\begin{theorem}
\label{thm:NPG_stability}
Let $\rho\in \clP(S)$, $\bar{\theta}\in \bbH$ and  $R,\lambda\in C(\mathbb{R}_+; (0,\infty))$. Then there exists a unique $\theta \in C^1(\mathbb{R}_+; \bbH)$ satisfying \eqref{eq:npg_clipping} for all $t>0$. Moreover,  for all $t>0$, 
\begin{align}
\label{eq:stability_estimate_NPG}
\begin{split}
\min_{r\in [0,t]}V^{\pi_{\theta_r}}_{\tau}(\rho)-V^{\pi^*_\tau}_{\tau}(\rho) & \le  \frac{\tau}{(1-\gamma)(e^{\tau t}-1)} \bigg(\int_S \operatorname{KL}(\pi^*_\tau(\cdot|s)|\pi_{\bar{\theta}}(\cdot|s) )d^{\pi^*_\tau}_{\rho} (ds) \\
&\quad +2 \kappa \int_0^t e^{\tau r } \left\|A^{\pi_{\theta_r}}_{\tau}(\cdot) -\langle w_r,   g_{\pi_{\theta_r}}(\cdot)\rangle_{\bbH}\right\|_{L^1\big(S\times A,  \rho \otimes\frac{\pi_r+\pi_{\rm ref}}{2}\big)} dr\bigg)\,,
\end{split}
\end{align}
where $\pi^*_\tau$ is the optimal policy  defined in \eqref{eq:optimal_policy}, $\pi_{\rm ref}\in \clP(A|S)$ is a reference policy such that $\mu\ll \pi_{\rm ref}$, and $\kappa\ge 1$ is the concentrability coefficient defined by
\[
\kappa \coloneqq \left\|\frac{\mathrm{d} d^{\pi^*_\tau  }_\rho}{\mathrm{d}\rho}\right\|_{L^\infty(S , \rho )}+ \left\|\frac{\mathrm{d} d^{\pi^*_\tau  }_\rho  \otimes  \pi^*_\tau }{\mathrm{d} \rho \otimes  \pi_{\rm{ref}} }\right\|_{L^\infty(S\times A, \rho \otimes  \pi_{\textrm{ref}})}\,.
\]
\end{theorem}
\begin{remark}
Theorem~\ref{thm:NPG_stability} shows that the policies $(\pi_{\theta_t})_{t\ge 0}$ associated with \eqref{eq:npg_clipping} converge exponentially to the optimal policy,  up to the approximation error $(A^{\pi_{\theta_t}}_{\tau} -\langle w_t, g_{\pi_{\theta_t}}\rangle_{\bbH})_{t\ge 0}$ (cf.~\eqref{eq:regularised_loss}). This result extends the performance guarantees of natural policy gradient methods for MDPs with discrete state and action spaces in \cite{agarwal2020optimality, cayci2021linear} to the present setting with continuous state and action spaces. It would be of interest to select appropriate regularisation parameters $R$ and $\lambda$ to ensure that the approximation error decays sufficiently fast as $t\to \infty$. A comprehensive convergence analysis in this direction is left for future research.
\end{remark}

\subsection{Discussion: discretizations of gradient flows and their convergences}
\label{sec:discrete}

The analysis of the continuous-time flow \eqref{eq:mirror_descent}
provides a theoretical foundation for deriving various discrete policy gradient algorithms through appropriate timestepping schemes. 
Potential timestepping choices include the explicit Euler method and higher-order explicit RK methods.
The analysis of these discrete policy gradient methods is more intricate than for the continuous-time flow, as it requires establishing suitable discrete analogues of the continuous-time analysis.
In this section, we analyze the convergence of the discrete update resulting from   
the simplest discretization of \eqref{eq:mirror_descent}, i.e., the explicit Euler method.
The convergence of other discretization methods is illustrated numerically in Appendix \ref{sec:numerical}.

More precisely, let $\lambda>0$,
$Z_0\in B_b(S\times A)$,
and consider updating $(Z^n)_{n\in \mathbb{N}_0}$   using the Euler discretization of \eqref{eq:mirror_descent}:
$Z^0=Z_0$ and 
for all $n\in \mathbb N_0$,
\begin{equation}
\label{eq:mirror_descent_euler}
Z^{n+1}(s,a) =Z^{n}(s,a)    -\frac{1}{\lambda}  \left( Q^{\pi^n}_\tau  (s,a) + \tau Z^n (s,a) -V^{\pi^n }_\tau (s) \right), 
\quad \textnormal{with $ \pi^n    \coloneqq  \boldsymbol{\pi}(Z^n)$}\,. 
\end{equation}
As shown in \eqref{eq:pointwise_min},
the corresponding update for $(\pi^n)_{n\in\mathbb N_0}$ 
is given by
\begin{align}
\label{eq:mirror_des_direct_appendix}
\begin{split}
\pi^{n+1}(\cdot|s) 
=  \underset{m \in \clP(A)}{\operatorname{arg\,min}}\left(\int_{A}\frac{\delta V^{\pi^n}_\tau}{\delta \pi}(s,a)(m(da) - \pi^n(da|s)) + \lambda \operatorname{KL}(m|\pi^n(\cdot|s))\right)
\,,
\end{split}
\end{align}
which  extends  the policy mirror descent algorithm proposed  in \cite{lan2022policy, xiao2022convergence, zhan2023policy}  for tabular MDPs to the current setting.
The following theorem presents the convergence rate of the update \eqref{eq:mirror_des_direct_appendix},
whose proof is given in Section \ref{sec:proof_mirror_descent}.

\begin{theorem}
\label{thm:convergence_mirror_steps}
Let 
$(\pi^n)_{n\in \bbN_0} \subset  \Pi_\mu$ be given by~\eqref{eq:mirror_des_direct_appendix}.
For all
$\rho\in \mathcal{P}(S)$,  $\lambda \ge \tau$ and $n\in \bbN_0$,
\begin{equation}
\label{eq:linear_convergence}
0 \leq  V^{\pi^n}_\tau (\rho)-V^\ast_\tau(\rho)
\leq  
\frac1{1-\gamma}\bigg(\lambda\int_S \operatorname{KL}
(\pi^*_\tau(\cdot|s)|\pi_0(\cdot|s) )
d^{\pi^\ast_\tau}_\rho(ds) +  \int_S (V^0_\tau-V^\ast_\tau)(s)	d^{\pi^\ast_\tau}_\rho(ds) \bigg)\frac{1}{n}\,.
\end{equation}
Assume in addition that 
$d^{\pi^\ast_\tau}_\rho\ll \rho$
and   $\xi := (1-\gamma)^{-1}\big\|\frac{\mathrm d d^{\pi^\ast_\tau}_\rho}{\mathrm d \rho}\big\|_{B_b(S)} < \infty$. Then for all   $\tau\le \lambda \leq \xi \tau$, 
\begin{equation}
\label{eq:mirror_exp_conv_with_dist_mismatch_coeff}
\begin{split}
& V^{\pi^n}_\tau (\rho)-V^\ast_\tau(\rho) + \frac{\lambda}{(1-\gamma)\xi}\int_S \operatorname{KL}(\pi^\ast_\tau(\cdot|s)| \pi^n(\cdot| s)) d^{\pi^\ast_\tau}_\rho(ds)\\
&\quad  \leq \bigg(\frac{\xi-1}{\xi}\bigg)^n\bigg( V^{\pi^0}_\tau(\rho)   - V^\ast_\tau(\rho) +\frac{\lambda}{(1-\gamma)\xi}\int_S \operatorname{KL}
(\pi^\ast_\tau(\cdot|s)| \pi^0(\cdot| s))
d^{\pi^\ast_\tau}_\rho(ds)\bigg)\,. 		
\end{split}
\end{equation}

\end{theorem}

\begin{remark}
\label{rmk:discrete_mirror}

Theorem \ref{thm:convergence_mirror_steps} 
achieves the best-known convergence rate for policy mirror descent updates, as established in \cite{lan2022policy} for MDPs with finite-dimensional action spaces, under the additional ergodicity condition that  each policy induces a unique stationary distribution of the MDP, and the initial distribution 
$\rho$  coincides with the stationary distribution of the optimal policy $\pi^*_\tau$. 
Theorem~\ref{thm:convergence_mirror_steps} improves upon this result by extending it to MDPs with Polish state and action spaces, removing the stringent ergodicity condition, and accommodating any initial distribution that dominates the optimal state occupancy measure.

However, it is unclear whether exponential convergence of the   update \eqref{eq:mirror_des_direct_appendix} can be established for general   initial distributions, as   is the case for the continuous-time flow \eqref{eq:mirror_descent}  in Theorem \ref{thm:linear_convergence}. The  condition  $d^{\pi^\ast_\tau}_\rho\ll \rho$  can be interpreted as a persistent exploration condition for the MDP, which is also imposed in \cite{xiao2022convergence} for analyzing policy gradient methods for tabular MDPs.
The distinction between the error bounds for the continuous-time flow
\eqref{eq:mirror_descent} in 
Theorem  
\ref{thm:linear_convergence} and 
the discrete-time update 
\eqref{eq:mirror_des_direct_appendix}
in Theorem 
\ref{thm:convergence_mirror_steps} 
suggests that there is potential to design more effective discrete updates through alternative timestepping schemes for \eqref{eq:mirror_descent}. Our numerical experiments in Appendix \ref{sec:numerical} provide an initial step  in this direction, showing that higher-order discretizations of  \eqref{eq:mirror_descent}  can achieve smaller errors than \eqref{eq:mirror_des_direct_appendix}. A detailed study of these alternatives is left for future work.

\end{remark}

\subsection{Discussion: gradient flow for unregularised MDPs}\label{sec:unregularised}
The gradient flow \eqref{eq:FR_flow} 
and its analysis  extend  natually   to unregularized MDPs.
Indeed, 
setting $\tau$ to $0$ in \eqref{eq:FR_flow} yields the following  dynamics for $(\pi_t)_{t\ge 0}$:
\begin{equation}\label{eq:FR_unregularised}
\partial_t \pi_t(da|s) =-\left(Q^{\pi_t}(s,a)-V^{\pi_t}(s) \right)\pi_t(da|s),\quad t>0\,,    \end{equation}
where $Q^{\pi_t}$ and $V^{\pi_t}$ are the $Q$-function and the value function for the unregularised MDP, respectively (cf.~\eqref{eq:V_pi_tau} and \eqref{def:Q_fn} with $\tau=0$). Formally sending $\tau\to 0$ in Theorem \ref{thm:linear_convergence} suggests that for all $\rho\in \clP(S)$ and $t>0$,
\begin{equation}\label{eq:sublinear}
V^{\pi_{t}}(\rho)-V^{\pi^*}(\rho)\le \frac{1 }{(1-\gamma)t } \int_S \operatorname{KL}(\pi^* (\cdot|s)|\pi_0(\cdot|s) )d^{\pi^*}_{\rho} (ds)\,,
\end{equation}
where $\pi^*$ is an optimal policy  of the unregularised MDP (if it exists, see, e.g., \cite{hernandez2012discrete}). In other words, $(V^{\pi_t})_{t\ge 0}$ converges to the optimal (unregularised) value function at a polynomial rate.  

However, three technical difficulties need to be addressed to prove the estimate \eqref{eq:sublinear} rigorously:
(i)  An optimal policy $\pi^*$ may  not exist for  continuous state and action spaces, 
due to a lack of compactness.
(ii) The term $\int_S \operatorname{KL}(\pi^*(\cdot|s)|\pi_0(\cdot|s))d^{\pi^*}_\rho(ds)$ may not be finite  if there is a substantial discrepancy between the optimal policy and the initial policy.
(iii) The map  $t\mapsto \int_S \operatorname{KL}(\pi^*(\cdot|s)|\pi_t(\cdot|s))d^{\pi^*}_\rho(ds)$ 
may not be differentiable, which poses challenges for the convergence analysis.

These technical challenges can be overcome by assuming the existence of an optimal policy $\pi^*$ and initialising \eqref{eq:FR_unregularised} with a policy that is compatible with $\pi^*$. For instance, assume that there exists an optimal policy $\pi^*\in \Pi_\mu$ for some  $\mu\in \clP(A)$.  In this case, one can initialise \eqref{eq:FR_unregularised} with $\pi_0\in \Pi_\mu$ (i.e., $\pi_0=\boldsymbol{\pi}(Z_0)$ for some $Z_0\in B_b(S\times A)$) and adapt the mirror descent approach to the unregularised setting. In particular, given $Z_0\in B_b(S\times A)$, consider for all $t\ge 0$,
\begin{equation}\label{eq:MD_unregularised}
\partial_t Z_t(s,a) = \, - \left( Q^{\pi_t}(s,a)  -V^{\pi_t} (s)\right), \quad  \pi_t (da|s)= \boldsymbol{\pi}(Z_t)(da|s)\,,
\end{equation}
with $\boldsymbol{\pi} :B_b(S\times A) \to \Pi_\mu$ is defined by \eqref{eq:pi_f_mu}.
Then similar to Section \ref{sec:MD_approach_main}, one can prove that \eqref{eq:MD_unregularised} admits a unique solution  $Z\in C^1(\mathbb{R}_+; B_b(S\times A) )$ and the policies $(\pi_t)_{t\ge 0}$ satisfies \eqref{eq:FR_unregularised}. The differentiability of   $t\mapsto \int_S \operatorname{KL}(\pi^*(\cdot|s)|\pi_t(\cdot|s))d^{\pi^*}_\rho(ds)$ can be proved by leveraging  the differentability of  $t\mapsto D_{d^{\pi^*}_{\rho}}(Z_t, Z^*)$, with $Z^*\in B_b(S\times A)$ satisfying $\pi^*=\boldsymbol{\pi} (Z^*)$. This allows for proving the polynomial convergence rate given in \eqref{eq:sublinear}. 

We end this section with an illustrative example that highlights the crucial role of the absolute continuity of $\pi^*$ with respect to the initial policy $\pi_0$ in the convergence of \eqref{eq:FR_unregularised}. 

\begin{example}\label{example:bandit}
Consider the bandit setting where $S=\emptyset$ and  $\gamma=0$. Let $a_0\in A$ and $c(a) = -1_{\{a_0\}}(a)$ for all $a\in A$. The optimal policy is $\pi^*=\delta_{a_0}$ and the optimal value is $V^* = -1$. For each  $\mu\in \mathcal{P}(A)$, \eqref{eq:FR_unregularised} simplifies to  the following dynamics: for all $t\ge 0$,
\begin{equation}\label{eq:birth-death_bandit}
\partial_t \pi_t(da)  =-\left(c( a)-\int_A c(a)\pi_t(da) \right)\pi_t(da),
\quad \pi\vert_{t=0} =\mu\,.    
\end{equation}
Equation \eqref{eq:birth-death_bandit} has a unique solution in $C^1(\mathbb{R}_+; \clP(A))$ and is given by $\pi_t(da)=\frac{\exp(Z_t(a))\mu(da)}{\int_A \exp(Z_t(a))\mu(da)}$,
where  $Z\in C^1(\mathbb{R}_+; B_b(A))$ satisfies $Z_0=0$ and $\partial_t Z_t  =- c $ for all $t>0$.
Since  $c=-1_{\{a_0\}}$, for all $t>0$, $Z_t = t 1_{\{a_0\}}$ and   
\begin{align*}
V^{\pi_t}&=\int_A c(a)\pi_t(da)=-\pi_t(\{a_0\})=-\frac{\int_{\{a_0\}} \exp(Z_t(a))\mu(da) }{
\int_{\{a_0\}} \exp(Z_t(a))\mu(da)+\int_{A\setminus\{a_0\}} \exp(Z_t(a))\mu(da)}\\
&=-\frac{ \exp(t)\mu(\{a_0\}) }{\exp( t)\mu(\{a_0\})+\mu(A\setminus\{a_0\})}=-\frac{ \mu(\{a_0\}) }{
\mu(\{a_0\})+\exp(-t) \mu(A\setminus\{a_0\})}\,.
\end{align*}
Hence, $\lim_{t\to \infty} V^{\pi_t}=V^*$ if and only if $\mu(\{a_0\})>0$, which is equivalent to the condition that $\pi^* $ is absolutely continuous with respect to $\mu$.
\end{example}

Example \ref{example:bandit} shows that in the unregularised setting,
without imposing additional regularity conditions on the cost functions, ensuring the convergence of \eqref{eq:FR_unregularised}  requires  initialising \eqref{eq:FR_unregularised} with positive masses assigned to the exact support of $\pi^*$,
which is required to guarantee persistent exploration along the flow (see Remark~\ref{rmk:KL_role}).
This is not an issue if the action space has finite cardinality, as one can initialise the flow with the uniform distribution, which dominates all probability measures supported on the action space.
However, for continuous action spaces,
this necessitates knowing the precise support of $\pi^*$ a priori, which is typically infeasible for unregularised MDPs.
In contrast, for regularised MDPs, the support of the optimal policy  $\pi^*_\tau$ is determined by the reference measure specified in the KL divergence, providing more manageable conditions for the convergence of gradient flows.

\section{Properties of regularised MDPs}
\label{sec:prop_of_MDPs}
Before we can prove the main results on convergence and stability of the mirror descent flow, we need to establish some key properties of entropy-regularised MDPs.

\subsection{Boundedness and regularity}
\label{sec:basis_properties_MDP}

This section proves the regularity of the value function and $Q$-function, including Lemma \ref{lem:performance_diff}. We begin by establishing the boundedness of the value function and the $Q$-function.

\begin{proposition}\label{prop:boundedness_Q}
Let $f\in B_b(S\times A)$ and $\pi\in \Pi_\mu$  be such that  $\pi(da|s)= \frac{\exp(f(s,a))\mu(da)}{\int_A \exp(f(s,a'))\mu(da')}$ 
for all $s\in S$. Then
\begin{align*}
&\left\|\ln \frac{\mathrm{d} \pi}{\mathrm{d} \mu}\right\|_{B_b(S\times A)} \le 2\|f\|_{B_b(S\times A)}\,, \quad \|V^\pi_\tau\|_{B_b(S)} \le \frac{1}{1-\gamma}\left(\|c\|_{B_b(S\times A)}+ 2\tau \|f\|_{B_b(S\times A)}\right)\,,\\
&\|Q^\pi_\tau\|_{B_b(S\times A)}\le \frac{1}{1-\gamma}\left(\|c\|_{B_b(S\times A)}+ 2\tau \gamma\|f\|_{B_b(S\times A)}\right)\,.
\end{align*}   
\end{proposition}
\begin{proof}
As $\mu(A)=1$,
for all $g\in B_b(S\times A)$ and $s\in S$, 
\begin{align*}
\begin{split}
\ln\int_A\exp(g(s,a'))\mu(da')&\le \ln\left(e^{\|g\|_{B_b(S\times A)}}\mu(A)\right) =\|g\|_{B_b(S\times A)}\,,\\
\ln\int_A\exp(g(s,a'))\mu(da')&\ge \ln\left(e^{-\|g\|_{B_b(S\times A)}}\mu(A)\right) =   -\|g\|_{B_b(S\times A)}\,.
\end{split}
\end{align*}
Then, for all $(s,a)\in S\times A$, using  $ \ln \frac{\mathrm{d} \pi}{\mathrm{d} \mu}(a|s) = f(s,a)-\ln\int_A\exp(f(s,a'))\mu(da')$, 
\begin{align*}
\left|\ln \frac{\mathrm{d} \pi}{\mathrm{d} \mu}(a|s)\right| \le |f(s,a)| + \left|\ln\int_A\exp(f(s,a'))\mu(da')\right|\le 2 \|f\|_{B_b(S\times A)}\,,
\end{align*}
which implies that 
\begin{align*}
\left|\bbE_{s}^{\pi}\left[\sum_{t=0}^\infty\gamma^t \bigg(  \tau \ln\frac{\mathrm{d} \pi}{\mathrm{d} \mu}(a_t|s_t)\bigg)\right]\right|
\le  2\tau \|f\|_{B_b(S\times A)}\sum_{t=0}^\infty \gamma^t =\frac{2\tau \|f\|_{B_b(S\times A)}}{1-\gamma}\,.
\end{align*}
By \eqref{eq:V_pi_tau}, for all $s\in S$,  
\begin{align*} 
|V^\pi_\tau (s)|&\le \bbE_{s}^{\pi}\left[\sum_{t=0}^\infty\gamma^t |c(s_t,a_t)|\right]+ 
\left|\bbE_{s}^{\pi}\left[\sum_{t=0}^\infty\gamma^t \bigg(  \tau \ln\frac{\mathrm{d} \pi}{\mathrm{d} \mu}(a_t|s_t)\bigg)\right]\right|\\
& \le \frac{1}{1-\gamma}\left(\|c\|_{B_b(S\times A)}+ 2\tau \|f\|_{B_b(S\times A)}\right).
\end{align*}
Hence, for all $(s,a)\in S\times A$, by \eqref{def:Q_fn},
\[
|Q^\pi_\tau (s,a)|\le \|c\|_{B_b(S\times A)}+\gamma   \|V^\pi_\tau\|_{B_b(S)}
\le \frac{1}{1-\gamma}\|c\|_{B_b(S\times A)}
+\frac{2\tau\gamma}{1-\gamma}\|f\|_{B_b(S\times A)}\,.
\]
This proves the desired bound of $Q_\tau^\pi$.
\end{proof}

The following lemma expresses the resolvent of the transition kernel using the occupancy kernel, which will be used in proving Lemma \ref{lem:performance_diff}.

\begin{lemma}\label{lem:inverse_1-gammaP}
Let  $\pi\in \clP(A|S)$ and $f,g\in B_b(S)$ be such that for all  $s\in S$, 
\begin{equation*}
f(s) = \gamma  \int_S \int_A f(s) P(ds'|s,a)\pi(da|s)+g(s)\,.    
\end{equation*}
Then $f(s) = \frac{1}{1-\gamma}\int_S  g(s') d^{\pi}(ds'|s) $ for all $s\in S$.
\end{lemma}
\begin{proof}
Recall that a kernel $k\in b\clM(S|S)$ induces a linear operator $L_k\in \clL(B_b(S))$ such that for all $h\in B_b(S)$, $L_k h(s)=\int_S h(s')k(ds'|s)$. Since $\|L_{k}  h\|_{ B_b(S)}\le \|h\|_{B_b(S)}\|k\|_{b\clM(S|S)}$ for all $h\in B_b(S)$, $\|L_k\|_{\clL(B_b(S))}\le \|k\|_{b\clM(S|S)}$. Consider the kernel $\gamma P_{\pi} \in b\clM(S|S)$  defined by $(\gamma P_{\pi})(B)=\gamma\int_B \int_A P(ds'|s,a) \pi(da|s)$ for all $B\in \clB(S)$. 
Then  as $P_\pi\in \clP(S|S)$
and $\|  P_{\pi}\|_{b\clM(S|S)}=1$,
\[
\|L_{\gamma P_{\pi}}  \|_{\clL(B_b(S))} \le \|\gamma P_{\pi}\|_{b\clM(S|S)}=\gamma \| P_{\pi}\|_{b\clM(S|S)}=\gamma\,.
\] 

The condition on $f$ and $g$ implies that $(\operatorname{id}-L_{\gamma P_{\pi}})f =g $, where $\operatorname{id} $ is the identity operator on $B_b(S)$. As $\|L_{\gamma P_{\pi}}\|_{\clL(B_b(S))}\le \gamma <1$, the operator $\operatorname{id}-L_{\gamma P_{\pi}} \in \clL(B_b(S))$ is invertible, and the inverse operator is given by the Neumann series $(\operatorname{id}-L_{\gamma P_{\pi}})^{-1}=\sum_{n=0}^\infty L_{\gamma P_{\pi}}^n$.  Thus, $f=\sum_{n=0}^\infty L_{\gamma P_{\pi}}^n g$. Observe that $L_{\gamma P_{\pi}}^n= L_{\gamma^n P^n_{\pi}}$ for all $n\in \bbN_0$, where  $P^n_{\pi}$ is the $n$-times  product of the kernel $P_{\pi}$ with $P^0_{\pi}(ds'|s)\coloneqq \delta_s(ds')$. Then by the definition \eqref{eq:occupancy_s} of $d^{\pi}\in \clP(S|S)$,  $f=\sum_{n=0}^\infty L_{\gamma P_{\pi}}^n g= L_{(1-\gamma)^{-1}d^{\pi}}g $. This proves the desired identity. 
\end{proof} 

We are now ready to prove Lemma  \ref{lem:performance_diff}.

\begin{proof}[Proof of Lemma \ref{lem:performance_diff}]
By \eqref{eq:on_policy}, for all $s\in S$,
\begin{align*}
\begin{split}
&V^\pi_{\tau}(s) - V^{\pi'}_{\tau}(s) \\
&\quad = \int_A\left( Q^{\pi}_{\tau}(a|s)+ \tau \ln \frac{\mathrm{d} \pi}{\mathrm{d} \mu}(a|s) \right)\pi(da|s) - \int_A \left( Q^{\pi'}_{\tau}(s,a)+ \tau \ln \frac{\mathrm{d} \pi'}{\mathrm{d} \mu}(a|s) \right)\pi'(da|s)\\
&\quad  =\int_A\left(Q^{\pi'}_{\tau}(s,a)+\tau \ln \frac{\mathrm{d} \pi'}{\mathrm{d} \mu}(a|s)\right)(\pi-\pi')(da|s) \\
&\qquad+ \int_A\left( Q^{\pi}_{\tau}(s,a)+ \tau \ln \frac{\mathrm{d} \pi}{\mathrm{d}\mu}(a|s)-Q^{\pi'}_{\tau}(s,a) -\tau \ln \frac{\mathrm{d} \pi'}{\mathrm{d} \mu}(a|s) \right)\pi(da|s)\\
&\quad  =\int_A\left(Q^{\pi'}_{\tau}(s,a)+\tau \ln \frac{\mathrm{d}\pi'}{\mathrm{d} \mu}(a|s)\right)(\pi-\pi')(da|s)\\
&\qquad +\gamma\int_A \int_S \left(V^\pi_{\tau}(s') - V^{\pi'}_{\tau}(s')\right)P(ds'|s,a) \pi(da|s)+ \tau  \operatorname{KL}(\pi(\cdot|s)|\pi'(\cdot|s))  \,.
\end{split}
\end{align*}
where the last equality used \eqref{def:Q_fn} and the fact that  $ \operatorname{KL}(\pi(\cdot|s)|\pi'(\cdot|s)) =\int_A  \ln \frac{\mathrm{d} \pi}{\mathrm{d} \pi'}(a|s) \pi(da|s)$.
Hence, by Fubini's theorem and  Lemma \ref{lem:inverse_1-gammaP}, for all $s\in S$, 
\begin{align*}
&V^{\pi}_\tau(s)-V^{\pi'}_\tau(s) \\
&\quad = \frac{1}{1-\gamma}\int_S \bigg[\int_A\left(Q^{\pi'}_{\tau}(s',a)+\tau \ln \frac{\mathrm{d} \pi'}{\mathrm{d}\mu}(a|s')\right)(\pi-\pi')(da|s') + \tau    \operatorname{KL}(\pi(\cdot | s')|\pi'(\cdot | s')) \bigg]d^{\pi}(ds'|s).
\end{align*}
Integrating both sides with respect to $\rho$ yields the desired identity.
\end{proof}

Based on Lemma \ref{lem:performance_diff} and Proposition \ref{prop:boundedness_Q}, the following proposition quantifies the difference $Q^{\pi'}_\tau-Q^{\pi}_\tau$ in terms of $\pi$ and $\pi'$.

\begin{proposition}\label{prop:Q_continuity}
Let $\pi,\pi'\in \Pi_\mu$ be such that $\pi(da|s)= \frac{\exp(f(s,a))\mu(da)}{\int_A \exp(f(s,a'))\mu(da')}$ for all $s\in S$. Then 
\begin{align*}
\|Q^{\pi'}_\tau-Q^{\pi}_\tau\|_{B_b(S\times A)} &\le \frac{\gamma}{(1-\gamma)^2}
\left(\|c\|_{B_b(S\times A)}+2\tau  \|f\|_{B_b(S\times A)}\right)\| \pi-\pi'\|_{b\clM(A|S)} \\
&\quad +\frac{\tau \gamma}{1-\gamma}\left\|\ln \frac{\mathrm{d} \pi'}{\mathrm{d} \pi}   \right\|_{B_b(S\times A)}\,. 
\end{align*}
\end{proposition}
\begin{proof}
By Lemma \ref{lem:performance_diff}, for all $s\in S$,
\begin{align*}
|V^{\pi}_{\tau}(s)-V^{\pi'}_{\tau}(s) |&\le \frac{1}{1-\gamma}\left|\int_S\int_A\left(Q^{\pi'}_{\tau}(s',a)+\tau \ln \frac{\mathrm{d} \pi'}{\mathrm{d} \mu}(a|s')\right)(\pi-\pi')(da|s') d^{\pi}_{s}(ds')\right|\\
&\quad +\frac{\tau}{1-\gamma} \int_S\int_A\ln \frac{\mathrm{d} \pi}{\mathrm{d} \pi'}(a|s')\pi(da|s')d^{\pi}_{s}(ds') \\
&\le \frac{1}{1-\gamma}\left\| Q^{\pi'}_{\tau}+\tau \ln \frac{\mathrm{d} \pi'}{\mathrm{d} \mu}\right\|_{B_b(S\times A)}\| \pi-\pi'\|_{b\clM(A|S)} +\frac{\tau }{1-\gamma}\left\|\ln \frac{\mathrm{d}\pi}{\mathrm{d} \pi'}   \right\|_{B_b(S\times A)}.
\end{align*}
Thus, by \eqref{def:Q_fn}, for all $(s,a)\in S\times A$, 
\[
|Q^{\pi'}_\tau (s,a)-Q^{\pi}_\tau (s,a)| \le \gamma \int_S | V^{\pi'}_\tau (s')- V^{\pi}_\tau (s')| P(ds'|s,a)\le \gamma\| V^{\pi'}_\tau  - V^{\pi}_\tau  \|_{B_b(S)}\,.
\]
By Proposition \ref{prop:boundedness_Q},
\begin{align*}
\left\| Q^{\pi'}_{\tau}+\tau \ln \frac{\mathrm{d} \pi'}{\mathrm{d} \mu}\right\|_{B_b(S\times A)}& \le \frac{1}{1-\gamma}\left(\|c\|_{B_b(S\times A)}+2\tau\gamma \|f\|_{B_b(S\times A)}\right)+2\tau \|f\|_{B_b(S\times A)}\\
& \le \frac{1}{1-\gamma}\left(\|c\|_{B_b(S\times A)}+2\tau  \|f\|_{B_b(S\times A)}\right).
\end{align*}
Combining the above inequalities yields the desired estimate.
\end{proof}

\subsection{Differentiability of the value function, policy, and Bregman divergence}
\label{sec:derivatives}

We first introduce the notion of Hadamard differentiability used in this paper (see, e.g., \cite{bonnans2013perturbation}).  
This notion is chosen as it is the weakest form of differentiability  for which a chain rule holds (see Lemma \ref{lemma:chain_rule}).

\begin{definition}
Let $X$ and $Y$ be Banach spaces.  We say $f:X\to Y$ is (Hadamard) differentiable if there exists $\mathfrak{d} f:X\to \clL(X,Y)$,  
called the  differential of $f$, such that for all $x, v \in X$,
\[
\lim_{n\to \infty}\frac{f(x+t_n v_n)-f(x)}{t_n}=\mathfrak{d}f(x)v
\]
for all sequences  $(t_n)_{n\in \mathbb{N}} \subset (0,1)$ and $(v_n)_{n\in \mathbb{N}}\subset X$ such that  $\lim_{n\to \infty} t_n =0$ and $\lim_{n\to \infty} v_n =v$. 
\end{definition}

The following lemma shows that function composition preserves the Hadamard differentiability. The proof can be found in \cite[Proposition 2.47]{bonnans2013perturbation} and is given below for the reader's convenience. 

\begin{lemma}[Chain rule]\label{lemma:chain_rule}
Let $X$, $Y$ and $Z$ be Banach spaces, and let $f:Y\to Z$ and $g:X\to Y$ be  differentiable functions with differentials $\mathfrak{d}f:Y\to \clL(Y,Z)$ and $\mathfrak{d}g:X\to \clL(X,Y)$, respectively. Then  $f\circ g:X\to Z$ is differentiable and  $\mathfrak{d}(f\circ g)(x)=\mathfrak{d}f(g(x))\mathfrak{d}g(x) $ for all $x\in X$.
\end{lemma}
\begin{proof}
Fix $x,v\in X$, and sequences  $(t_n)_{n\in \mathbb{N}} \subset (0,1)$ and $(v_n)_{n\in \mathbb{N}}\subset X$  with $\lim_{n\to \infty} t_n =0$ and  $\lim_{n\to \infty} v_n =v$. Then for all $n\in \mathbb{N}$,
\[
\frac{f(g(x+t_nv_n))-f(g(x))}{t_n}=\frac{f\left(g(x)+t_n\frac{g(x+t_nv_n)-g(x)}{t_n}\right)-f(g(x))}{t_n}\,.
\]
The differentiability of $g$ implies that $\lim_{n\to \infty}\frac{g(x+t_nv_n)-g(x)}{t_n}=\mathfrak{d}g(x)v$ in $Y$, which along with the differentiability of $f$  implies that $\lim_{n\to\infty}\frac{f(g(x+t_nv_n))-f(g(x))}{t_n}=\mathfrak{d}f(g(x))\mathfrak{d}g(x)v$ in $Z$. Since $\mathfrak{d}g(x)\in \clL(X,Y)$ and $\mathfrak{d}f(g(x))\in \clL(Y,Z)$,  $\mathfrak{d}f(g(x))\mathfrak{d}g(x)\in  \clL(X,Z)$. This proves the differentiability of $f\circ g$.
\end{proof}

In the sequel, we establish the differentiability of several important functions. The following proposition proves the differentiability of the policy operator $\boldsymbol{\pi}$ defined by \eqref{eq:pi_f_mu}.

\begin{proposition}\label{prop:differentiability_Pi}
The map $\boldsymbol{\pi} :B_b(S\times A) \to \clP_{\mu}(A|S)$ defined by \eqref{eq:pi_f_mu} is differentiable and for all  $f,g\in B_b(S\times A)$,
\begin{equation}\label{eq:Pi_gateaux}
(\mathfrak{d}\boldsymbol{\pi}(f)g)(da|s)=\left(g(s,a)-\int_A g(s,a')\boldsymbol{\pi}(f)(da'|s) \right)\boldsymbol{\pi}(f)(da|s)\,.    
\end{equation}
Moreover, for all $f\in B_b(S\times A)$, $\|\mathfrak{d}\boldsymbol{\pi}(f)\|_{\clL(B_b(S\times A), b\clM(A|S))}\le 2$.
\end{proposition}
\begin{proof}
Fix $f,g \in B_b(S\times A)$ and sequences  $(t_n)_{n\in \mathbb{N}} \subset (0,1)$ and $(g_n)_{n\in \mathbb{N}}\subset B_b( S\times A)$ such that  $\lim_{n\to \infty} t_n =0$ and  $\lim_{n\to \infty} g_n =g$. Define $\pi_\infty=\boldsymbol{\pi}(f)$ and $\pi_{n}=\boldsymbol{\pi}(f+t_ng_n)$ for all $n\in \mathbb {N}$. For each $s\in S$, let $\Psi(da|s)$ be the right hand side of \eqref{eq:Pi_gateaux}. Recall that for any $\nu\in \clM(A)$, if  $\nu=h d\mu$ for some $h\in L^1(A,\mu)$, then  $\|\nu\|_{\mathcal{M}(A)}=\int_A|h(a)|\mu (da)$.  As  $\pi_{n}(da|s)$ is absolutely continuous with respect to $\mu$, for all $n\in \mathbb N$,
\begin{align}
\label{eq:differnce_quotation_pi}
\begin{split}
& \left\|\frac{\pi_n-\pi_\infty}{t_n}-\Psi \right\|_{b\clM(A|S)}= \sup_{s\in S}
\int_A \left|\frac{\frac{\mathrm{d} \pi_n}{\mathrm{d} \mu}(a|s)-\frac{\mathrm{d}\pi_\infty}{\mathrm{d} \mu}(a|s)}{t_n}-\frac{\mathrm{d} \Psi }{\mathrm{d} \mu}(a|s)\right| \mu(da)\,.
\end{split}
\end{align}
To simplify the notation, we define the unnormalized kernels $\tilde{\pi}_n(da|s) = e^{f(s,a)+t_n g_n(s,a)} \mu(d a)$ and $\tilde{\pi}_\infty(da|s) = e^{f(s,a)} \mu(d a)$.  For all $(s,a)\in S\times A$,
\begin{align}\label{eq:dPi_t-dPi_0}
\begin{split}
&\frac{\mathrm{d}\pi_n}{\mathrm{d} \mu}(a|s)-\frac{\mathrm{d} \pi_\infty}{\mathrm{d} \mu}(a|s)
=\frac{\frac{\mathrm{d} \tilde{\pi}_n}{\mathrm{d} \mu}(a|s)}{\tilde{\pi}_n(A|s)}-\frac{\frac{\mathrm{d} \tilde{\pi}_\infty}{\mathrm{d} \mu}(a|s)}{\tilde{\pi}_\infty(A|s)}\\
&\quad =\frac{\frac{\mathrm{d} \tilde{\pi}_n}{\mathrm{d} \mu}(a|s)-\frac{\mathrm{d} \tilde{\pi}_\infty}{\mathrm{d} \mu}(a|s)}{\tilde{\pi}_n(A|s)}+\frac{\mathrm{d} \tilde{\pi}_\infty}{\mathrm{d} \mu}(a|s)\frac{\tilde{\pi}_\infty(A|s)-\tilde{\pi}_n(A|s)} {\tilde{\pi}_n(A|s)\tilde{\pi}_\infty(A|s)}\\
&\quad =\left(\frac{\frac{\mathrm{d} \tilde{\pi}_n}{\mathrm{d} \mu}(a|s)-\frac{\mathrm{d} \tilde{\pi}_\infty}{\mathrm{d} \mu}(a|s)}{\tilde{\pi}_\infty(A|s)}+
\frac{\mathrm{d} {\pi}_\infty}{\mathrm{d} \mu}(a|s)\frac{\tilde{\pi}_\infty(A|s)-\tilde{\pi}_n(A|s)} {\tilde{\pi}_\infty(A|s)}\right)\frac{\tilde{\pi}_\infty(A|s)}{\tilde{\pi}_n(A|s)}\\
&\quad =\left[\frac{\frac{\mathrm{d} \tilde{\pi}_n}{\mathrm{d} \mu}(a|s)-\frac{\mathrm{d} \tilde{\pi}_\infty}{\mathrm{d} \mu}(a|s)}{\tilde{\pi}_\infty(A|s)}-\frac{\mathrm{d} {\pi}_\infty}{\mathrm{d} \mu}(a|s)\int_A\left(\frac{\frac{\mathrm{d} \tilde{\pi}_n}{\mathrm{d} \mu}(a|s)-\frac{\mathrm{d} \tilde{\pi}_\infty}{\mathrm{d} \mu}(a|s)}{\tilde{\pi}_\infty(A|s)}\right)\mu(da)\right]\frac{\tilde{\pi}_\infty(A|s)}{\tilde{\pi}_n(A|s)}\,.
\end{split}
\end{align}
Write $\Delta \pi_n (a|s)=t_n^{-1}\frac{\frac{\mathrm{d} \tilde{\pi}_n}{\mathrm{d} \mu}(a|s)-\frac{\mathrm{d} \tilde{\pi}_\infty}{\mathrm{d} \mu}(a|s)}{\tilde{\pi}_\infty(A|s)}$.
By \eqref{eq:differnce_quotation_pi} and \eqref{eq:dPi_t-dPi_0},
\[
\left\|\frac{\pi_n-\pi_\infty}{t_n}-\Psi \right\|_{b\clM(A|S)}=  \sup_{s\in S}\int_A \left|
\frac{\frac{\mathrm{d}\pi_n}{\mathrm{d} \mu}(a|s)-\frac{\mathrm{d} \pi_\infty}{\mathrm{d} \mu}(a|s)}{t_n}-\frac{\mathrm{d}\Psi}{\mathrm{d} \mu}(a|s)\right| \mu(da)\le I_n^1 +I_n^2\,,
\]
where the terms $I_n^1$ and $I_n^2$ are defined as follows:
\begin{align*}
I_n^1&=\sup_{s\in S}\int_A \left|\left[\Delta \pi_n (a|s)-\frac{\mathrm{d} {\pi}_\infty}{\mathrm{d}\mu}(a|s)\int_A\Delta \pi_n (a|s)\mu(da)-\frac{\mathrm{d} \Psi}{\mathrm{d} \mu}(a|s)\right]\frac{\tilde{\pi}_\infty(A|s)}{\tilde{\pi}_n(A|s)}\right| \mu(da),\\
I^2_n&=\sup_{s\in S}\int_A \left|\frac{\mathrm{d} \Psi}{\mathrm{d} \mu}(a|s)\left[\frac{\tilde{\pi}_\infty(A|s)}{\tilde{\pi}_n(A|s)}-1\right]\right| \mu(da)\,.
\end{align*}
We now prove $\lim_{n\to \infty} I^1_n=\lim_{n\to \infty} I^2_n=0$. For each $n\in \mathbb{N}$, the definition of $\tilde{\pi}_n$ implies for all $(s,a)\in S\times A$,
\[
\frac{\mathrm{d} \tilde{\pi}_n}{\mathrm{d} \mu}(a|s)-\frac{\mathrm{d} \tilde{\pi}_\infty}{\mathrm{d} \mu}(a|s)
=e^{f(s,a)}\left(e^{t_ng_n(s,a)} -1\right),
\]
which, along with the Taylor expansion of the exponential function at $0$ yields
\begin{equation}
\label{eq:delta_pi_n}
\Delta \pi_n(a|s)=  t_n^{-1}\frac{\frac{\mathrm{d} \tilde{\pi}_n}{\mathrm{d} \mu}(a|s)-\frac{\mathrm{d} \tilde{\pi}_\infty}{\mathrm{d} \mu}(a|s)}{\tilde{\pi}_\infty(A|s)}
=\frac{\mathrm{d} {\pi}_\infty}{\mathrm{d} \mu}(a|s)\left(g_n(s,a)+t_n\sum_{k=2}^\infty t_n^{k-2}\frac{\left(g_n(s,a)\right)^k}{k!}\right).
\end{equation}
Therefore, by the definition of $\Psi$, for all $(s,a)\in S\times A$ and $n\in \mathbb{N}$,
\begin{align*}
& \left| \Delta \pi_n (a|s)-\frac{\mathrm{d} {\pi}_\infty}{\mathrm{d} \mu}(a|s)\int_A
\Delta \pi_n (a|s)\mu(da)-\frac{\mathrm{d}\Psi}{\mathrm{d} \mu}(a|s)\right|\\
&=\left|t_n\left(\sum_{k=2}^\infty t_n^{k-2}\frac{\left(g_n(s,a)\right)^k}{k!}- \frac{\mathrm{d} {\pi}_\infty}{\mathrm{d} \mu}(a|s)\int_A \sum_{k=2}^\infty t_n^{k-2}\frac{\left(g_n(s,a)\right)^k}{k!}\pi_\infty(da|s)\right)\right|\\
&\quad +\left|\frac{\mathrm{d} {\pi}_\infty}{\mathrm{d} \mu}(a|s) (g_n(s,a)-g(s,a))- \frac{\mathrm{d} {\pi}_\infty}{\mathrm{d} \mu}(a|s) \int_A (g_n(s,a')-g(s,a'))\pi_\infty(da'|s)\right|\\
& \le t_n\exp\left(\sup_{n\in \mathbb{N}}\|g_n\|_{B_b(S\times A)}\right)\left(1 +\frac{\mathrm{d} {\pi}_\infty}{\mathrm{d} \mu}(a|s)\right)+2\|g_n-g\|_{B_b(S\times A)}\frac{\mathrm{d} {\pi}_\infty}{\mathrm{d} \mu}(a|s),
\end{align*}
where the last inequality used $\pi_\infty(A|s)=1$. Substituting this estimate in the definition of $I^1_n$ yields for all $n\in \mathbb{N}$,
\begin{align*}
\begin{split}
I^1_n&\le  \sup_{s\in S}\int_A \left| \Delta \pi_n (a|s)-\frac{\mathrm{d} {\pi}_\infty}{\mathrm{d} \mu}(a|s)\int_A\Delta \pi_n (a|s)\mu(da)-\frac{\mathrm{d} \Psi}{\mathrm{d} \mu}(a|s)\right| \mu(da) \sup_{s\in S}\frac{\tilde{\pi}_\infty(A|s)}{\tilde{\pi}_n(A|s)}\\
&\le\sup_{s\in S}\int_A \left|t_n\exp\left(\sup_{n\in \mathbb{N}}\|g_n\|_{B_b(S\times A)}\right)\left(1+\frac{\mathrm{d} {\pi}_\infty}{\mathrm{d} \mu}(a|s)\right)\right| \mu(da)
\sup_{s\in S}\frac{\tilde{\pi}_\infty(A|s)}{\tilde{\pi}_n(A|s)}\\
&\quad + \sup_{s\in S}\int_A \left|2\|g_n-g\|_{B_b(S\times A)}\frac{\mathrm{d} {\pi}_\infty}{\mathrm{d} \mu}(a|s)\right| \mu(da)\sup_{s\in S}\frac{\tilde{\pi}_\infty(A|s)}{\tilde{\pi}_n(A|s)}\\
&\le2\left[t_n\exp\left(\sup_{n\in \mathbb{N}}\|g_n\|_{B_b(S\times A)}\right) +\|g_n-g\|_{B_b(S\times A)}\right]\sup_{s\in S}\frac{\tilde{\pi}_\infty(A|s)}{\tilde{\pi}_n(A|s)}\,.
\end{split}
\end{align*}
Since $\lim_{n\to \mathbb{N}}\|g_n- g\|_{B_b(S\times A)}=0$,
$\sup_{s\in S}\frac{\tilde{\pi}_\infty(A|s)}{\tilde{\pi}_n(A|s)}$ is  uniformly bounded for all $n\in \mathbb{N}$,
which along with $\lim_{n\to \mathbb{N}}t_n=0$ yields 
$\lim_{n\to \infty} I^1_n=0$.
To prove the convergence of $I^2_n$,
note that 
\begin{align*}
I^2_n&\le \sup_{s\in S}\int_A \left|\frac{\mathrm{d} \Psi}{\mathrm{d} \mu}(a|s)\right| \mu(da)
\sup_{s\in S}\left|\frac{\tilde{\pi}_\infty(A|s)}{\tilde{\pi}_n(A|s)}-1\right|\\
& \le \sup_{s\in S}\int_A \left|\frac{\mathrm{d} \Psi}{\mathrm{d} \mu}(a|s)\right| \mu(da)
\frac{1}{\inf_{s\in S}\tilde{\pi}_n(A|s)}\sup_{s\in S}\left|\tilde{\pi}_\infty(A|s)- \tilde{\pi}_n(A|s)\right|.
\end{align*}
The uniform boundedness of $f$ and $(g_n)_{n\in \mathbb{N}}$ implies that 
\[
\sup_{n\in \mathbb{N}}\left(\sup_{s\in S}\int_A \left|\frac{\mathrm{d}\Psi}{\mathrm{d} \mu}(a|s)\right|\mu(da)\frac{1}{\inf_{s\in S}\tilde{\pi}_n(A|s)}\right)<\infty\,,
\]
and the Taylor expansion of the exponential function at $0$ implies that 
\begin{align*}
\lim_{n\to \infty}  \sup_{s\in S}\left|\tilde{\pi}_\infty(A|s)- \tilde{\pi}_n(A|s)\right|\le \lim_{n\to \infty}\sup_{s\in S}\left|\int_A e^{f(s,a)}(e^{t_n g_n(s,a)}-1)\mu (da)\right|=0\,.
\end{align*}
This yields that $\lim_{n\to \infty}I^2_n=0$ and  proves the  differentiability of $\boldsymbol{\pi}:B_b(S\times A)\to b\clM_b(A|S)$. Finally, for all  $f,g\in B_b(S\times A)$,
\begin{align*}
\|(\mathfrak{d} \boldsymbol{\pi}(f)g)\|_{b\clM(A|S)}&=\sup_{s\in S}\int_A \left|g(s,a)\frac{\mathrm{d} \boldsymbol{\pi}(f)}{\mathrm{d} \mu}(a|s)-\frac{\mathrm{d}\boldsymbol{\pi}(f)}{\mathrm{d} \mu}(a|s)\int_A g(s,a')\boldsymbol{\pi}(f)(da'|s)\right|\mu(da)\\
&\le \|g\|_{B_b(S\times A) }\sup_{s\in S}\int_A \left(\frac{\mathrm{d}\boldsymbol{\pi}(f)}{\mathrm{d} \mu}(a|s)+\frac{\mathrm{d}\boldsymbol{\pi}(f)}{\mathrm{d} \mu}(a|s)\int_A \boldsymbol{\pi}(f)(da'|s)\right)\mu(da)\\
&=2\|g\|_{B_b(S\times A) }\,,
\end{align*}
where the last inequality used the fact that  $\boldsymbol{\pi}(f)(da|s)\in \clP(A)$ for all $s\in S$. This proves $\|\mathfrak{d}\boldsymbol{\pi}(f)\|_{\clL(B_b(S\times A), b\clM(A|S))}\le 2$.
\end{proof}

The next proposition proves that the logarithm of the Radon-Nikodym derivative of a policy is differentiable.  

\begin{proposition}
\label{prop:differentiablity_log_density}
Let $\boldsymbol{\pi}:B_b(S\times A) \to \clP_{\mu}(A|S)$ be defined by \eqref{eq:pi_f_mu}. Then the map $\ln \frac{\mathrm{d} \boldsymbol \pi(\cdot)}{\mathrm{d} \mu}:B_b(S\times A) \to B_b(S\times A)$ is differentiable and for all $f,g\in B_b(S\times A)$,
\begin{equation}\label{eq:d_ln_dpi_dmu}
\left(\mathfrak{d}\ln\frac{\mathrm{d}\boldsymbol{\pi}(f)}{\mathrm{d} \mu} g\right)(s,a) = g(s,a)-\int_A g(s,a') {\boldsymbol{\pi}(f)} (da'|s)\,.
\end{equation}
Moreover, for all $f\in B_b(S\times A)$, 
$\|\mathfrak{d}\ln\frac{\mathrm{d} \boldsymbol{\pi}(f)}{\mathrm{d} \mu}\|_{\clL(B_b(S\times A),B_b(S\times A))}
\le 2$.
\end{proposition}
\begin{proof}
Fix  $f,g\in B_b(S\times A)$ and sequences  $(t_n)_{n\in \mathbb{N}} \subset (0,1)$ and $(g_n)_{n\in \mathbb{N}}\subset B_b(S\times A)$  such that  $\lim_{n\to \infty} t_n =0$ and $\lim_{n\to \infty} g_n =g$.
Define $\pi_\infty=\boldsymbol{\pi}(f)$, $\tilde{\pi}_\infty(da|s) = e^{f(s,a)} \mu(d a)$, and 
for each $n\in \mathbb{N}$, define $\pi_n=\boldsymbol{\pi}(f+t_ng_n)$ and $\tilde{\pi}_n(da|s) = e^{f(s,a)+t_n g_n(s,a)} \mu(da)$.  For each $(s,a)\in S\times A$, let $\Psi(s,a)$ be the right-hand-side of \eqref{eq:d_ln_dpi_dmu}. By \eqref{eq:dPi_t-dPi_0} and \eqref{eq:delta_pi_n}, for all $(s,a)\in S\times A$,
\begin{align}
\label{eq:dpi_dmu_dt}
\begin{split}
&\frac{1}{t_n}\left(\frac{\mathrm{d}\pi_n}{\mathrm{d} \mu}(a|s)-\frac{\mathrm{d}\pi_\infty}{\mathrm{d} \mu}(a|s)\right)\\
&\quad =\left[\frac{\frac{\mathrm{d}\tilde{\pi}_n}{\mathrm{d} \mu}(a|s)-\frac{\mathrm{d}\tilde{\pi}_\infty}{\mathrm{d} \mu}(a|s)}{t_n\tilde{\pi}_\infty(A|s)}-\frac{\mathrm{d}{\pi}_\infty}{\mathrm{d} \mu}(a|s)\int_A
\left(\frac{\frac{\mathrm{d}\tilde{\pi}_n}{\mathrm{d} \mu}(a|s)-\frac{\mathrm{d}\tilde{\pi}_\infty}{\mathrm{d} \mu}(a|s)} {t_n\tilde{\pi}_\infty(A|s)}
\right)\mu(da)\right]\frac{\tilde{\pi}_\infty(A|s)}{\tilde{\pi}_n(A|s)}\\
&\quad =\frac{\mathrm{d} 
{\pi}_\infty}{\mathrm{d} \mu}(a|s)\bigg\{\left(g_n(s,a)+t_n\sum_{k=2}^\infty t_n^{k-2}\frac{\left(g_n(s,a)\right)^k}{k!}\right)\\
&\qquad -\int_A\frac{\mathrm{d} {\pi}_\infty}{\mathrm{d} \mu}(a|s)\left(g_n(s,a)+t_n\sum_{k=2}^\infty t_n^{k-2}\frac{\left(g_n(s,a)\right)^k}{k!}\right) \mu(da)\bigg\}\frac{\tilde{\pi}_\infty(A|s)}{\tilde{\pi}_n(A|s)}\,.
\end{split}
\end{align}
which along with $t_n\in (0,1]$, $\sup_{n\in \mathbb{N}}\|g_n\|_{B_b(S\times A)}<\infty$ and $\sup_{n\in \mathbb{N}}\frac{\tilde{\pi}_\infty(A|s)}{\tilde{\pi}_n(A|s)}<\infty$ implies that there exists a constant $C>0$ such that for all $n\in \bbN$,
\begin{equation}\label{eq:quotation_RN_derivative}
\left\| \frac{ \frac{\mathrm{d} \pi_n}{\mathrm{d} \mu} -\frac{\mathrm{d} \pi_\infty}{\mathrm{d} \mu}}{\frac{\mathrm{d} {\pi}_\infty}{\mathrm{d} \mu}  }\right\|_{B_b(S\times A)}\le Ct_n\,. 
\end{equation} 
As $\lim_{x\to 0}\frac{\ln(1+x)-x}{x^2}=-\frac{1}{2}$, there exists $N\in \mathbb{N}$ such that for all $n\ge N$ and $(s,a)\in S\times A$,  
\begin{align*}
&\left|  \ln \frac{\mathrm{d} \pi_n}{\mathrm{d} \mu}(a|s)-\ln\frac{\mathrm{d} \pi_\infty}{\mathrm{d} \mu}(a|s) -\frac{\frac{\mathrm{d} \pi_n}{\mathrm{d} \mu}(a|s)- \frac{\mathrm{d}\pi_\infty}{\mathrm{d} \mu}(a|s)}{\frac{\mathrm{d}\pi_\infty}{\mathrm{d} \mu}(a|s)}\right|\\
&\quad =\left|\ln \left(1+\frac{\frac{\mathrm{d} \pi_n}{\mathrm{d}  \mu}(a|s)- \frac{\mathrm{d} \pi_\infty}{\mathrm{d}  \mu}(a|s)}{\frac{\mathrm{d} \pi_\infty}{\mathrm{d}  \mu}(a|s)}\right) 
-\frac{\frac{\mathrm{d} \pi_n}{\mathrm{d}  \mu}(a|s)- \frac{\mathrm{d} \pi_\infty}{\mathrm{d}  \mu}(a|s)}{\frac{\mathrm{d} \pi_\infty}{\mathrm{d}  \mu}(a|s)}\right|\le \left|\frac{\frac{\mathrm{d} \pi_n}{\mathrm{d}  \mu}(a|s)- \frac{\mathrm{d} \pi_\infty}{\mathrm{d}  \mu}(a|s)}{\frac{\mathrm{d} \pi_\infty}{\mathrm{d}  \mu}(a|s)}\right|^2.
\end{align*}
Multiplying both sides by $1/t_n$ and applying the triangle inequality, for all $(s,a)\in S\times A$,  
\begin{align*}
&\left| \frac{1}{t_n}\left( \ln \frac{\mathrm{d} \pi_n}{\mathrm{d}  \mu}(a|s)-\ln\frac{\mathrm{d} \pi_\infty}{\mathrm{d}  \mu}(a|s) \right)-\Psi(s,a) \right|\\
&\quad \le \frac{1}{t_n}\left|\frac{\frac{\mathrm{d} \pi_n}{\mathrm{d}  \mu}(a|s)- \frac{\mathrm{d} \pi_\infty}{\mathrm{d}  \mu}(a|s)}{\frac{\mathrm{d} \pi_\infty}{\mathrm{d}  \mu}(a|s)}\right|^2+\left|\Psi(s,a) 
-\frac{1}{t_n}\frac{\frac{\mathrm{d} \pi_n}{\mathrm{d}  \mu}(a|s)- \frac{\mathrm{d} \pi_\infty}{\mathrm{d} \mu}(a|s)}{\frac{\mathrm{d} \pi_\infty}{\mathrm{d}  \mu}(a|s)} \right|.
\end{align*}
As $n\to \infty$, the first term converges to $0$ in the norm $\|\cdot\|_{B_b(S\times A)}$ due to \eqref{eq:quotation_RN_derivative}. By \eqref{eq:dpi_dmu_dt}, the second term can be upper bounded by
\begin{align*}
& \left|
\Psi(s,a)  -\frac{1}{t_n}\frac{\frac{\mathrm{d} \pi_n}{\mathrm{d}  \mu}(a|s)- \frac{\mathrm{d} \pi_\infty}{\mathrm{d}  \mu}(a|s)}{\frac{\mathrm{d} \pi_\infty}{\mathrm{d}  \mu}(a|s)} \right|\\
&\quad \le \left|\left(g(s,a)-\int_A g(s,a') {\pi_\infty} (da'|s)\right)-\left(g_n(s,a')-\int_A g_n(s,a') {\pi_\infty} (da'|s)\right)\right|\\
&\qquad +\left|\left(g_n(s,a)-\int_A g_n(s,a') {\pi_\infty} (da'|s)\right) \left(1-\frac{\tilde{\pi}_\infty(A|s)}{\tilde{\pi}_n(A|s)}\right)\right|\\
&\qquad +t_n \left|\left(\sum_{k=2}^\infty t_n^{k-2}\frac{\left(g_n(s,a)\right)^k}{k!}-\int_A
\sum_{k=2}^\infty t_n^{k-2}\frac{\left(g_n(s,a')\right)^k}{k!}{\pi}_\infty(da')\right)\frac{\tilde{\pi}_\infty(A|s)}{\tilde{\pi}_n(A|s)} \right|,
\end{align*}
which converges to zero in the norm $\|\cdot\|_{B_b(S\times A)}$ due to  $\lim_{n\to\infty}\|g_n-g\|_{B_b(S\times A)}=0$ and $\lim_{n\to \infty}t_n=0$. This proves the desired differentiability. Note that for all $f, g\in B_b(S\times A)$ and $(s,a)\in S\times A$,
\[
\left|\left(\mathfrak{d}\ln\frac{\mathrm{d} \boldsymbol{\pi}(f)}{\mathrm{d}  \mu} g\right)(s,a) \right|=
\left| g(s,a)-\int_A g(s,a') {\boldsymbol{\pi}(f)} (da'|s)\right|\le 2\|g\|_{B_b(S\times A)}\,,
\]
which shows that $\|\mathfrak{d}\ln\frac{\mathrm{d} \boldsymbol{\pi}(f)}{\mathrm{d}  \mu}\|_{\clL(B_b(S\times A),B_b(S\times A))}\le 2$.
\end{proof}

The next proposition proves the differentiability of the value function. 
\begin{proposition}\label{prop:differentiability_dV_df}
Let 
$\boldsymbol{\pi}:B_b(S\times A) \to \clP_{\mu}(A|S)$ be defined by \eqref{eq:pi_f_mu}, and for each $s\in S$, let $V^\cdot_\tau (s): \clP(A|S)\to {\mathbb{R}}\cup\{\infty\}$ be defined by \eqref{eq:V_pi_tau}. Then the map $B_b(S\times A)\ni f\mapsto J(f)\coloneqq V^{\boldsymbol{\pi}(f)}_\tau\in   B_b(S) $ is differentiable and for all $f,g\in B_b(S\times A)$ and $s\in S$,
\begin{equation}\label{eq:d_J_fg}
\left(\mathfrak{d}  J(f)g\right)(s)=\frac{1}{1-\gamma} \int_S \int_A \left(Q^{\boldsymbol{\pi}(f)}_\tau (s',a) +\tau \ln\frac{\mathrm{d}  \boldsymbol{\pi}(f)}{\mathrm{d}  \mu}(a|s')\right)(\mathfrak{d}\boldsymbol{\pi}(f)g)(da|s')d^{\boldsymbol{\pi}(f)}(ds'|s)\,,
\end{equation}
where $\mathfrak{d}\boldsymbol{\pi}(f)g$ is defined as in Proposition \ref{prop:differentiability_Pi}.
Moreover, for all $f\in B_b(S\times A)$,\linebreak $\|\mathfrak{d}  J(f)\|_{\clL(B_b(S\times A),B_b(S))}\le \frac{2}{(1-\gamma)^2}\left(\|c\|_{B_b(S\times A)}+2\tau  \|f\|_{B_b(S\times A)}\right)$.
\end{proposition}
\begin{proof}
Fix $f,g\in B_b(S\times A)$,  and sequences  $(t_n)_{n\in \mathbb{N}} \subset (0,1)$ and $(g_n)_{n\in \mathbb{N}}\subset B_b(S\times A)$ such that  
$\lim_{n\to \infty} t_n =0$ and $\lim_{n\to \infty} g_n =g$. Define $\pi_\infty=\boldsymbol{\pi}(f)$, and  for each $n\in \mathbb{N}$, define  $\pi_n=\boldsymbol{\pi}(f+t_ng_n)$. For each $s\in S$, let  $\Psi(s)$ be the right-hand side of \eqref{eq:d_J_fg}.
By \eqref{def:Q_fn} and  \eqref{eq:on_policy},
\begin{align*}
& \frac{J(f+t_n g_n)(s)-J(f)(s)}{t_n}\\
& =\frac{1}{t_n}\bigg[\int_A \left( Q^{\pi_n}_{\tau}(s,a)+ \tau \ln \frac{\mathrm{d} \pi_n}{\mathrm{d} \mu}(s,a) \right) \pi_n(da| s)-\int_A \left( Q^{\pi_\infty}_{\tau}(s,a)+ \tau \ln \frac{\mathrm{d}\pi_\infty}{\mathrm{d} \mu}(s,a) \right) \pi_\infty(da|s)\bigg]\\
&=\int_A   \left( Q^{\pi_n}_{\tau}(s,a)+ \tau \ln \frac{\mathrm{d} \pi_n}{\mathrm{d} \mu}(s,a) \right) 
\frac{\pi_n(da | s)-\pi_\infty(da| s)}{t_n}\\
&\quad +\int_A   \frac{Q^{\pi_n}_{\tau}(s,a)- Q^{\pi_\infty}_{\tau}(s,a)}{t_n}\pi_\infty(da| s)
+\tau \int_A \frac{1}{t_n}\left(\ln \frac{\mathrm{d} \pi_n}{\mathrm{d} \mu}(s,a) -\ln \frac{\mathrm{d} \pi_\infty}{\mathrm{d} \mu}(s,a)  \right)\pi_\infty(da| s)\\
&=I^1_n(s)+I^2_n(s)+\gamma \int_A\int_S \frac{V^{\pi_n}_{\tau}(s')-V^{\pi_\infty}_{\tau}(s')}{t_n} P(ds'|s,a)\pi_\infty(da|s)\\
&=I^1_n(s)+I^2_n(s)+\gamma \int_A\int_S \frac{J(f+t_n g_n)(s')-J(f)(s')}{t_n}  P(ds'|s,a) \pi_\infty(da|s)\,,
\end{align*}
where $I^i_n(s)$, $i=1,2$, are defined as follows:
\begin{align*}
I^1_n(s) &= \int_A \left( Q^{\pi_n}_{\tau}(s,a)+ \tau \ln \frac{\mathrm{d} \pi_n}{\mathrm{d} \mu}(s,a) \right)\frac{\pi_n(da | s)-\pi_\infty(da| s)}{t_n}\,,\\
I^2_n(s) &=\tau \int_A\frac{1}{t_n}\left(\ln \frac{\mathrm{d} \pi_n}{\mathrm{d} \mu}(s,a)-\ln\frac{\mathrm{d} \pi_\infty}{\mathrm{d} \mu}(s,a)\right)\pi_\infty(da| s)\,.
\end{align*}
By Fubini's theorem and Lemma \ref{lem:inverse_1-gammaP}, for all $s\in S$,
\begin{align*}
\frac{J(f+t_n g_n)(s)-J(f)(s)}{t_n} &= \frac{1}{1-\gamma}\int_S \left(I_n^1(s')+I_n^2(s')\right)
d^{\pi_\infty} (ds'|s)\,,
\end{align*}
which implies that 
\begin{align*}
&\frac{J(f+t_ng_n) (s)-J(f)(s)}{t_n}-\Psi(s)\\
&\quad =\frac{1}{1-\gamma}\int_S \left[ I_n^1(s')+I_n^2(s')-\int_A\left(Q^{\pi_\infty}_\tau (s',a) +\tau \ln\frac{\mathrm{d}   \pi_\infty}{\mathrm{d}  \mu}(a|s')\right)
(\mathfrak{d}  \boldsymbol{\pi}(f) g)(da|s')\right]d^{\pi_\infty}  (ds'|s)\\
&\quad = \frac{1}{1-\gamma} (I^3_n(s) +I^4_n(s)+I^5_n(s))\,,
\end{align*}
where $I^i_n(s)$, $i=3,4,5$, are defined as follows:
\begin{align*}
I^3_n(s)&=\int_S\int_A\bigg[\left(Q^{\pi_n}_\tau (s',a) +\tau \ln\frac{\mathrm{d}\pi_n}{\mathrm{d}  \mu}(a|s')\right)\\
&\qquad -\left(Q^{\pi_\infty}_\tau (s',a) +\tau \ln\frac{\mathrm{d}\pi_\infty}{\mathrm{d}\mu}(a|s')\right)\bigg] \frac{\pi_n(da | s')-\pi_\infty(da| s')}{t_n} d^{\pi_\infty}  (ds'|s)\,,\\
I^4_n(s)&=\tau \int_S \int_A \frac{1 }{t_n}\left(\ln \frac{\mathrm{d} \pi_n}{\mathrm{d}\mu}(s',a)-\ln \frac{\mathrm{d} \pi_\infty}{\mathrm{d} \mu}(s',a)  \right)\pi_\infty(da| s')d^{\pi_\infty}  (ds'|s)\,,\\
I^5_n(s)& =\int_S \int_A \left(Q^{\pi_\infty}_\tau (s',a) +\tau \ln\frac{\mathrm{d}   \pi_\infty}{\mathrm{d}  \mu}(a|s')\right)\\
&\qquad \qquad \qquad \times \left(\frac{\pi_n(da | s')-\pi_\infty(da| s')}{t_n}-(\mathfrak{d} \boldsymbol{\pi}(f) g)(da|s')\right)d^{\pi_\infty}  (ds'|s)\,.
\end{align*}
It remains to prove $I_n^i$, $i=3,4,5$, tends to zero in the norm $\|\cdot\|_{B_b(S)}$ as $n\to \infty$.
Observe that by Proposition \ref{prop:differentiability_Pi}, $\boldsymbol{\pi}:B_b(S\times A) \to \clP_{\mu}(A|S)$ is differentiable,  and hence 
\[
\lim_{n\to\infty}\left\|\frac{\pi_n -\pi_\infty }{t_n}- \mathfrak{d}  \boldsymbol{\pi}(f) g  \right\|_{b\clM(A|S)}=0.
\]
Then as $\|d^{\pi_\infty}(\cdot|s)\|_{ \clM(S)}=1$ for all $s\in S$,
\[
\lim_{n\to\infty}\|I_n^5\|_{B_b(S)}\le  \left\|Q^{\pi_\infty}_\tau   +\tau \ln\frac{\mathrm{d}   \pi_\infty}{\mathrm{d}  \mu}\right\|_{B_b(S\times A)}\lim_{n\to\infty} \left\|\frac{\pi_n -\pi_\infty }{t_n}-(\mathfrak{d}  \boldsymbol{\pi}(f) g) \right\|_{b\clM(A|S)}=0\,.
\]
Since $\int_A \left(\mathfrak{d}\ln\frac{\mathrm{d} \pi_\infty}{\mathrm{d}  \mu} g\right)(s,a) {\pi_\infty} (da|s)=\int_A\left(g(s,a)-\int_A g(s,a) {\pi_\infty} (da|s)\right){\pi_\infty} (da|s)=0
$,
\begin{align*}
\lim_{n\to\infty}\|I_n^4\|_{B_b(S)}&
\le \lim_{n\to\infty}\tau \left\|
\int_A \left[\frac{\ln \frac{\mathrm{d} \pi_n}{\mathrm{d} \mu}(\cdot,a)-\ln \frac{\mathrm{d} \pi_\infty}{\mathrm{d} \mu}(\cdot,a)  }{t_n}
-\left(\mathfrak{d}\ln\frac{\mathrm{d} \pi_\infty}{\mathrm{d}  \mu} g\right)(\cdot,a) \right]\pi_\infty(da| \cdot) \right\|_{B_b(S)}\\
&\le \tau \lim_{n\to\infty} \left\|\frac{1}{t_n}\left(\ln \frac{\mathrm{d} \pi_n}{\mathrm{d}\mu}-\ln \frac{\mathrm{d} \pi_\infty}{\mathrm{d}\mu} \right)- \mathfrak{d}\ln\frac{\mathrm{d} \pi_\infty}{\mathrm{d}  \mu} g  \right\|_{B_b(S\times A)}=0\,,
\end{align*}
where the last identity follows from   Lemma \ref{prop:differentiablity_log_density}. Finally, by the triangle inequality, for all $n\in \mathbb{N}$,
\begin{align*}
\|I_n^3\|_{B_b(S)}\le    
\left(\left\|
Q^{\pi_n}_\tau  -Q^{\pi_\infty}_\tau  
\right\|_{B_b(S\times A)}
+\tau 
\left\|
\ln\frac{\mathrm{d}   \pi_n}{\mathrm{d}  \mu} 
- \ln\frac{\mathrm{d}   \pi_\infty}{\mathrm{d}  \mu} 
\right\|_{B_b(S\times A)}\right)
\left\|\frac{\pi_n -\pi_\infty }{t_n}
\right\|_{b\clM(A|S)}.
\end{align*}
By Propositions~\ref{prop:differentiability_Pi} and~\ref{prop:differentiablity_log_density},
\[
\sup_{n\in\mathbb{N}}\left\|\frac{\pi_n -\pi_\infty }{t_n}
\right\|_{b\clM(A|S)}<\infty,
\quad 
\lim_{n\to\infty} \left\|
\ln\frac{\mathrm{d}   \pi_n}{\mathrm{d}  \mu} 
- \ln\frac{\mathrm{d}   \pi_\infty}{\mathrm{d}  \mu} 
\right\|_{B_b(S\times A)}=0\,.
\]  
Moreover, by Proposition \ref{prop:Q_continuity},
\begin{align*}
\|Q^{\pi_n}_\tau-Q^{\pi_\infty}_\tau\|_{B_b(S\times A)} 
&\le \frac{\gamma}{(1-\gamma)^2}
\left(
\|c\|_{B_b(S\times A)}
+2\tau  \|f \|_{B_b(S\times A)}
\right)
\| \pi_n-\pi_\infty\|_{b\clM(A|S)} 
\\
&\quad 
+\frac{\tau \gamma}{1-\gamma}
\left\|\ln \frac{\mathrm{d} \pi_n}{\mathrm{d} \mu}-\ln \frac{\mathrm{d} \pi_\infty}{\mathrm{d} \mu}   \right\|_{B_b(S\times A)},
\end{align*}
which converges to zero as $n\to \infty$.
This shows that $\lim_{n\to\infty}\|I_n^3\|_{B_b(S)}=0$
and proves the differentiability of $J$.

Note that for all 
$f, g\in B_b(S\times A)$ and $s\in S$,
\begin{align*}
|\left(\mathfrak{d}  J(f)g\right)(s)|
&\le \left|\frac{1}{1-\gamma} \int_S 
\int_A   
\left(
Q^{\boldsymbol{\pi}(f)}_\tau (s',a) +\tau \ln\frac{\mathrm{d}  \boldsymbol{\pi}(f)}{\mathrm{d}  \mu}(a|s')
\right)
(d\boldsymbol{\pi}(f)g)(da|s')
d^{\boldsymbol{\pi}(f)}  (ds'|s)\right| 
\\
&\le \frac{1}{1-\gamma}
\left\|Q^{\boldsymbol{\pi}(f)}_\tau   +\tau \ln\frac{\mathrm{d}  \boldsymbol{\pi}(f)}{\mathrm{d}  \mu} \right\|_{B_b(S\times A)}\|\mathfrak{d}\boldsymbol{\pi}(f)g\|_{b\clM(A|S)}\,,
\end{align*}
which along with Propositions \ref{prop:boundedness_Q} and 
\ref{prop:differentiability_Pi}
imply that 
\begin{align*}
\|\mathfrak{d}  J(f)g\|_{B_b(S)}
&\le  
\frac{1}{(1-\gamma)^2}\left(\|c\|_{B_b(S\times A)}
+2\tau  \|f\|_{B_b(S\times A)}\right)
\|\mathfrak{d}\boldsymbol{\pi}(f)\|_{\clL(B_b(S\times A), b\clM(A|S))}
\| g\|_{B_b(S\times A)} 
\\
&\le  
\frac{2}{(1-\gamma)^2}\left(\|c\|_{B_b(S\times A)}
+2\tau  \|f\|_{B_b(S\times A)}\right)
\| g\|_{B_b(S\times A)}\,.
\end{align*}
This proves that
$\|\mathfrak{d}  J(f)\|_{\clL(B_b(S\times A),B_b(S))}
\le 
\frac{2}{(1-\gamma)^2}\left(\|c\|_{B_b(S\times A)}
+2\tau  \|f\|_{B_b(S\times A)}\right)
$. 
\end{proof}

The following proposition will be used to prove the convergence of the flow \eqref{eq:mirror_descent}. 
\begin{proposition}
\label{prop:derivative_integral_logexp}
Let
$\nu\in \mathcal{P}(S)$ and 
$\mu\in \clP(A)$. The map  
$T:B_b(S\times A) \to \mathbb{R}$ defined
by 
\[
T(f)= \int_{S} \ln \left(\int_A e^{f(s,a)}\mu(da)\right)\nu(ds)\,
\]
is  differentiable and for all $f,g\in B_b(S\times A)$,
\begin{equation*}
\mathfrak{d}T(f)g = \int_{S} \int_A g(s,a)  \boldsymbol{\pi}(f)(da|s)  \nu(ds)\,,
\end{equation*}
where $\boldsymbol{\pi}:B_b(S\times A) \to \clP_{\mu}(A|S)$
is defined by
\eqref{eq:pi_f_mu}.
Moreover, 
$\|\mathfrak{d}T(f)\|_{\clL(B_b(S\times A), \mathbb{R})}\le 1$
for all $f\in B_b(S\times A)$.
\end{proposition}
\begin{proof}
Let $T_1: B_b(S\times A)\to B_b(S)$ be such that $T_1(f)(s)=\ln \left(\int_A e^{f(s,a)}\mu(da)\right)$ for all $s\in S$, and  let $T_2:B_b(S)\to \bbR$ be such that  $T_2(f)=\int_S f(s)\nu(ds)$. It is clear that $T_2$ is differentiable with  $\mathfrak{d}T_2(f)g =\int_S g(s)\nu(ds)$. 

We now prove that $T_1$ is differentiable. Fix  $f,g\in B_b(S\times A)$  
and sequences  $(t_n)_{n\in \mathbb{N}} \subset (0,1)$ and $(g_n)_{n\in \mathbb{N}}\subset B_b(S\times A)$ 
such that  
$\lim_{n\to \infty} t_n =0$ and 
$\lim_{n\to \infty} g_n =g$.
Let  
$ \Psi (s) \coloneqq \int_A g(s,a)  \boldsymbol{\pi}(f)(da|s)$
for all $s\in S$.
For all $n\in \bbN$ and $s\in S$,
\begin{align*}
T_1(f+t_ng_n)(s)-T_1(f)(s) 
&=
\ln\left(1+\frac{\int_A e^{(f+t_ng_n)(s,a)}\mu(da)-\int_A e^{f(s,a)}\mu(da)}{\int_A e^{f(s,a)}\mu(da)}\right)
\\
&=\ln\left(1+ \int_A \left(e^{ t_ng_n (s,a)}-1\right) \boldsymbol{\pi}(f)(da|s) \right)\,.
\end{align*}
Thus, for all $n\in \bbN$ and $s\in S$,
\begin{align}
\label{eq:T1_difference}
\begin{split}
&\left|\frac{T_1(f+t_ng_n)(s)-T_1(f)(s)}{t_n}-\Psi(s)
\right|
\\
&
=
\frac{1}{t_n} \left| \ln\left(1+ \int_A \left(e^{ t_ng_n (s,a)}-1\right) \boldsymbol{\pi}(f)(da|s) \right)
-
\int_A \left(e^{ t_ng_n (s,a)}-1\right) \boldsymbol{\pi}(f)(da|s) 
\right|
\\
&\quad 
+ \left| 
\int_A \left(
\frac{1}{t_n} \left(e^{ t_ng_n (s,a)}-1\right)-g_n(s,a)\right) \boldsymbol{\pi}(f)(da|s) 
\right|
+ \left| 
\int_A \left(
g_n(s,a)-g(s,a)\right) \boldsymbol{\pi}(f)(da|s) 
\right|    \,.
\end{split}
\end{align}
The Taylor expansion of the exponential function at $0$ yields that 
\begin{align*}
\left\|
\frac{1}{t_n} \left(e^{ t_ng_n}-1\right)-g_n \right\|_{B_b(S\times A)}
=\left\|
\sum_{k=2}^\infty t_n^{k-1}\frac{g_n^k}{k!}
\right\|_{B_b(S\times A)}  
\le t_n \exp\left(\sup_{n\in \bbN}\|g_n\|_{B_b(S\times A)}\right)\,.
\end{align*}
This implies that 
$\sup_{n\in \bbN} 
\frac{1}{t_n}\|e^{ t_ng_n }-1\|_{B_b(S\times A )}<\infty $
and 
\[\lim_{n\to \infty}\left\| \int_A \left(e^{ t_ng_n (\cdot,a)}-1\right) \boldsymbol{\pi}(f)(da|\cdot) 
\right\|_{B_b(S)}=0\,.
\]
Thus, by
$\lim_{x\to 0}\frac{\ln(1+x)-x}{x^2}=-\frac{1}{2}$ and \eqref{eq:T1_difference},
one can show that 
\[
\lim_{n\to\infty}\left\|\frac{T_1(f+t_ng_n)-T_1(f)}{t_n}-\Psi
\right\|_{B_b(S)}=0\,.
\]
This proves that    for all $f,g\in B_b(S\times A)$
and $s\in S$,
$(\mathfrak{d}T_1(f)g)(s)=\int_A g(s,a)  \boldsymbol{\pi}(f)(da|s)$.

By Lemma \ref{lemma:chain_rule},
$T=T_2\circ T_1$ is 
differentiable, and 
for all $f, g\in B_b(S\times A)$,
\[
\mathfrak{d}T(f)g =  \mathfrak{d}T_2(T_1(f))
\big(\mathfrak{d}T_1(f)g\big)=\int_S \int_A g(s,a)  \boldsymbol{\pi}(f)(da|s) \nu(ds)\,.
\]
Finally, observe  that  for all $s\in S$, 
$\boldsymbol{\pi}(f)(da|s)\in \clP(A)$ 
and $\nu\in \clP(S)$.
Hence, for all $g\in B_b(S\times A)$,
\[
|\mathfrak{d}T(f)g|
=\left|\int_{S}
\int_A g(s,a)   \boldsymbol{\pi}(f)(da|s) \nu(ds)\right|
\le \|g\|_{B_b(S\times A)}\left|\int_{S}
\int_A    \boldsymbol{\pi}(f)(da|s) \nu(ds)\right|
=\|g\|_{B_b(S\times A)}\,,
\]
which proves that $\|\mathfrak{d}T(f)\|_{\clL(B_b(S\times A), \mathbb{R})}\le 1$.
\end{proof}

\section{Proofs of main results}\label{sec:proof_main_results}
\subsection{Proofs of Lemma \ref{lemma:dpi_t}, 
Proposition \ref{prop:value_monotone},
Theorem  \ref{thm:wp_MD}
and Theorem \ref{thm:linear_convergence}
}
\label{sec:convergence_FR_flow}
\begin{proof}[Proof of Lemma \ref{lemma:dpi_t}]
To prove Item 
\ref{item:from_MD_to_FR}, 
let $Z\in  C^1([0,T); B_b(S\times A))$
satisfy \eqref{eq:mirror_descent},
and  $\pi_t =\boldsymbol{\pi}(Z_t)$ for all $t\in [0,T) $.
By the chain rule (see Lemma 
\ref{lemma:chain_rule}) and 
Proposition 
\ref{prop:differentiability_Pi}, 
for all $t\in [0,T) $ and $s\in S$,
\begin{align}
\label{eq:dpi_dt}
\begin{split}
\partial_t {\pi_t} (da|s)
&=
(\mathfrak{d}\boldsymbol{\pi}(Z_t)\partial_t Z_t)(da|s)
=\left(
\partial_t Z_t(s,a)
-
\int_A \partial_t Z_t(s,a')\pi_t (da'|s)
\right)\pi_t(da|s)
\\
&=-\left(
Q^{\pi_t}_\tau(s,a)
+ \tau Z_t(s,a)
-
\int_A (Q^{\pi_t}_\tau(s,a')
+ \tau Z_t(s,a')) \pi_t (da'|s)
\right)\pi_t(da|s)\,,
\end{split}
\end{align}
where the last line used 
$V_{\tau}^{\pi_t}(s) = \int_A V_{\tau}^{\pi_t}(s) \pi_t (da'|s)$.
The definition   $\pi_t =\boldsymbol{\pi}(Z_t)$ 
implies that for all $(s',a)\in S\times A$,
\begin{equation*}
\ln\frac{\mathrm{d}  \pi_t}{\mathrm{d}  \mu}(a|s') =Z_t(s',a)
-\ln
\left( {\int_A e^{Z_t(s',a)}\mu(d a)} \right).
\end{equation*}
Substituting this identity into 
\eqref{eq:dpi_dt} yields
\begin{align*}
\partial_t {\pi_t} (da|s)
&=-\left(
Q^{\pi_t}_\tau(s,a)
+ \tau \ln\frac{\mathrm{d}  \pi_t}{\mathrm{d}  \mu}(a|s) 
-
\int_A \left(Q^{\pi_t}_\tau(s,a')
+ \tau \ln\frac{\mathrm{d}  \pi_t}{\mathrm{d}  \mu}(a'|s) \right) \pi_t (da'|s)
\right)\pi_t(da|s)
\\
&= -\left(
Q^{\pi_t}_\tau(s,a)
+ \tau \ln\frac{\mathrm{d}  \pi_t}{\mathrm{d}  \mu}(a|s) 
-V^{\pi_t}_\tau(s)
\right)\pi_t(da|s)\,,
\end{align*}
where the last line used  \eqref{eq:on_policy}.
This proves Item \ref{item:from_MD_to_FR}. 

To prove Item  \ref{item:from_FR_to_MD},  let  $\pi\in  C^1((0,T); \Pi_\mu )$ satisfy \eqref{eq:FR_flow}. For all $t\in (0,T)$, $s\in S$ and $B\in \clB(A)$,
\begin{align*}
\pi_t (B|s)&=\pi_0(B|s)-\int_{(0,t)\times B} \beta_r(s,a)\pi_r(da|s) \,dr\\
&=\pi_0(B|s)-\int_{(0,t)\times B} \beta_r(s,a)\frac{\mathrm{d}  \pi_r}{\mathrm{d} \mu}(a|s) \mu(da)dr\,,
\end{align*}
where $\beta_t :=Q^{\pi_t}_\tau + \tau \ln\frac{\mathrm{d}  \pi_t}{\mathrm{d}  \mu} 
-V^{\pi_t}_\tau  $.
Hence,
\begin{align*}
\int_B \frac{\mathrm{d}  \pi_t}{\mathrm{d} \mu}(a|s) \mu(da)
&=\int_B\frac{d\pi_0}{d\mu}(a|s)\mu(da)-\int_B \int_0^t \beta_r(s,a)
\frac{\mathrm{d}  \pi_r}{\mathrm{d}  \mu}  (a|s) dr\mu(da)\,,
\end{align*} 
which implies that
\[
\frac{ \mathrm{d} \pi_t}{\mathrm{d} \mu}(a|s)=\frac{ \mathrm{d} \pi_0}{\mathrm{d} \mu}(a|s)-  \int_0^t \beta_r(s,a)\frac{\mathrm{d}  \pi_r}{\mathrm{d}  \mu}  (a|s) dr\,.
\] 
By  the fundamental theorem of calculus, $\partial_t \frac{\mathrm{d}  \pi_t}{\mathrm{d}  \mu} = -\beta_t \frac{\mathrm{d}  \pi_t}{\mathrm{d}  \mu}$ for all $t\in (0,T) $.
This along with   $\partial_t  \ln  \frac{\mathrm{d}  \pi_t}{\mathrm{d}  \mu}(a|s) =\frac{1}{\frac{\mathrm{d}  \pi_t}{\mathrm{d}  \mu}(a|s)} \partial_t \frac{\mathrm{d}  \pi_t}{\mathrm{d}  \mu}(a|s) $  implies that 
$Z_t = \ln\frac{\mathrm{d}  \pi_t}{\mathrm{d}  \mu}$
satisfies for all $t\in (0,T)$,
\begin{align*}
\partial_t {Z}_t(s,a) =  - \left( Q^{\pi_t}_\tau(s,a) + \tau  Z_t (a|s)   -V^{\pi_t}_\tau(s) \right)
=  - \left( Q^{\boldsymbol{\pi} (Z_t)}_\tau(s,a) + \tau  Z_t (a|s)   -V^{\boldsymbol{\pi} (Z_t)}_\tau(s) \right),
\end{align*}
where the last identity used 
$\boldsymbol{\pi} \left(\ln\frac{\mathrm{d}  \pi_t}{\mathrm{d}  \mu}\right)(da|s)
= \frac{\mathrm{d}  \pi_t}{\mathrm{d}  \mu}(a|s) \mu(da) =\pi_t(da|s)$.
This proves Item 
\ref{item:from_FR_to_MD}. 
\end{proof}

\begin{proof}[Proof of Proposition \ref{prop:value_monotone}]
By the chain rule (see Lemma 
\ref{lemma:chain_rule}), 
Proposition 
\ref{prop:differentiability_dV_df}
and Lemma \ref{lemma:dpi_t} Item \ref{item:from_MD_to_FR},
for all $t\in [0,T)$ and $s\in S$,
\begin{align}
\label{eq:dV_dt_step1}
\begin{split}
\partial_t V^{\pi_t }_\tau (s)
&=(\mathfrak{d}  V^{\boldsymbol{\pi}(Z_t)}_\tau \partial_t Z_t)(s)
\\
&=
\frac{1}{1-\gamma} \int_S 
\int_A   
\left(
Q^{\pi_t}_\tau (s',a) +\tau \ln\frac{\mathrm{d}  \pi_t}{\mathrm{d}  \mu}(a|s')
\right)
(\mathfrak{d}\boldsymbol{\pi}(Z_t) \partial_t Z_t)(da|s')
d^{\pi_t}(ds'|s)
\\
&=
\frac{1}{1-\gamma} \int_S 
\int_A   
\left(
Q^{\pi_t}_\tau (s',a) +\tau \ln\frac{\mathrm{d}  \pi_t}{\mathrm{d}  \mu}(a|s')
\right)\partial_t \pi_t(da|s')
d^{\pi_t}(ds'|s)
\\
&=
- \frac{1}{1-\gamma} \int_S 
\int_A   
\left(
Q^{\pi_t}_\tau (s',a) +\tau \ln\frac{\mathrm{d}  \pi_t}{\mathrm{d}  \mu}(a|s')
\right)
\\
&\quad \times 
\left(
Q^{\pi_t}_\tau(s',a)
+ \tau \ln\frac{\mathrm{d}  \pi_t}{\mathrm{d}  \mu}(a|s') 
-V^{\pi_t}_\tau(s')
\right)\pi_t (da|s')
d^{\pi_t}(ds'|s)\,.
\end{split}
\end{align}
By \eqref{eq:on_policy},
for all $s'\in S$,
\begin{align*}
\begin{split}
& 
\int_A \left(Q^{\pi_t}_\tau (s',a)
+ \tau \ln\frac{\mathrm{d}  \pi_t}{\mathrm{d}  \mu}(a|s')  
-V^{\pi_t}_\tau(s')\right)\pi_t (da|s')=0\,.
\end{split}
\end{align*}
Substituting this into 
\eqref{eq:dV_dt_step1} leads to  
\begin{align*}
\begin{split}
\partial_t V^{\pi_t }_\tau (s)
&= 
- \frac{1}{1-\gamma} \int_S 
\int_A   
\left(
Q^{\pi_t}_\tau (s',a)
+ \tau \ln\frac{\mathrm{d}  \pi_t}{\mathrm{d}  \mu}(a|s')  
-V^{\pi_t}_\tau(s')
\right)^2\pi_t (da|s')
d^{\pi_t}(ds'|s)
\le 0\,.
\end{split}
\end{align*}
This proves the desired conclusion. 
\end{proof}

We   proceed to prove 
the well-posedness of \eqref{eq:FR_flow} and \eqref{eq:mirror_descent}
(i.e.,
Theorem \ref{thm:wp_MD}).
The following lemma proves the local Lipschitz continuity of 
the maps
$Z\mapsto Q^{\boldsymbol{\pi}(Z)}_\tau  $
and $Z\mapsto V^{\boldsymbol{\pi}(Z)}_\tau  $.
\begin{lemma}\label{lemma:Q_local_lipschitz}
Let  $ \boldsymbol{\pi} :B_b(S\times A) \to \clP_{\mu}(A|S)$ be defined as in \eqref{eq:pi_f_mu}. Then for all $f,g\in B_b(S\times A)$,
\begin{align*}
&
\left\| V^{\boldsymbol{\pi}(f)}_\tau  - V^{\boldsymbol{\pi}(g)}_\tau\right\|_{B_b(S)}  \le \frac{2\gamma }{(1-\gamma)^2}\left(\|c\|_{B_b(S\times A)}+2\tau (\|f\|_{B_b(S\times A)}+\|g\|_{B_b(S\times A)})\right) \|f-g\|_{B_b(S\times A)}\,,
\\
&\left\| Q^{\boldsymbol{\pi}(f)}_\tau -Q^{\boldsymbol{\pi}(g)}_\tau \right\|_{B_b(S\times A)}\le 
\frac{2\gamma^2 }{(1-\gamma)^2}\left(\|c\|_{B_b(S\times A)}+2\tau (\|f\|_{B_b(S\times A)}+\|g\|_{B_b(S\times A)})\right) \|f-g\|_{B_b(S\times A)}\,.
\end{align*}
\end{lemma}
\begin{proof}
Let $f,g \in B_b(S\times A)$ be fixed. Proposition \ref{prop:differentiability_dV_df} and the mean value theorem show that 
\begin{align}
& \left\| V^{\boldsymbol{\pi}(f)}_\tau  - V^{\boldsymbol{\pi}(g)}_\tau\right\|_{B_b(S)}\le \gamma  \sup_{\theta \in [0,1]}  \left\| \mathfrak{d} 
V^{\boldsymbol{\pi}(g + \theta (f-g))}_\tau \right\|_{\clL(_{B_b(S\times A)}, B_b(S))}\|f-g\|_{B_b(S\times A)}
\notag \\
&\quad \le \frac{2\gamma }{(1-\gamma)^2}\left(\|c\|_{B_b(S\times A)}+2\tau (\|f\|_{B_b(S\times A)}+\|g\|_{B_b(S\times A)})\right) \|f-g\|_{B_b(S\times A)}\,.\label{ineq:ll_v}
\end{align}
By \eqref{def:Q_fn}, we have
\[
\left\| Q^{\boldsymbol{\pi}(f)}_\tau  - Q^{\boldsymbol{\pi}(g)}_\tau \right\|_{B_b(S\times A)}\le \gamma \left\| V^{\boldsymbol{\pi}(f)}_\tau  - V^{\boldsymbol{\pi}(g)} _\tau\right\|_{B_b(S)}\,,
\]
which combined with \eqref{ineq:ll_v}  proves the desired locally Lipschitz continuity of $f\mapsto Q^{\boldsymbol{\pi}(f)}_\tau$.
\end{proof}

Based on Proposition \ref{prop:value_monotone},
we prove that $(Q^{\pi_t}_\tau -V^{\pi_t}_\tau )_{t\ge 0}$ remains bounded along the flow \eqref{eq:mirror_descent},  
and hence solutions to \eqref{eq:mirror_descent}
do not explode in any finite time.  

\begin{proposition}
\label{prop:a_priori_bound_Z}
Let   $Z_0\in B_b(S\times A)$ and  $T>0$. If $Z\in  C^1([0,T); B_b(S\times A))$ satisfies \eqref{eq:mirror_descent} and   $\pi_t =\boldsymbol{\pi}(Z_t)$ for all $t\in [0,T)$, then   for all $t\in [0,T)$,
\begin{align*} 
\|Z_t\|_{B_b(S\times A)} \le  e^{-\tau t} \| Z_0 \|_{B_b(S\times A)} +\frac{1-e^{-\tau t}}{\tau}\left( \|c\|_{B_b(S\times A)} +(1+\gamma) \max \left( \|V^{\pi_0}_{\tau}\|_{B_b(S)},  \|V^{*}_{\tau}\|_{B_b(S)} \right)\right),
\end{align*}
where $V^{*}_{\tau}\in B_b(S) $ is the optimal value function defined in \eqref{eq:optimal_value}.
\end{proposition}
\begin{proof}
Recall that by Proposition \ref{prop:value_monotone}, the value function
$t\mapsto  V^{\pi_t}_\tau $ 
is decreasing along the flow. Thus, for all $s\in S$ and $t\in [0,T)$,  
\[
V^*_\tau (s)\le V^{\pi_t }_\tau (s) \le  V^{\pi_0}_{\tau}(s)\,.
\]
This implies that
\[
\|V^{\pi_t }_\tau\|_{B_b(S)}  \le \max \left( \|V^{\pi_0}_{\tau}\|_{B_b(S)},  \|V^*_\tau\|_{B_b(S)}\right),
\]
which along with \eqref{def:Q_fn} shows that for all $t\in [0,T)$,
\[
\|Q^{\pi_t}_{\tau}\|_{B_b(S\times A)} \le \|c\|_{B_b(S\times A)} +\gamma \max \left( \|V^{\pi_0}_{\tau}\|_{B_b(S)},  \|V^*_\tau\|_{B_b(S)} \right).
\]
Since $Z$ satisfies \eqref{eq:mirror_descent} for all $t\in [0,T)$,  by Duhamel's principle, for all $t\in [0,T)$,
\[
Z_t(s,a) = e^{-\tau t} Z_0(s,a) -\int_0^t e^{-\tau(t-r)} ( Q_{\tau}^{\pi_{r}}(s,a) 
-V_{\tau}^{\pi_{r}}(s)) 
dr\,.
\]
It follows that  for all $t\in [0,T)$,
\begin{align*} 
\|Z_t\|_{B_b(S\times A)} 
&\le  e^{-\tau t} \| Z_0 \|_{B_b(S\times A)} +\int_0^t e^{-\tau(t-r)} 
\sup_{r\in [0,t]}\left(\|Q_{\tau}^{\pi_{r}} \|_{B_b(S\times A)} 
+\|V_{\tau}^{\pi_{r}} \|_{B_b(S)} 
\right)
dr
\\
&\le 
e^{-\tau t} \| Z_0 \|_{B_b(S\times A)} +\frac{1-e^{-\tau t}}{\tau} \left(\|c\|_{B_b(S\times A)} + (1+\gamma) \max \left( \|V^{\pi_0}_{\tau}\|_{B_b(S)}, \| V^*_\tau\|_{B_b(S)} \right)\right),
\end{align*}
where the last inequality used $\int_0^t e^{-\tau(t-r)} dr =\frac{1}{\tau}(1-e^{-\tau t})$. This proves the desired a-priori bound. 
\end{proof}

Now, we  prove Theorem \ref{thm:wp_MD}  
based on Lemma \ref{lemma:Q_local_lipschitz},
Proposition \ref{prop:a_priori_bound_Z}
and a truncation argument. 

\begin{proof}[Proof of Theorem \ref{thm:wp_MD}]
By Lemma \ref{lemma:dpi_t},
it suffices to prove the well-posedness of \eqref{eq:mirror_descent}. 
Throughout this proof, 
let 
$Z_0\in B_b(S\times A)$ be fixed, 
let 
\[
L=\| Z_0 \|_{B_b(S\times A)} +\frac{1 }{\tau} 
\left(\|c\|_{B_b(S\times A)} +(1+\gamma) \max \left( \|V^{\pi_0}_{\tau}\|_{B_b(S)},  \|V^{*}_{\tau}\|_{B_b(S)} \right)\right)<\infty\,,
\]
let 
$\mathcal{F}:B_b(S\times A)\to B_b(S\times A)  $ be such that for all $f\in B_b(S\times A)$,
\[
\mathcal{F}(f)(s,a) =-(Q^{\boldsymbol{\pi}(f)}_\tau(s,a) + \tau f(s,a)
-V^{\boldsymbol{\pi}(f)}_\tau(s)
)\,,
\]
and 
let 
$\mathcal{F}_L:B_b(S\times A)\to B_b(S\times A)  $ be such that 
for all $f\in B_b(S\times A)$,
\begin{align*}
\mathcal{F}_L (f)=
\begin{cases}
\mathcal{F}(f)(s,a), & \|f\|_{B_b(S\times A)}\le L,
\\
\mathcal{F}\left(\frac{Lf}{\|f\|_{B_b(S\times A)}}\right)(s,a), & \|f\|_{B_b(S\times A)}> L.
\end{cases}
\end{align*}
By Lemma \ref{lemma:Q_local_lipschitz},
there exists a constant $C_0>0$
such that for all $f,g\in B_b(S\times A)$,
\[
\|\mathcal{F}_L (f)-\mathcal{F}_L (g)\|_{B_b(S\times A)}\le C_0\|f-g\|_{B_b(S\times A)}\,.
\]

We first prove that
there exists a unique   
$u \in  C^1(\mathbb{R}_+; B_b(S\times A))$
satisfying 
\begin{equation}
\label{eq:mirror_descent_truncated}
\partial_t u_t = \mathcal{F}_L(u_t),\quad t>0; \quad u_0=Z_0\,.
\end{equation}
For each $T\in (0,\infty)$,
let 
$C([0,T];B_b(S\times A))$
be the space of continuous functions 
$u:[0,T] \to B_b(S\times A)$.
For each $\lambda\in \mathbb{R}$, consider the norm  
$\|\cdot\|_{T,\lambda}$ 
on $C([0,T];B_b(S\times A))$
such that for all $u\in C([0,T];B_b(S\times A))$, 
$\|u\|_{T,\lambda} =
\sup_{t\in [0,T]}e^{\lambda t} \|u_t\|_{B_b(S\times A)}$.
Note that   $\left(C([0,T];B_b(S\times A)),\|\cdot\|_{T,0}\right)$ is a Banach space (see e.g.,  \cite[Theorem 3.2-2]{ciarlet2013linear})
and 
for all $\lambda\in \mathbb{R}$,
$\|\cdot\|_{T,0}$ and 
$\|\cdot\|_{T,\lambda}$  are equivalent norms.
Hence, for all $\lambda\in \mathbb{R}$,
$X_{T,\lambda}\coloneqq \left(C([0,T];B_b(S\times A)),\|\cdot\|_{T,\lambda}\right)$ is a Banach space.   
For each   $T>0$
and $\lambda>0$,
consider the map
$\Psi:X_{T,\lambda}\to X_{T,\lambda}$
such that for all $u\in X_{T,\lambda}$,
\[
\Psi(u)_t=Z_0+\int_0^t \mathcal{F}_L (u_s) ds,
\quad t\in [0,T]\,.
\]
Then for all $u,v\in X_{T,\lambda}$ and $t\in [0,T]$,
\begin{align*}
&\|\Psi(u)_t-\Psi(u)_t\|_{B_b(S\times A)}
\le \int_0^t 
\|\mathcal{F}_L (u_s)-\mathcal{F}_L (v_s)\|_{B_b(S\times A)}ds 
\le \int_0^t C_0 
\| u_s -  v_s \|_{B_b(S\times A)}ds 
\\
&\quad =\int_0^t 
C_0\|u_s-v_s\|_{B_b(S\times A)}e^{-\lambda s}e^{  \lambda s} ds 
\le C_0
\left(\sup_{t\in [0,T]} 
\|u_t-v_t\|_{B_b(S\times A)}e^{-\lambda t}
\right)
\int_0^t 
e^{  \lambda s} ds 
\\
&\quad \le 
C_0
\|u-v\|_{T,-\lambda}\frac{1}{\lambda}  (e^{\lambda t}-1)
\le \frac{C_0}{\lambda}  
\|u-v\|_{T,-\lambda} e^{\lambda t}\,.
\end{align*}
Hence, for  $\lambda_0=C_0+1$, 
$ \|\Psi(u)-\Psi(v)\|_{T,-\lambda_0} \le \frac{C_0}{C_0+1} \|u-v\|_{T,-\lambda_0} $
for all $u,v\in X_{T,-\lambda_0}$. 
By the Banach fixed point theorem, 
$\Psi$ admits a unique fixed point in 
$X_{T,-\lambda_0}$, which 
along with the fundamental theorem of calculus
implies that 
\eqref{eq:mirror_descent_truncated}
admits a unique solution 
in $C^1(\mathbb{R}_+; B_b(S\times A))$.

Now let $u\in  C^1(\mathbb{R}_+; B_b(S\times A))$ be the solution to \eqref{eq:mirror_descent_truncated}.
The definition of $\mathcal{F}_L$ implies that 
$u$ satisfies \eqref{eq:mirror_descent}
for all $t\in [0,t_{\max})$,
where 
$t_{\max}=\inf\{t\in \mathbb{R}_+\mid \|u_t\|_{B_b(S\times A)}>L\}$.
The definition of $L$,
$u_0=Z_0$
and  Proposition 
\ref{prop:a_priori_bound_Z}
imply that
$t_{\max}=\infty$,
which subsequently shows  that  
$u\in  C^1(\mathbb{R}_+; B_b(S\times A))$ is a  solution to \eqref{eq:mirror_descent}.
To prove  the uniqueness of solutions to 
\eqref{eq:mirror_descent}, 
let $u, u'\in  C^1(\mathbb{R}_+; B_b(S\times A))$ be solutions of  \eqref{eq:mirror_descent} such that 
$u_0=u'_0=Z_0$.
By Proposition 
\ref{prop:a_priori_bound_Z},
$\|u_t\|_{B_b(S\times A)}\le L$
and $\|u'_t\|_{B_b(S\times A)}\le L$ for all $t>0$, 
which implies that 
$u$ and $u'$ are solutions to \eqref{eq:mirror_descent_truncated}.
The uniqueness of solutions to \eqref{eq:mirror_descent_truncated} implies that $u=u'$, which yields the uniqueness of solutions to \eqref{eq:mirror_descent}  and completes the proof of well-posedness of \eqref{eq:mirror_descent}.
\end{proof}

It remains to establish  the exponential  convergence of \eqref{eq:mirror_descent}, namely Theorem \ref{thm:linear_convergence}.
To this end,  for each $\nu\in \clP(S)$, we introduce the Bregman divergence  $D_\nu : B_b(S\times A) \times  B_b(S\times A)\to \mathbb{R} $ defined for all $f,g\in B_b(S\times A)$ by
\begin{equation}
\label{eq:bregman_Dfg}
D_\nu(f, g) 
=\int_S \left(\Phi(f)(s)-\Phi(g)(s)-\int_A  (f(s,a)-g(s,a))\boldsymbol{\pi}(g)(da|s)\right)\nu(ds)\,,
\end{equation}   
where $\boldsymbol{\pi}:B_b(S\times A)\to \clP_{\mu}(A|S)$ is defined by \eqref{eq:pi_f_mu} and  $\Phi: B_b(S\times A)\to B_b(S)$ is defined  for all $f\in B_b(S\times A)$ by
\begin{equation}\label{def:Phi}
\Phi(f)(s)=\ln \left(\int_A e^{f(s,a)}\mu(da)\right),
\quad s\in S\,.
\end{equation}
We first establish  the duality relation between the map $D_\nu$ defined in \eqref{eq:bregman_Dfg} and the KL-divergence. 

\begin{lemma}
\label{lemma:D_KL}
For all  $\rho\in \clP(S)$ and 
$f,g\in B_b(S\times A)$,
\[
D_{d_{\rho}^{\boldsymbol{\pi}(g)}}(f,g)=\int_S \operatorname{KL}(\boldsymbol{\pi}(g)(\cdot|s)|  \boldsymbol{\pi}(f)(\cdot|s))d_{\rho}^{\boldsymbol{\pi}(g)}(ds)\,.
\]

\end{lemma}

\begin{proof}
Throughout this proof, 
fix  $f,g\in  B_b(S\times A)$  
and write $\pi_f =\boldsymbol{\pi}(f)$ and 
$\pi_g =\boldsymbol{\pi}(g)$.  Then
\begin{align*}
&\int_S \operatorname{KL}(\pi_g (\cdot|s)|  \pi_f (\cdot|s))d_{\rho}^{\pi_g}(ds)
=\int_S\int_A
\ln\frac{\mathrm{d}  \pi_g }{\mathrm{d}  \pi_f }(a|s)  
\pi_g (da|s) d_{\rho}^{\pi_g  }(ds)
\\
&\quad =
\int_S\int_A
\left(
\ln\frac{\mathrm{d}  \pi_g }{\mathrm{d}  \mu}(a|s)  
- \ln\frac{\mathrm{d}  \pi_f }{\mathrm{d}  \mu}(a|s)  
\right)\pi_g(da|s) d_{\rho}^{\pi_g}(ds)\,.
\end{align*}
Observe that  
for all 
$h\in B_b(S\times A)$ and 
$(s ,a)\in S\times A$,
$  \ln\frac{\mathrm{d}  \boldsymbol{\pi}(h)}{\mathrm{d}  \mu}(a|s ) =h(s,a)
-\ln
\left( {\int_A e^{h(s,a)}\mu(d a)} \right)$. 
The desired conclusion then follows from the following identity: 
\begin{align*}
&\int_S \operatorname{KL}(\pi_g (\cdot|s)|  \pi_f (\cdot|s))d_{\rho}^{\pi_g}(ds)
=
\int_S\int_A
\left(
\ln\frac{\mathrm{d}  \pi_g }{\mathrm{d}  \mu}(a|s)  
- \ln\frac{\mathrm{d}  \pi_f }{\mathrm{d}  \mu}(a|s)  
\right)\pi_g(da|s) d_{\rho}^{\pi_g}(ds)\\
&  \quad =
\int_S
\left(
\ln
\left( {\int_A e^{f(s,a)}\mu(d a)} \right)
-
\ln
\left( {\int_A e^{g(s,a)}\mu(d a)} \right)
-\int_A (f(s,a)-g(s,a))
\pi_g(da|s)  \right) d_{\rho}^{\pi_g}(ds)
\\
&  \quad =D_{d_{\rho}^{\boldsymbol{\pi}(g)}}(f,g)\,.\qedhere
\end{align*}

\end{proof}

\begin{proof}[Proof of Theorem \ref{thm:linear_convergence}]

By \eqref{eq:optimal_policy},
there exists $Z^*\in B_b(S\times A)$
such that 
$\pi^*_\tau =\boldsymbol{\pi}(Z^*) $.     
By the chain rule (see Lemma 
\ref{lemma:chain_rule}), Proposition \ref{prop:derivative_integral_logexp}, and the definition of $\Phi$ given in \eqref{def:Phi},
for all $t>0$,
\begin{align}
\label{eq:DZtZ_derivative}
\begin{split}
\partial_t   D_{d_{\rho}^{\pi^*_\tau  }}(Z_t,Z^*)
&=\partial_t  \left(
\int_S 
\left(
\Phi(Z_t)(s) - \Phi(Z^*)(s)       -\int_A  (Z_t(s,a)-Z^*(s,a))\pi^*_\tau (da|s)
\right)d^{\pi^*_\tau  }_\rho(ds)\right)
\\
&=
\partial_t  \left(
\int_S 
\ln \left(\int_A e^{Z_t(s,a)}\mu(da)\right) d^{\pi^*_\tau  }_\rho(ds)\right)
-\int_S \int_A  \partial_t Z_t(s,a)  \pi^*_\tau (da|s)d^{\pi^*_\tau  }_\rho(ds)
\\
&=\int_{S} \int_A \partial_t Z_t(s,a)  \pi_t(da|s)  d^{\pi^*_\tau  }_\rho(ds)
-\int_S \int_A  \partial_t Z_t(s,a)  \pi^*_\tau (da|s)d^{\pi^*_\tau  }_\rho(ds)
\\
&=\int_{S} \int_A \partial_t Z_t(s,a)  (\pi_t(da|s)  - \pi^*_\tau (da|s))d^{\pi^*_\tau  }_\rho(ds)
\\
&=
-\int_{S} \int_A (Q_\tau^{\pi_t}(s,a)  +\tau Z_t (s,a)
-V_\tau^{\pi_t}(s) )  (\pi_t(da|s)  - \pi^*_\tau (da|s))d^{\pi^*_\tau  }_\rho(ds)
\\
&=- \int_{S} \int_A \left(
Q_\tau^{\pi_t}(s,a) +\tau \ln\frac{\mathrm{d}  \pi_t}{\mathrm{d}  \mu}(a|s)   \right)  (\pi_t(da|s)  - \pi^*_\tau (da|s))d^{\pi^*_\tau  }_\rho(ds)
\,,
\end{split}
\end{align}
where the second to last line used  the fact that  $Z\in  C^1(\mathbb{R}_+; B_b(S\times A))$
satisfies \eqref{eq:mirror_descent}, and the last line used the facts that  
$  
Z_t(s,a)=\ln\frac{\mathrm{d}  \pi_t}{\mathrm{d}  \mu}(a|s) + \ln
\left( {\int_A e^{Z_t(s,a')}\mu(d a')} \right)
$  for all $(s,a)\in S\times A$ and
\begin{align*}
&\int_A
g(s) (\pi_t(da|s)  - \pi^*_\tau (da|s))
=
g(s)   (\pi_t(A|s)-\pi^*_\tau(A|s))
=0
\end{align*}
for all $s\in S$ and $g\in B_b(S)$.
Note that by
Lemma \ref{lem:performance_diff}, for all $t>0$,
\begin{align}
\label{eq:performance_gap_pi_t}
\begin{split}
&
\int_S 
\int_A\left(Q^{\pi_t}_{\tau}(s,a)+\tau \ln \frac{\mathrm{d} \pi_t}{\mathrm{d} \mu}(s,a)\right)(\pi_t-\pi^*_\tau)(da|s)d^{\pi^*_\tau}_{\rho} (ds)
\\
&\quad
=(1-\gamma)
(V^{\pi_t}_{\tau}(\rho)-V^{\pi^*_\tau}_{\tau}(\rho))  
+\tau \int_S 
\operatorname{KL}(\pi^*_\tau(\cdot|s)|\pi_t(\cdot|s) )d^{\pi^*_\tau}_{\rho} (ds)\, .
\end{split}
\end{align}
Substituting this identity into \eqref{eq:DZtZ_derivative} yields          
\begin{align*}
\partial_t   D_{d_{\rho}^{\pi^*_\tau  }}(Z_t,Z^*)
&=- \left((1-\gamma)
(V^{\pi_t}_{\tau}(\rho)-V^{\pi^*_\tau}_{\tau}(\rho))
+\tau \int_S 
\operatorname{KL}(\pi^*_\tau(\cdot|s)|\pi_t(\cdot|s) )d^{\pi^*_\tau}_{\rho} (ds)
\right)  
\\
&=- (1-\gamma)
(V^{\pi_t}_{\tau}(\rho)-V^{\pi^*_\tau}_{\tau}(\rho))
-\tau  D_{d_{\rho}^{\pi^*_\tau }}(Z_t,Z^*)\,,
\end{align*}
where the last line follows from 
Lemma \ref{lemma:D_KL} (with 
$f= Z_t$ and $g=Z^*$).
Then for all $t>0$,
by Duhamel's principle
and the fact that $ V^{\pi_t}_\tau (\rho)\le V^{\pi_{t'}}_\tau (\rho)$
for all $0\le t'\le t$ (see Proposition \ref{prop:value_monotone}), 
\begin{align}
\label{eq:D_ZtZ_exponential}
\begin{split}
D_{d_{\rho}^{\pi^*_\tau  }}(Z_t,Z^*)&=e^{-\tau t}  D_{d_{\rho}^{\pi^*_\tau  }}(Z_0,Z^*)   
-(1-\gamma)\int_0^t e^{-\tau (t-t')}(V^{\pi_{t'}}_{\tau}(\rho)-V^{\pi^*_\tau}_{\tau}(\rho))dt'
\\
&\le 
e^{-\tau t}  D_{d_{\rho}^{\pi^*_\tau  }}(Z_0,Z^*)   
-(1-\gamma)\int_0^t e^{-\tau (t-t')}dt'(V^{\pi_{t}}_{\tau}(\rho)-V^{\pi^*_\tau}_{\tau}(\rho))\,.
\end{split}
\end{align}
This, along with the identity
$\int_0^t e^{-\tau(t-r)} dr =\frac{1}{\tau}(1-e^{-\tau t})$
shows that 
\begin{align*}
V^{\pi_{t}}_{\tau}(\rho)-V^{\pi^*_\tau}_{\tau}(\rho)
\le \frac{\tau }{(1-\gamma)(1-e^{-\tau t})}\left( e^{-\tau t}  D_{d_{\rho}^{\pi^*_\tau  }}(Z_0,Z^*) -D_{d_{\rho}^{\pi^*_\tau  }}(Z_t,Z^*)\right)\,.
\end{align*}
Lemma \ref{lemma:D_KL} (with 
$f= Z_t$ and $g=Z^*$) and Pinsker's inequality
shows that for all $t>0$,
\begin{equation}
\label{eq:D_ZtZ_pinsker}
D_{d_{\rho}^{\pi^*_\tau  }}(Z_t,Z^*)
=\int_S 
\operatorname{KL}(\pi^*_\tau(\cdot|s)|\pi_t(\cdot|s) )d^{\pi^*_\tau}_{\rho} (ds)
\ge 
\int_S 
\frac{1}{2}\|\pi^*_\tau(\cdot|s)-\pi_t(\cdot|s) \|^2_{\clM(A)}d^{\pi^*_\tau}_{\rho} (ds)\ge 0,
\end{equation}
which implies that for all $t>0$,
\begin{align*}
V^{\pi_{t}}_{\tau}(\rho)-V^{\pi^*_\tau}_{\tau}(\rho)
\le \frac{\tau }{(1-\gamma)(e^{ \tau t}-1)} 
\int_S 
\operatorname{KL}(\pi^*_\tau(\cdot|s)|\pi_0(\cdot|s) )d^{\pi^*_\tau}_{\rho} (ds)\,.
\end{align*}
This proves the desired exponential convergence
of $(V^{\pi_t}_\tau)_{t>0}$.
The convergence of $(\pi_t)_{t>0}$ follows from 
\eqref{eq:D_ZtZ_exponential},
\eqref{eq:D_ZtZ_pinsker} and the fact that $V^{\pi_t}_\tau (\rho)\ge V^{\pi^*_\tau}_\tau (\rho)$ for all $t>0$.
\end{proof}

\subsection{Proof of Theorem \ref{thm:stability_mirror}}
\label{sec:stability_FR}

The following two technical lemmas on change of measures will be used for proving Theorem \ref{thm:stability_mirror}. 

\begin{lemma}\label{lemma:averaged_measure}
Let  $(X,\mathcal{X})$
be a measurable space
and let 
$\mu, \nu_1,  \mu_2\in \clP(X)$.   
Let $\alpha\in (0,1)$ and define $\nu =\alpha \nu_1+(1-\alpha)\nu_2$. 
If $\mu\ll \nu_1$,
then $\mu\ll \nu$ and  
$0\le \frac{\mathrm{d} \mu}{\mathrm{d} \nu}
\le \frac{1}{\alpha}\frac{\mathrm{d} \mu}{\mathrm{d} \nu_1}
$, $\nu$-a.s.
\end{lemma}
\begin{proof}
Let $A\in \mathcal{X}$ with 
$\nu(A)=0$. The definition of $\nu$
and $\alpha\in (0,1)$
imply that $\nu_1(A)=0$,
which along with $\mu \ll \nu_1$ shows  that $\mu(A)=0$. This proves $\mu\ll \nu$, and hence $ \frac{\mathrm{d} \mu}{\mathrm{d} \nu}$ exists and is non-negative $\nu$-a.s.
We then establish the upper bound of $ \frac{\mathrm{d} \mu}{\mathrm{d} \nu}$ by contradiction.
Suppose that there exists $A\in \mathcal{X}$ such that
$\nu(A)>0$ and $\frac{\mathrm{d} \mu}{\mathrm{d} \nu}
> \frac{1}{\alpha}\frac{\mathrm{d} \mu}{\mathrm{d} \nu_1}$ on $A$.
For each $n\in \mathbb{N}$, 
define 
$A_n= \left\{x\in A|\left( \frac{\mathrm{d} \mu}{\mathrm{d} \nu}
- \frac{1}{\alpha}\frac{\mathrm{d} \mu}{\mathrm{d} \nu_1}\right)(x)>\frac{1}{n}\right\}$.
Since $A_n\subset A_{n+1}$ for all $n\in\mathbb{N}$, the monotone convergence theorem of measurable sets shows that 
\[
\lim_{n\to \infty}\nu(A_n)=
\nu(\cup_{n\in \mathbb{N}} A_n)=\nu(A)>0\, .
\]
This implies that there exists $n_0\in \mathbb{N}$ such that $\nu(A_{n_0})>0$
and 
\[
\int_{A}
\left(
\frac{\mathrm{d} \mu}{\mathrm{d} \nu}
- \frac{1}{\alpha}\frac{\mathrm{d} \mu}{\mathrm{d} \nu_1}
\right) d \nu  
\ge 
\int_{A_{n_0}}
\left(
\frac{\mathrm{d} \mu}{\mathrm{d} \nu}
- \frac{1}{\alpha}\frac{\mathrm{d} \mu}{\mathrm{d} \nu_1}
\right) d \nu   \ge \frac{1}{n_0}\nu(A_{n_0})>0\, .
\]
Consequently,  
\begin{align*}
\mu(A)&=\int_A \frac{\mathrm{d} \mu}{\mathrm{d} \nu} d \nu 
>\frac{1}{\alpha}\int_A \frac{\mathrm{d} \mu}{\mathrm{d} \nu_1} d \nu 
=\frac{1}{\alpha}\int_A \frac{\mathrm{d} \mu}{\mathrm{d} \nu_1} d (\alpha \nu_1+(1-\alpha)\nu_2)
\\
&=  \int_A \frac{\mathrm{d} \mu}{\mathrm{d} \nu_1} d   \nu_1 
+\frac{1-\alpha}{\alpha}\int_A \frac{\mathrm{d} \mu}{\mathrm{d} \nu_1} d  \nu_2
\ge \mu(A)\,,
\end{align*}
where the last inequality used  the fact that   $\frac{\mathrm{d} \mu}{\mathrm{d} \nu_1} $ is a non-negative measurable function, and hence $\int_A \frac{\mathrm{d} \mu}{\mathrm{d} \nu_1} d  \nu_2\ge 0$.
This yields a contradiction and proves 
the desired inequality that  
$\frac{\mathrm{d} \mu}{\mathrm{d} \nu}
\le \frac{1}{\alpha}\frac{\mathrm{d} \mu}{\mathrm{d} \nu_1}
$ $\nu$-a.s. 
\end{proof}    

\begin{lemma}
\label{lemma:absolute_continuity_kernel}
Let  $(X,\mathcal{X})$ and $(Y,\mathcal{Y})$
be  two measurable spaces,
let $\mu,\nu\in \clP(X)$ and 
let $\pi_1,\pi_2\in \clP(Y|X)$.
If 
$\mu\ll \nu$  and  
$\pi_1\ll \pi_2$ (i.e., 
there exists a measurable function $\eta :X\times Y\to \mathbb{R}_+$ such that $\pi_1(B|x)= \int_B \eta(x,y)\pi_2(dy|x)$
for all $x\in X$ and $B\in \clY$),
then $\mu\otimes \pi_1 \ll \nu\otimes \pi_2$ and  
$  \frac{\mathrm{d} (\mu\otimes \pi_1 ) }{\mathrm{d} (\nu\otimes \pi_2)} 
=\eta \frac{\mathrm{d}  \mu }{\mathrm{d} \nu }
$,  $\nu\otimes \pi_2$-a.s.
Consequently, 
$\left\|  \frac{\mathrm{d} (\mu\otimes \pi_1) }{\mathrm{d} (\nu\otimes \pi_1)}\right\|_{L^\infty(X\times Y, \nu\otimes \pi_1) }
\le 
\left\|  \frac{\mathrm{d} \mu  }{\mathrm{d} \nu }\right\|_{L^\infty(X , \nu ) } 
$.
\end{lemma}
\begin{proof}
Let $A\in \mathcal{X}$ and $B\in \mathcal{Y}$. 
By the definitions of $(\mu\otimes \pi_1 ) $ and $\eta$ and Fubini's theorem,
\begin{align*}
(\mu\otimes \pi_1 ) (A\times B)
&=\int_A  \pi_1(B|x)\mu(dx)
=\int_A\left(\int_B 
\eta(x,y)\pi_2(dy|x)\right)\mu(dx)
\\
&=\int_A\int_B
\eta(x,y)\pi_2(dy|x)\frac{\mathrm{d} \mu }{\mathrm{d} \nu}(x)\nu(dx)
=\int_{A\times  B}
\eta(x,y)\frac{\mathrm{d} \mu }
{\mathrm{d} \nu}(x)(\nu\otimes \pi_2 )(dx,dy)\,.
\end{align*}
Consider the collections 
\begin{align*} 
\mathcal{A}&\coloneqq \{A\times B | A\in \mathcal{X}, B\in \mathcal{Y}\}\,,
\\
\mathcal{C}& \coloneqq 
\left\{ 
E\in \mathcal{X}\otimes \mathcal{Y}
\,\middle\vert\, 
(\mu\otimes \pi_1 ) (E)=
\int_E \eta(x,y)\frac{\mathrm{d} \mu }{\mathrm{d} \nu}(x)(\nu\otimes \pi_2 )(dx,dy)  
\right\}\,,
\end{align*}
and note that  $\mathcal{A}$ is a $\pi$-system that generates $\mathcal{X}\otimes \mathcal{Y}$
and $\mathcal{A} \subset \mathcal{C} \subset \mathcal{X}\otimes \mathcal{Y}$.
It remains   to show $\mathcal{C}$ is a $\lambda$-system,
which along with
Dynkin's lemma (see \cite[Lemma 4.11]{guide2006infinite})
implies that 
$\mathcal{X}\otimes \mathcal{Y}\subset \mathcal{C} $, and hence $\mathcal{X}\otimes \mathcal{Y}= \mathcal{C}$.
More precisely, we aim to verify that 
$\mathcal{C}$ satisfies the following three properties: 
(1) $X\times Y\subset \mathcal{C}$,
(2) if $E_1,E_2\in \mathcal{C}$ and $E_1\subset E_2$, then $E_2\setminus E_1\in \mathcal{C}$,
and (3) if  $\{E_n| n\in \mathbb{N}\}\subset  \mathcal{C}$ satisfies 
$E_n \subset E_{n+1}$ for all $n\in \mathbb{N}$, then $\cup_{n\in \mathbb{N}} E_n \in \mathcal{C}$.
Property (i) follows from $X\times Y\in \mathcal{A}$, 
Property (ii) follows from 
$(\mu\otimes \pi_1 ) (E_2\setminus E_1)
=(\mu\otimes \pi_1 ) (E_2)-(\mu\otimes \pi_1 ) (E_1)$,
and Property (iii) follows from the monotone convergence theorem. 
The identity 
$ \mathcal{C}=\mathcal{X}\otimes \mathcal{Y}$
proves 
$\mu\otimes \pi_1 \ll \nu\otimes \pi_2$ 
and characterises the Radon-Nikodym derivative 
$ \frac{\mathrm{d} (\mu\otimes \pi_1 ) }{\mathrm{d} (\nu\otimes \pi_2)}$.

Observe that 
if $\pi_1=\pi_2$, then 
$ \frac{\mathrm{d} (\mu\otimes \pi_1 ) }{\mathrm{d} (\nu\otimes \pi_1)} 
=1_{X\times Y}  \frac{\mathrm{d}  \mu }{\mathrm{d} \nu } $ , $\nu\otimes \pi_1$-a.s. 
Let  $C\ge 0$ be
such that 
$\frac{\mathrm{d}  \mu }{\mathrm{d} \nu } \le C$, $\nu$-a.s. 
Then there exists
$A\in \mathcal{X}$  
such that $\nu(A)=1$ and $
\frac{\mathrm{d}  \mu }{\mathrm{d} \nu } \le C$ on $A$.
As   $(\nu\otimes \pi_1) (A\times Y)=1$
and $|1_{X\times Y}  \frac{\mathrm{d}  \mu }{\mathrm{d} \nu }|\le C$ on $A\times Y$,
it holds  that $0\le \frac{\mathrm{d} (\mu\otimes \pi_1 ) }{\mathrm{d} (\nu\otimes \pi_1)} \le C$,  $\nu\otimes \pi_1$-a.s., and hence 
$\left\|  \frac{\mathrm{d} (\mu\otimes \pi_1) }{\mathrm{d} (\nu\otimes \pi_1)}\right\|_{L^\infty(X\times Y, \nu\otimes \pi_1) }\le C$.
Taking the infimum over $C$
in the set 
$\{C\ge 0| \textnormal{$\frac{\mathrm{d}  \mu }{\mathrm{d} \nu } \le C$ $\nu$-a.s.}
\}$ 
yields 
the  desired estimate  
$\left\|  \frac{\mathrm{d} (\mu\otimes \pi_1) }{\mathrm{d} (\nu\otimes \pi_1)}\right\|_{L^\infty(X\times Y, \nu\otimes \pi_1) }
\le   
\left\|  \frac{\mathrm{d} \mu  }{\mathrm{d} \nu }\right\|_{L^\infty(X , \nu ) } $.
\end{proof}

\begin{proof}[Proof of Theorem \ref{thm:stability_mirror}]
Throughout this proof,
let  $Z^*\in B_b(S\times A)$ be
such that 
$\pi^*_\tau =\boldsymbol{\pi}(Z^*) $ (see \eqref{eq:optimal_policy}), 
let $\pi\in  C^1((0,T); \Pi_\mu   )$
satisfy   \eqref{eq:FR_flow_approx},
and for each $t>0$, let 
$Z_t =\ln\frac{\mathrm{d}  \pi_t}{\mathrm{d}  \mu} 
\in B_b(S\times A)$.
Then similar to 
Lemma \ref{lemma:dpi_t},
one can prove that 
for all $t>0$,
\begin{equation}
\label{eq:mirror_descent_inexact}
\begin{aligned}
\partial_t Z_t(s,a)  & =  - \left( Q_t(s,a)  + \tau Z_t(s,a)
-\int_A 
(Q_t(s,a')  + \tau Z_t(s,a'))
\pi_t(da'|s)
\right)\,.
\end{aligned} 
\end{equation}
By following similar computations to that of
\eqref{eq:DZtZ_derivative}, for all $t>0$,
\begin{align}
\label{eq:d_tD_stable}
\begin{split}
\partial_t   D_{d_{\rho}^{\pi^*_\tau  }}(Z_t,Z^*)
&=\int_{S} \int_A \partial_t Z_t(s,a)  (\pi_t(da|s)  - \pi^*_\tau (da|s))d^{\pi^*_\tau  }_\rho(ds)
\\
&=
-\int_{S} \int_A (Q_t(s,a)  +\tau Z_t (s,a))  (\pi_t(da|s)  - \pi^*_\tau (da|s))d^{\pi^*_\tau  }_\rho(ds)
\\
&=- \int_{S} \int_A \left(
Q_t(s,a) +\tau \ln\frac{\mathrm{d}  \pi_t}{\mathrm{d}  \mu}(a|s)   \right)  (\pi_t(da|s)  - \pi^*_\tau (da|s))d^{\pi^*_\tau  }_\rho(ds)
\\
&=
\mathcal{E}_t
- \int_{S} \int_A \left(
Q^{\pi_t}_\tau (s,a) +\tau \ln\frac{\mathrm{d}  \pi_t}{\mathrm{d}  \mu}(a|s)   \right)  (\pi_t(da|s)  - \pi^*_\tau (da|s))d^{\pi^*_\tau  }_\rho(ds)\,,
\end{split}
\end{align}
where 
\begin{align}
\label{eq:error_t}
\mathcal{E}_t \coloneqq 
\int_{S} \int_A \left(
Q^{\pi_t}_\tau (s,a) -Q_t(s,a)\right)  (\pi_t(da|s)  - \pi^*_\tau (da|s))d^{\pi^*_\tau  }_\rho(ds)\,.
\end{align}
Substituting  \eqref{eq:performance_gap_pi_t} into \eqref{eq:d_tD_stable} yields  that
for all $t>0$,
\begin{align*}
\partial_t   D_{d_{\rho}^{\pi^*_\tau  }}(Z_t,Z^*)
&=- \left((1-\gamma)
(V^{\pi_t}_{\tau}(\rho)-V^{\pi^*_\tau}_{\tau}(\rho))
+\tau \int_S 
\operatorname{KL}(\pi^*_\tau(\cdot|s)|\pi_t(\cdot|s) )d^{\pi^*_\tau}_{\rho} (ds)
\right)  
+\mathcal{E}_t
\\
&=- (1-\gamma)
(V^{\pi_t}_{\tau}(\rho)-V^{\pi^*_\tau}_{\tau}(\rho))
-\tau  D_{d_{\rho}^{\pi^*_\tau }}(Z_t,Z^*)+\mathcal{E}_t\,,
\end{align*}
where the last line follows from 
Lemma \ref{lemma:D_KL} (with 
$f= Z_t$ and $g=Z^*$).
By Duhamel's principle,
for all $t>0$,
\begin{align*}
\begin{split}
D_{d_{\rho}^{\pi^*_\tau  }}(Z_t,Z^*)&=e^{-\tau t}  D_{d_{\rho}^{\pi^*_\tau  }}(Z_0,Z^*)   
-\int_0^t e^{-\tau (t-r)}\left(
(1-\gamma)(
V^{\pi_{r}}_{\tau}(\rho)-V^{\pi^*_\tau}_{\tau}(\rho))-\mathcal{E}_{r}\right)dr\,.
\end{split}
\end{align*}
Using the fact that 
$D_{d_{\rho}^{\pi^*_\tau  }}(Z_t,Z^*)\ge 0$ for all $t>0$ (cf.~\eqref{eq:D_ZtZ_pinsker}),  
for all $t>0$,
\begin{align*}
\begin{split}
0&\le e^{-\tau t}  D_{d_{\rho}^{\pi^*_\tau  }}(Z_0,Z^*)    
-
(1-\gamma)
\int_0^t e^{-\tau (t-r)}dr\min_{r\in [0,t]}(V^{\pi_{r}}_{\tau}(\rho)-V^{\pi^*_\tau}_{\tau}(\rho))
+\int_0^t e^{-\tau (t-r)}\mathcal{E}_{r}dr\,,
\end{split}
\end{align*}
which along with 
$\int_0^t e^{-\tau(t-r)} dr =\frac{1}{\tau}(1-e^{-\tau t})$
and Lemma \ref{lemma:D_KL}
shows that for all $t>0$,
\begin{align}
\label{eq:stability_estimate}
\begin{split}
\min_{r\in [0,t]}(V^{\pi_{r}}_{\tau}(\rho)-V^{\pi^*_\tau}_{\tau}(\rho)) 
&\le  \frac{\tau  e^{-\tau t}  }{(1-\gamma)(1-e^{-\tau t})} 
\left(
D_{d_{\rho}^{\pi^*_\tau  }}(Z_0,Z^*)    
+\int_0^t e^{\tau r}\mathcal{E}_{r}dr
\right)
\\
&=   \frac{\tau   }{(1-\gamma)(e^{\tau t}-1)} 
\left(
\int_S 
\operatorname{KL}(\pi^*_\tau(\cdot|s)|\pi_0(\cdot|s) )d^{\pi^*_\tau}_{\rho} (ds)   
+\int_0^t e^{\tau r }\mathcal{E}_{r}dr
\right)\,.
\end{split}
\end{align}

It remains to derive an upper bound of  $\mathcal{E}_t$ for all $t>0$.
For each $t>0$,
let $m_t =\rho_{\rm ref}\otimes \frac{\pi_t+ \pi_{\rm ref}}{2}$. 
By Lemma \ref{lemma:absolute_continuity_kernel}, 
$\pi_t, \pi^*_\tau\in  \clP_\mu(A|S)$
and the conditions on $\rho_{\textrm{ref}}$ and $\pi_{\textrm{ref}}$,
$\pi_t$
imply that 
$d^{\pi^*_\tau  }_\rho  \otimes 
\pi^*_\tau 
\ll \rho_{\textrm{ref}}\otimes  \pi_{\textrm{ref}} $
and $d^{\pi^*_\tau  }_\rho  \otimes 
\pi_t \ll \rho_{\textrm{ref}}\otimes  \pi_{\textrm{ref}} $, which along with Lemma \ref{lemma:averaged_measure}
implies that 
$d^{\pi^*_\tau  }_\rho  \otimes  
\pi_t\ll   m_t$
and 
$d^{\pi^*_\tau  }_\rho  \otimes  
\pi^*_\tau\ll   m_t$.
Then by \eqref{eq:error_t},
for all $t>0$,
\begin{align*}
|\mathcal{E}_t| &=
\left|
\int_{S} \int_A \left(
Q^{\pi_t}_\tau (s,a) -Q_t(s,a)\right)  (\pi_t(da|s)  - \pi^*_\tau (da|s))d^{\pi^*_\tau  }_\rho(ds)
\right|
\\
&=
\left|
\int_{S\times A}  \left(
Q^{\pi_t}_\tau (s,a) -Q_t(s,a)\right)  
\left(
\frac{\mathrm{d} d^{\pi^*_\tau  }_\rho  \otimes  
\pi_t  }{\mathrm{d} m_t}
-\frac{\mathrm{d} d^{\pi^*_\tau  }_\rho  \otimes  \pi^*_\tau   }{\mathrm{d} m_t}
\right)(s,a)
m_t (ds, da)
\right|
\\
&\le 
\left\|Q^{\pi_t}_\tau   -Q_t\right\|_{L^1(S\times A, m_t)}
\left(
\left\|
\frac{\mathrm{d} d^{\pi^*_\tau  }_\rho  \otimes  
\pi_t  }{\mathrm{d} m_t}
\right\|_{L^\infty(S\times A, m_t)}
+\left\|
\frac{\mathrm{d} d^{\pi^*_\tau  }_\rho  \otimes 
\pi^*_\tau   }{\mathrm{d} m_t}
\right\|_{L^\infty(S\times A, m_t)}
\right)\,.
\end{align*}
By Lemmas \ref{lemma:averaged_measure}
and \ref{lemma:absolute_continuity_kernel}, 
\begin{align*}
\left\|
\frac{\mathrm{d} d^{\pi^*_\tau  }_\rho  \otimes  
\pi_t  }{\mathrm{d} m_t}
\right\|_{L^\infty(S\times A, m_t)}
& \le 2 \left\|
\frac{\mathrm{d} d^{\pi^*_\tau  }_\rho  \otimes  
\pi_t  }{\mathrm{d} \rho_{\rm ref}\otimes  \pi_t }
\right\|_{L^\infty(S\times A, m_t)}
\le 2 \left\|
\frac{\mathrm{d} d^{\pi^*_\tau  }_\rho    }{\mathrm{d} \rho_{\rm ref}  }
\right\|_{L^\infty(S , \rho_{\rm ref})}\,.
\end{align*}
Similarly, by 
Lemmas \ref{lemma:averaged_measure}
and \ref{lemma:absolute_continuity_kernel}, 
and the fact that $m_t \ll  \rho_{\rm ref}\otimes  \pi_{\textrm{ref}} $,
\begin{align*}
\left\|
\frac{\mathrm{d} d^{\pi^*_\tau  }_\rho  \otimes  
\pi^*_\tau  }{\mathrm{d} m_t}
\right\|_{L^\infty(S\times A, m_t)}
& \le 2 \left\|
\frac{\mathrm{d} d^{\pi^*_\tau  }_\rho  \otimes  
\pi^*_\tau }{\mathrm{d} \rho_{\rm ref}\otimes  \pi_{\textrm{ref}} }
\right\|_{L^\infty(S\times A, m_t)}
\le 2 \left\|
\frac{\mathrm{d} d^{\pi^*_\tau  }_\rho  \otimes  
\pi^*_\tau }{\mathrm{d} \rho_{\rm ref}\otimes  \pi_{\textrm{ref}} }
\right\|_{L^\infty(S\times A, \rho_{\rm ref}\otimes  \pi_{\textrm{ref}})}\,.
\end{align*}
This shows that for all $t>0$,
$|\mathcal{E}_t|\le 2\kappa  \left\|Q^{\pi_t}_\tau   -Q_t\right\|_{L^1(S\times A, m_t)}$,
with $\kappa\ge 1$ being defined in \eqref{eq:concentration_coefficient}. 
Along with \eqref{eq:stability_estimate}, this yields the desired estimate \eqref{eq:stability_estimate_statement}. 
Finally, observe that for any $F\in B_b(S)$,
\begin{align*}
\mathcal{E}_t \coloneqq 
\int_{S} \int_A \left(
Q^{\pi_t}_\tau (s,a) +F(s)-Q_t(s,a)\right)  (\pi_t(da|s)  - \pi^*_\tau (da|s))d^{\pi^*_\tau  }_\rho(ds)\,.
\end{align*}
Proceeding along the above lines  shows  that \eqref{eq:stability_estimate_statement}
holds with $\left\|Q^{\pi_r}_\tau   -Q_r\right\|_{L^1(S\times A, m_r)}$ replaced by
$
\left\|Q^{\pi_r}_\tau +F_r   -Q_r\right\|_{L^1(S\times A, m_r)}$ for any measurable $F:\mathbb{R}_+\to   B_b(S)$.
\end{proof}

\subsection{Proof of Theorem~\ref{thm:NPG_stability}} \label{sec:NPG_proof}
Observe that \eqref{eq:npg_clipping}
can be equivalently rewritten as:
\begin{equation}
\label{eq:npg_clipping_equivalent}
\partial_t {\theta}_t = -\mathfrak{f}(\theta_t,\lambda_t,R_t)- \tau \theta_t,
\quad t>0,
\end{equation}
where
for each  $ R,\lambda >0$ and  $\theta\in \bbH$,
\begin{align}
\label{eq:constrained_loss}
\mathfrak{f}(\theta,\lambda, R)\coloneqq 
\argmin_{\|w\|_{\bbH}\le R} \left(
\int_{S}\int_A | A^{\pi_{\theta}}_{\tau}(s,a) - \langle w,  g_{\pi_{\theta}}(s,a)\rangle_{\bbH}|^2 \pi_{\theta}(da|s)d_{\rho}^{\pi_{\theta}}(ds)+\lambda \|w\|^2_{\bbH}
\right)
\,.
\end{align}
Thus, to prove the well-posedness of \eqref{eq:npg_clipping},
the key step  is to analyse  the regularity of $\mathfrak{f}$ in \eqref{eq:constrained_loss} with respect to $\theta, \lambda $ and $R$.
We start by proving the continuity of the orthogonal projection with respect to the target set. To this end, for any nonempty subsets $X$ and $Y$ of a metric space $(M,d )$, we define their Hausdorff distance by
\[
d_H(X,Y) =\max\left\{\sup_{x\in X} d(x,Y),\, \sup_{y\in Y}d(y,X)\right\}\,,
\]
where $d(a,B)=\inf_{b\in B} d( a, b)$ is the distance of $a \in M $ to the subset $B\subset M$. 

\begin{lemma}
\label{lemma:project_continuity}
Let $(\bbH,\|\cdot\|_\bbH)$ be a Hilbert space 
and $x\in \bbH$.
For each nonempty closed convex set 
$K\subset \bbH$,
let $P_K x=\argmin_{y\in K}\|y-x\|_{\bbH}$.
Then  for all 
nonempty closed convex   subsets
$(K_n)_{n\in \bbN}\subset \bbH$ and $K\subset \bbH$ satisfying 
$\lim_{n\to\infty} d_H(K_n,K)=0$, we have    $\lim_{n\to\infty} \| P_{K_n} x-P_K x\|_{\bbH}=0$.
\end{lemma}

\begin{proof}
Observe that if $x\in K$, then $P_K x= x$,
and   $ \| P_{K_n} x-P_K x\|_{\bbH} = \|x-P_{K_n} x\|_{\bbH}\le d(x,K_n)\le d_H(K_n, K)$. 
Hence, we assume without loss of generality that 
$x\not\in K$, or equivalently $d(x,K)>0$.
For any  $\epsilon>0$,
define 
$ \delta =\min\left\{\frac{\epsilon}{2}, \frac{d(x,K)}{2}, \frac{\epsilon^2}{64d(x,K)}\right\}>0$. Since         $\lim_{n\to\infty} d_H(K_n,K)=0$, 
there exists $N_0\in \bbN$ such that for all $n\ge N_0$, $d_H(K_n,K)< \delta$,
which implies that 
$|d(x,K_n)-d(x,K)|\le d_H(K_n,K)<\delta $.
For each $n\ge N_0$, by $d(P_K x,K_n)\le d_H(K,K_n)< \delta$, 
there exists $z_n\in K_n$ such that 
$\|P_K x-z_n\|_{\bbH}\le \delta $.
Then for all $n\ge N_0$,
\begin{align*}
\|x-P_{K_n} x\|_{\bbH}&= d(x,K_n)
\le d(x,K)+ d_H(K_n,K) 
\le d(x,K)+\delta\,, 
\\
\|x-z_n\|_{\bbH}&\le 
\|x-P_Kx\|_{\bbH}+\|P_Kx-z_n\|_{\bbH}
\le d(x,K)+\delta\,.
\end{align*} 
Suppose by contradiction that 
$\|P_{K_n} x-z_n\|_{\bbH}\ge \frac{\epsilon}{2}$.
Then by the parallelogram law,
\begin{align*}
\left\|x-\frac{P_{K_n}x+z_n}{2}\right\|^2_\bbH
&=\frac{1}{4}\left\|(x-P_{K_n}x)+ (x-z_n) \right\|_{\bbH}^2
\\
&=
\frac{1}{4}\left(  2 \| x-P_{K_n}x\|_{\bbH}^2 
+2\| x-z_n   \|_{\bbH}^2
-\|P_{K_n}x-z_n\|_{\bbH}^2\right)
\\
&\le 
\frac{1}{4}\left(  4 (d(x,K)+\delta)^2  
-\left(\frac{\epsilon}{2}\right)^2
\right)
\le  (d(x,K)-\delta)^2\,,
\end{align*}
where the last inequality used 
$\delta \le  \frac{\epsilon^2}{64d(x,K)}$.
This, along with  the convexity of  $K_n$ 
(i.e., \linebreak $ {(P_{K_n}x+z_n)}/{2}\in K_n$)
and $\delta \le  d(x,K)/2$ 
implies that 
\[ 
d(x,K_n)\le \left\|x-\frac{P_{K_n}x+z_n}{2}\right\|_{\bbH}\le d(x,K)-\delta\, .
\]
However, this contradicts the fact that
$d(x,K_n)\ge d(x,K)-d_H(K_n,K)>d(x,K)-\delta $.
Thus   $\|P_{K_n} x-z_n\|_{\bbH} < \frac{\epsilon}{2}$,
and hence for all $n\ge N_0$,
\[
\|P_K x -P_{K_n} x\|_{\bbH}\le
\|P_K x -z_n \|_{\bbH}+\|z_n- P_{K_n} x\|_{\bbH}
<d_H(K,K_n)+\frac{\epsilon}{2}
\le  \epsilon\,,
\]
where the last inequality used 
$d_H(K,K_n)\le \delta\le \frac{\epsilon}{2}$.
This proves the desired continuity. 
\end{proof}

Based on Lemma \ref{lemma:project_continuity}, the following lemma establishes the stability of a constrained quadratic minimisation problem (cf.~\eqref{eq:constrained_loss}). 
For any self-adjoint operator $G\in \clL({\bbH}) $,
we denote by $\lambda_{\min}(G)$ the minimum eigenvalue of $G$.
\begin{lemma}
\label{lemma:stability_quadratic}
Let $(\bbH, \|\cdot\|_{\bbH}) $
be a Hilbert space with the inner product $\langle\cdot, \cdot\rangle_{\bbH}$.
For each nonempty closed convex   subset $K\subset \bbH$,  self-adjoint operator 
$G\in \clL({\bbH}) $
with $\lambda_{\min}(G)>0$,
and $g\in \bbH$, define
\begin{equation}
\label{eq:quadratic_loss}
x_{G,g,K}\coloneqq \argmin_{x\in K}\left(\frac{1}{2}\langle Gx, x\rangle_{\bbH}+\langle g , x\rangle_{\bbH}\right)\,.
\end{equation}
Then  for all nonempty closed convex   subsets $K\subset \bbH$, self-adjoint operators 
$G,G'\in \clL(\bbH)$ 
with $\lambda_{\min}(G)>0$,
and $g,g'\in \bbH$,  
\begin{enumerate}[(1)]
\item 
\label{item:lipschit_G_g}
$  \|x_{G,g ,K}-x_{G',g',K}\|_{\bbH}
\le \frac{1}{{\lambda_{\min }(G)}} \left(\|G-G'\|_{\clL(\bbH)}\|x_{G',g',K}\|_{\bbH}+\|g-g'\|_{\bbH}\right)$;  
\item 
\label{item:conts_K}
If  $(K_n)_{n\in \bbN}\subset \bbH$ 
are   nonempty closed convex   subsets  such that 
$\lim_{n\to\infty} d_H(K_n,K)=0$,
then $\lim_{n\to\infty} \| x_{G,g ,K_n}-x_{G ,g ,K}\|_{\bbH}=0$.
\end{enumerate}
\end{lemma}
\begin{proof}
To prove   Item \ref{item:lipschit_G_g}, let $x = x_{G,g,K}$ and  $x' = x_{G',g',K}$ . 
The first order condition of \eqref{eq:quadratic_loss} yields that for all $y\in K$,
$ \langle Gx_{G,g,K}+g, y-x_{G,g,K}\rangle_{\bbH}\ge 0$. 
This along with   $x,x'\in K$  implies  
$
\langle G x +g , x'-x \rangle_{\bbH}\ge 0$ and $
\langle G' x'+g', x -x'\rangle_{\bbH}\ge 0$.
Hence 
\[
\langle G x +g -( G' x'+g'), x'-x \rangle_{\bbH}\ge 0\,.
\]
Rearranging the terms yields that 
\begin{align*}
\langle (G-G')x' +g -g', x'-x \rangle_{\bbH}\ge 
\langle G(x- x'), x-x'\rangle_{\bbH}
\ge \lambda_{\min }(G)\|x -x'\|^2_{\bbH}\,,
\end{align*}
which, along with the Cauchy-Schwarz inequality, leads to the desired estimate. 

To prove Item \ref{item:conts_K}, observe by completing the square of \eqref{eq:quadratic_loss} that 
\[
x_{G,g,K}= 
\argmin_{x\in K}\left\|G^{\frac{1}{2}}(x+G^{-1}g)\right\|^2_{\bbH}
=
\argmin_{x\in K}\left\|G^{\frac{1}{2}}(x+G^{-1}g)\right\|_{\bbH} 
= 
\argmin_{x\in K}\left\| x-(-G^{-1}g)\right\|_{G^{\frac{1}{2}}} \,,
\]
where
$G^{\frac{1}{2}}$ is the  positive square root of $G$,
and the norm 
$ \| \cdot  \|_{ G^{\frac{1}{2}}}$
is defined by 
$ \| z  \|_{ G^{\frac{1}{2}}}=\|G^{\frac{1}{2}} z\|_{\bbH}$
for all $z\in  \bbH $. 
As
$G\in \clL(\bbH)$
is self-adjoint and 
$\lambda_{\min}(G)>0$,
$(\bbH, \| \cdot  \|_{ G^{\frac{1}{2}}})$
is a Hilbert space with the inner product
$\langle   x,y\rangle_{G}
=\langle Gx,   y\rangle_{\bbH}$ for all $x,y\in \bbH$,
and   $\|\cdot\|_{G^{\frac{1}{2}}}$ and $\|\cdot\|_{\bbH}$
are equivalent norms on $\bbH$.
Then by Lemma \ref{lemma:project_continuity},
$\lim_{n\to\infty} \| G^{\frac{1}{2}} (x_{G,g ,K_n}-x_{G ,g ,K})\|_{\bbH} =0$,
provided that 
\[
\lim_{n\to\infty}\max\left\{\sup_{x\in K_n} 
\inf_{y\in K}\|G^{\frac{1}{2}} (x-y)\|_{\bbH},
\;
\sup_{y\in K}
\inf_{y\in K_n}\|G^{\frac{1}{2}} (x-y)\|_{\bbH}\right\}=0\,.
\]
This yields the desired convergence in Item \ref{item:conts_K},  due to the equivalence of  $\|\cdot\|_{G^{\frac{1}{2}}}$ and $\|\cdot\|_{\bbH}$.
\end{proof}

The following proposition proves the continuity of $\mathfrak{f}$ defined in \eqref{eq:constrained_loss}. 
\begin{proposition}
\label{prop:lipschitz_npg} 
For each  $ R,\lambda >0$ and  $\theta\in \bbH$,
let   $\mathfrak{f}(\theta,\lambda, R)$
be defined as in  \eqref{eq:constrained_loss}.
Then there exists 
$\mathfrak{m}\in C(\bbR_+;\bbR_+)$
such that for all $R, \lambda,\lambda'>0$ and  $\theta , \theta'\in\bbH$,
\[
\|\mathfrak{f}(\theta,\lambda, R)-\mathfrak{f}(\theta',\lambda', R)\|_{\bbH}
\le \frac{\max(R,1)}{\lambda}
\left(|\lambda-\lambda'| +
\mathfrak{m}( \|\theta\|_{\bbH}+ \|\theta'\|_{\bbH}) \|\theta-\theta'\|_{\bbH}\right)\,,
\]
and the map $(0,\infty)\ni R\mapsto \mathfrak{f}(\theta,\lambda, R)\in \bbH$ is continuous. 
\end{proposition}
\begin{proof}
For each $R>0$,
let $B_R=\{x\in \bbR^N\mid \|x\|_{\bbH}\le R\}$.
By the definition of $f$ in  \eqref{eq:constrained_loss}, 
\begin{align*}
\mathfrak{f}(\theta,\lambda, R)
&=
\argmin_{w\in B_R}
\left(\frac{1}{2}
\langle(G_\theta+\lambda I_{\bbH}) w,
w\rangle_{\bbH}- \langle h_\theta,  w\rangle_{\bbH}\right)\,,
\end{align*}
where  $I_{\bbH}\in \clL(\bbH) $ is the  identity operator, and 
$G_\theta\in 
\clL(\bbH)$ and $h_\theta\in \bbH$
are defined as follows:  
\[
G_\theta = 
\int_S\int_A 
\big(g_{\pi_{\theta}}(s,a)     
\otimes 
g_{\pi_{\theta}}(s,a)
\big)
\pi_{\theta}(da|s)d_{\rho}^{\pi_{\theta}}(ds),
\quad 
h_\theta =  \int_S\int_A g_{\pi_{\theta}}(s,a)  A^{\pi_{\theta}}_{\tau}(s,a)   \pi_{\theta}(da|s)d_{\rho}^{\pi_{\theta}}(ds)\,,
\]
with $ A^{\pi_{\theta}}_{\tau}$ defined in \eqref{eq:loss_Q}.
Then the continuity of $R\mapsto \mathfrak{f}(\theta,\lambda, R)$
follows directly from
Lemma \ref{lemma:stability_quadratic}
and   that 
$ d_H(B_{R'},B_R)\le |R-R'|$
for all $R,R'>0$.
Moreover, 
let  $\theta , \theta'\in \bbH$
and $\lambda,\lambda'>0$ be given. 
By the variational characterization of eigenvalues and symmetric positive semidefinites of $G_{\theta}$,  $\lambda_{\min}(G_\theta+\lambda I_{\bbH})
\ge 
\lambda$. Thus, since $\|\mathfrak{f}(\theta',\lambda',R)\|_{\bbH}\le R$,
by Lemma \ref{lemma:stability_quadratic},
\begin{align*}
\|\mathfrak{f}(\theta,\lambda,R)-\mathfrak{f}(\theta',\lambda',R)\|_{\bbH}
\le 
\frac{1}{\lambda } \left((\|G_\theta-G_{\theta'}\|_{\clL(\bbH) }+|\lambda-\lambda'|)R+\|h_\theta -h_{\theta'}\|_{\bbH}\right)\,.    
\end{align*}

It remains to  prove that 
there exists $\mathfrak{m}\in C(\bbR_+;\bbR_+)$, independent of $\lambda$ and $R$, such that 
for all   $\theta,  \theta'\in \bbH$,
$
\|h_\theta -h_{\theta'}\|_{\bbH} 
+\|G_\theta -G_{\theta'}\|_{\clL(\bbH) }
\le \mathfrak{m}( \|\theta\|_{\bbH}+ \|\theta'\|_{\bbH}) \|\theta-\theta'\|_{\bbH}$.
Observe that 
\begin{align*}
\begin{split}
\left\|h_{\theta} - h_{\theta'}\right\|_{\bbH}
&\le
\left\|\int_S\int_A g_{\pi_{\theta}}(s,a) A^{\pi_{\theta}}_{\tau}(s,a)\pi_{\theta}(da|s)(d^{\pi_{\theta}}_{\rho} - d^{\pi_{\theta'}}_{\rho})(ds)\right\|_{\bbH}
\\
& \quad + \left\|\int_S\int_A g_{\pi_{\theta}}(s,a) A^{\pi_{\theta}}_{\tau}(s,a)(\pi_{\theta} - \pi_{\theta'})(da|s) d^{\pi_{\theta'}}_{\rho}(ds)\right\|_{\bbH}
\\
&\quad  + \left\|\int_S\int_A g_{\pi_{\theta}}(s,a) \left(A^{\pi_{\theta}}_{\tau}(s,a) - A^{\pi_{\theta'}}_{\tau}(s,a)\right) \pi_{\theta'}(da|s) d^{\pi_{\theta'}}_{\rho}(ds)\right\|_{\bbH}
\\
&\quad  + \left\|\int_S\int_A (g_{\pi_{\theta}}(s,a) - g_{\pi_{\theta'}}(s,a) )  A^{\pi_{\theta'}}_{\tau}(s,a) \pi_{\theta'}(da|s) d^{\pi_{\theta'}}_{\rho}(ds)\right\|_{\bbH}
\\
&\le \|g_{\pi_{\theta}}\|_{B_b(S\times A;\bbH)}  \|A^{\pi_{\theta}}_{\tau}\|_{B_b(S\times A)} \left\|d^{\pi_{\theta}}_{\rho} - d^{\pi_{\theta'}}_{\rho}\right\|_{\clM(S)}
\\
&\quad  +  \|g_{\pi_{\theta}}\|_{B_b(S\times A;\bbH)}  \|A^{\pi_{\theta}}_{\tau}\|_{B_b(S\times A)} \left\|\pi_{\theta} - \pi_{\theta'}\right\|_{b\clM(A|S)}
\\
&\quad +  \|g_{\pi_{\theta}}\|_{B_b(S\times A;\bbH)} \left\|A^{\pi_{\theta}}_{\tau} - A^{\pi_{\theta'}}_{\tau}\right\|_{B_b(S\times A)}
+  \|g_{\pi_{\theta}} - g_{\pi_{\theta'}}\|_{B_b(S\times A;\bbH)} \left\| A^{\pi_{\theta'}}_{\tau}\right\|_{B_b(S\times A)}\,.
\end{split}
\end{align*}
By the definition of $g_{\pi_{\theta}}$ in \eqref{eq:linear_policy},  
$\|g_{\pi_{\theta}}\|_{B_b(S\times A; \bbH)} \le 2  \|g\|_{B_b(S\times A; \bbH)}$
and 
\[
\|g_{\pi_{\theta}} - g_{\pi_{\theta'}}\|_{B_b(S\times A; \bbH)} \le \|g\|_{B_b(S\times A; \bbH)} \left\|\pi_{\theta} - \pi_{\theta'}\right\|_{b\clM(A|S)}\,.
\]
By the definition of $A^{\pi_\theta}_\tau $ in  \eqref{eq:loss_Q}, 
$
\|A^{\pi_{\theta}}_\tau\|_{B_b(S\times A)} \le 2 \|Q^{\pi_{\theta}}_\tau\|_{B_b(S\times A)}$
and
\[
\|A^{\pi_{\theta}}_\tau - A^{\pi_{\theta'}}_\tau\|_{B_b(S\times A)} \le 2 \|Q^{\pi_{\theta}}_\tau - Q^{\pi_{\theta'}}_\tau\|_{B_b(S\times A)} + \|Q^{\pi_{\theta'}}_\tau\|_{B_b(S\times A)} \left\|\pi_{\theta} - \pi_{\theta'}\right\|_{b\clM(A|S)}\,.
\]
Propositions \ref{prop:boundedness_Q}
and \ref{lemma:Q_local_lipschitz}
and the Cauchy-Schwarz inequality implies that 
\begin{align*}
\|Q^{\pi_{\theta}}_\tau\|_{B_b(S\times A)} &\le \frac{1}{1-\gamma}\left(\|c\|_{B_b(S\times A)} + 2\tau \gamma\|g\|_{B_b(S\times A;\bbH)}\|\theta\|_{\bbH}\right)\,, \\
\left\|Q^{\pi_{\theta}}_\tau -Q^{\pi_{\theta'}}_\tau \right\|_{B_b(S\times A)}&\le 
\frac{2\gamma^2 }{(1-\gamma)^2}\big(\|c\|_{B_b(S\times A)}+2\tau \|g\|_{B_b(S\times A; \bbH)} (\|\theta\|_{\bbH}+\|\theta'\|_{\bbH}) \big)\\
&\qquad \qquad \qquad  \times \|g\|_{B_b(S\times A;\bbH)} \|\theta - \theta'\|_{\bbH}\,.
\end{align*}
By \cite[Lemma 5]{leahy2022convergence},  
\begin{align*}
\left\|d^{\pi_{\theta}}_{\rho} - d^{\pi_{\theta'}}_{\rho}\right\|_{\clM(S)} \le \|d^{\pi_{\theta}} - d^{\pi_{\theta'}}\|_{b\clM(S|S)} \le \frac{\gamma}{1-\gamma}\|\pi_{\theta} - \pi_{\theta'}\|_{b\clM(A|S)}
\,.
\end{align*}
Finally, by Proposition \ref{prop:differentiability_Pi} and the mean value theorem,
\begin{equation*}
\begin{split}
&\left\|\pi_{\theta} - \pi_{\theta'}\right\|_{b\clM(A|S)} 
\\
&\quad \le \sup_{\zeta \in [0,1]}  \left\| \mathfrak{d} \boldsymbol{\pi}(
\langle \theta'+\zeta(\theta-\theta'), g\rangle_{\bbH}) 
\right\|_{\clL(B_b(S\times A), b\clM(A|S))} \|g\|_{B_b(S\times A; \bbH)} \|\theta-\theta'\|_{\bbH}\\
&\quad \le 2 \|g\|_{B_b(S\times A;\bbH)} \|\theta-\theta'\|\,.
\end{split}
\end{equation*}
Combining the above estimates 
yields the desired locally Lipschitz continuity of $h_\theta$. 
The  estimate of  $
\|G_\theta -G_{\theta'}\|_{\clL(\bbH) }$
follows from similar arguments,
and hence is omitted. This finishes the proof. 
\end{proof}

We are now ready to prove Theorem \ref{thm:NPG_stability}. 
\begin{proof}[Proof of Theorem~\ref{thm:NPG_stability}]
Let  
$R,\lambda\in C(\mathbb{R}_+; (0,\infty))$, $\bar{\theta}\in \bbH$ and $\rho \in \clP(S)$ be fixed.  
We first prove 
\eqref{eq:npg_clipping_equivalent}
(or equivalently \eqref{eq:npg_clipping}) 
has a unique solution on $[0,T]$
for any $T>0$.
Indeed, 
for any given 
$T>0$, consider the following dynamics 
\begin{equation}
\label{eq:npg_T}
\partial_t \theta_t = \mathcal{F}(t, \theta_t),\quad t\in (0,T];
\quad \theta_0=\bar{\theta}\,,
\end{equation}
where for all $t\in [0,T]$ and $\theta\in \bbH$,
\[
\mathcal{F}(t, \theta)
\coloneqq -\mathfrak{f}(\theta, \lambda_t,R_t)- \tau \theta\,,
\]
with $\mathfrak{f}$
defined by  \eqref{eq:constrained_loss}.
As $R,\lambda\in C([0,T]; (0,\infty))$, 
there exists $r_T>0$ and $M_T\ge 1 $  such that for all $t\in [0,T]$,
$\lambda_t \ge r_T>0$ and $0<R_t\le M_T$.
Then by Proposition \ref{prop:lipschitz_npg},
for all $\theta\in \bbH$, 
$\mathcal{F}(t, \theta)
\in C([0,T]; \bbH)$, 
and there exists $\mathfrak{m}\in C(\bbR_+;\bbR_+)$ such that for all $\theta,\theta' \in \bbH$, 
\begin{align}
\label{eq:npg_locally_lipschitz_T}
\|\mathcal{F}(t, \theta)
-\mathcal{F}(t, \theta')
\|_{\bbH}\le \frac{CM_T}{r_T}\mathfrak{m}( \|\theta\|_{\bbH}+ \|\theta'\|_{\bbH}) \|\theta-\theta'\|_{\bbH}\,.
\end{align}
Observe that   the definition 
\eqref{eq:constrained_loss} 
of  $\mathfrak{f}$ implies that 
if $\theta \in C^1([0,T];\bbH)$,
then for all $t\in [0,T]$,
\begin{align}
\label{eq:bound_theta_T}
\|\theta_t\|_{\bbH}=
\left\|e^{-\tau t} \bar{\theta}-
e^{-\tau t} \int_0^t e^{\tau r} 
\mathfrak{f}(\theta_r, \lambda_r,R_r)dr 
\right\|_{\bbH}
\le \|\bar{\theta}\|_{\bbH}
+e^{-\tau t}\int_0^t e^{\tau r} 
R_r dr
\le 
\|\bar{\theta}\|_{\bbH}
+\frac{M_T}{\tau }\,.
\end{align}
Combining the locally Lipschitz estimate \eqref{eq:npg_locally_lipschitz_T}, 
the a-priori bound  \eqref{eq:bound_theta_T}
and a truncation argument 
as in the proof of   Theorem \ref{thm:wp_MD}
yields that 
\eqref{eq:npg_T}
admits a unique solution in 
$C^1([0,T];\bbH)$.
As $T>0$ is arbitrarily given, 
\eqref{eq:npg_clipping} admits     a unique solution in 
$C^1(\bbR_+;\bbH)$. 

By Lemma  \ref{lemma:chain_rule}
and Proposition \ref{prop:differentiability_Pi},
$(\pi_{\theta_t})_{t\ge 0}$
is in $C^1(\mathbb{R}_+; \Pi_\mu)$
and 
satisfies 
\begin{align}
\label{eq:npg_Fisher_Rao}
\begin{split}
&  \partial_t \pi_{\theta_t} (da|s)
=
(\mathfrak{d}\boldsymbol{\pi}(\langle \theta_t, g\rangle_{\bbH})(
\langle\partial_t \theta _t,  g\rangle_{\bbH})(da|s)
\\
&\quad =
\left(
\langle \partial_t \theta _t, g(s,a)
\rangle_{\bbH}
-
\int_A 
\langle
\partial_t \theta_t,    g(s,a') 
\rangle_{\bbH}
\pi_{\theta_t}  (da'|s)
\right)
\pi_{\theta_t} (da|s)
\\
&\quad  =
-     \left(
\langle w_t+\tau   \theta _t, g(s,a)
\rangle_{\bbH}
-
\int_A 
\langle
w_t+\tau   \theta _t,    g(s,a') 
\rangle_{\bbH}
\pi_{\theta_t}  (da'|s)
\right)
\pi_{\theta_t} (da|s)
\\
&\quad  =   - \left(
Q_t(s,a)  +\tau   \ln\frac{\mathrm{d}  \pi_{\theta_t}}{\mathrm{d}  \mu}(a|s)
-
\int_A 
\left(Q_t(s,a')  +\tau   \ln\frac{\mathrm{d}  \pi_{\theta_t}}{\mathrm{d}  \mu}(a'|s)
\right)
\pi_{\theta_t}  (da'|s)
\right)
\pi_{\theta_t} (da|s)   
\end{split}
\end{align}
with $Q_t (\cdot) \coloneqq \langle w_t,  g(\cdot)\rangle_{\bbH}$,
where the last line 
used  
\[
\ln\frac{\mathrm{d}  \pi_{\theta_t}}{\mathrm{d}  \mu}(a|s)= \langle \theta_t,  g(s,a) \rangle_{\bbH}   -
\ln
\left( {\int_A e^{
\langle \theta_t,  g(s,a') \rangle_{\bbH} 
}\mu(d a')} \right).
\]
The estimate \eqref{eq:stability_estimate_NPG}
then follows from 
\eqref{eq:stability_estimate_statement}
with 
$Q^{\pi_{\theta_r}}_\tau   -Q_r$
replaced by 
$
A^{\pi_{\theta_r}}_{\tau}   -
\langle w_r,  g_{\pi_{\theta_r} }\rangle_{\bbH}
=
Q^{\pi_{\theta_r}}_\tau   -Q_r
+F_r$,
where 
\[
F_r(s) = -\int_{A} (Q^{\pi_{\theta_r}}_{\tau}(s,a')
-
\langle
w_r,  g(s,a')\rangle_{\bbH})
\pi_{\theta_r}(da'|s),
\quad \forall s\in S\,.
\]
This finishes the proof.
\end{proof}

\subsection{Proof of  Theorem \ref{thm:convergence_mirror_steps}}
\label{sec:proof_mirror_descent}
To ease the notation we will write $\operatorname{KL}(\pi|\pi')(s) := \operatorname{KL}(\pi(\cdot|s)|\pi'(\cdot|s))$ whenever $\pi,\pi' \in \mathcal P(A|S)$ and $s\in S$.
Let us recall that due to~\eqref{eq:delta_V_delta_pi-intro-state-dep-pairing} we have $\frac{\delta V^{\pi}_{\tau}}{\delta \pi} = Q^{\pi}_\tau  + \tau \ln\frac{\mathrm{d}  \pi}{\mathrm{d}  \mu}  -V^{\pi}_\tau\,.
$ 
To prove Theorem \ref{thm:convergence_mirror_steps},
we introduce the set  $M_\mu = \{\mathbf m(f) | f \in B_b(A) \}\subset \mathcal P(A)$, 
where $\boldsymbol{m} :B_b(A) \to \clP_{\mu}(A)$ is defined by 
\begin{equation}
\label{eq:M_f_mu}
\boldsymbol{m}(f)(da)= \frac{e^{f(a)}}{\int_A e^{f(a')}\mu(d a')}\mu(d a),
\quad \forall f\in B_b(A) \,.    
\end{equation}
Notice that $M_\mu$ is a convex subspace of $\mathcal P(A)$. 
The operator $\boldsymbol{m}$ is   analogous to the mirror map $\boldsymbol{\pi}:B_b(S\times A)\to \mathcal P(A|S)$.
We will need the three point lemma, which is proved e.g., in~\cite{aubin2022mirror} noting that the flat derivative of KL is well defined on $M_\mu$, see e.g.~\cite[Lemma 3.8]{kerimkulov2024mirror}.
\begin{lemma}[Three point lemma\,/\,Bregman proximal inequality] \label{lem three point}
Let $G:M_\mu \rightarrow \mathbb R$ be convex.
For all $m' \in M_\mu$ let 
\begin{equation}
m^\ast = \argmin_{m \in M_\mu}\left\{ G(m) + \operatorname{KL}(m|m') \right\}\,.
\end{equation}
Then, for all $m\in M_\mu$, we have 
\begin{equation}
G(m) + \operatorname{KL}(m|m') \geq G(m^\ast) + \operatorname{KL}(m|m^\ast) + \operatorname{KL}(m^\ast|m')\,.
\end{equation}
\end{lemma}

Next, we observe that the performance difference lemma shows policy improvement.

\begin{lemma}[Policy improvement]
\label{lemma:improve}
Let $V^n_\tau := V^{\pi^n}_\tau$ for $n\in \mathbb N$ and $\pi^n \in \Pi_\mu$ given by~\eqref{eq:mirror_des_direct_appendix}.
If $\tau \leq \lambda$ then for any $\rho \in \mathcal P(S)$ we have $V^{n+1}_\tau(\rho) \leq  V^n_\tau(\rho)$. 
\end{lemma}
\begin{proof}
From the performance difference lemma, see Lemma~\eqref{lem:performance_diff}, we see that 
\begin{equation}
\label{eq:imp_proof_1}
\begin{split}
(V^{n+1}_\tau - V^n_\tau)(\rho) & = \frac{1}{1-\gamma} \int_S \bigg(\int_A \frac{\delta V^n_\tau}{\delta \pi}(s,a)(\pi^{n+1}-\pi^n)(da|s) + \tau \operatorname{KL}(\pi^{n+1}|\pi^n)(s) \bigg)\,d^{\pi^{n+1}}_\rho (ds)\\
& \leq \frac{1}{1-\gamma} \int_S \bigg(\int_A \frac{\delta V^n_\tau}{\delta \pi}(s,a)(\pi^{n+1}-\pi^n)(da|s) + \lambda \operatorname{KL}(\pi^{n+1}|\pi^n)(s) \bigg)\,d^{\pi^{n+1}}_\rho (ds)\,.
\end{split}
\end{equation}	
From the mirror descent update~\eqref{eq:pointwise_min} we have, for all $\pi \in \Pi_\mu$ and $s\in S$ that
\begin{equation*}
\begin{split}
& \int_A \frac{\delta V^n_\tau}{\delta \pi}(s,a)(\pi-\pi^n)(da|s) + \lambda \operatorname{KL}(\pi|\pi^n)(s)\\
& \geq \int_A \frac{\delta V^n_\tau}{\delta \pi}(s,a)(\pi^{n+1}-\pi^n)(da|s) +  \lambda \operatorname{KL}(\pi^{n+1}|\pi^n)(s)\,.	
\end{split}
\end{equation*} 
This with $\pi = \pi^n$ allows us to conclude that for all $s\in S$ we have 
\begin{equation}
\label{eq:imp_proof_2}
\int_A \frac{\delta V^n_\tau}{\delta \pi}(s,a)(\pi^{n+1}-\pi^n)(da|s) +  \lambda \operatorname{KL}(\pi^{n+1}|\pi^n)(s) \leq 0\,.
\end{equation}
This with~\eqref{eq:imp_proof_1} concludes the proof. 
\end{proof}

The following L-smoothness is analogous to one established by~\cite{lan2022policy,lan2023policy}.
\begin{lemma}[L-smoothness]
\label{lemma:pointwise_estimate}
Let $\pi,\pi' \in \Pi_\mu$
satisfy  
$\int_A \frac{\delta V^{\pi'}_\tau}{\delta \pi}(s,a)(\pi-\pi')(da|s) + \tau \operatorname{KL}(\pi|\pi')(s) \leq 0$ for all 
$s\in S$.
Then for all $s\in S$,
\begin{equation*}
(V^{\pi}_\tau - V^{\pi'}_\tau)(s) \leq \int_A \frac{\delta V^{\pi'}_\tau}{\delta \pi}(s,a)(\pi-\pi')(da|s) + \tau \operatorname{KL}(\pi|\pi')(s) \,.
\end{equation*}	
\end{lemma}
\begin{proof}
Using Lemma~\ref{lem:performance_diff} (performance difference) we get for all $s\in S$ that 
\begin{equation}
\label{eq:sharper_proof_1}
(V^{\pi}_\tau - V^{\pi'}_\tau)(s) = \frac{1}{1-\gamma} \int_S F(s') \, d^{\pi}_s (ds')\,,
\end{equation}
where $F(s) := \int_A \frac{\delta V^{\pi'}_\tau}{\delta \pi}(s,a)(\pi-\pi')(da|s) + \tau \operatorname{KL}(\pi|\pi')(s)$.
From~\eqref{eq:occupancy_s} and~\eqref{eq:sharper_proof_1} we have  for all $s\in S$ that
\begin{equation*}
\!\!(V^{\pi}_\tau - V^{\pi'}_\tau)(s)\leq \int_S F(s') P^0_{\pi}(ds'|s) + \sum_{k=1}^\infty \int_S \gamma^k F(s') P^k_{\pi}(ds'|s)	 \leq \int_S F(s') \delta_s(ds') = F(s)\,.
\end{equation*} 
This concludes the proof.
\end{proof}

\begin{proof}[Proof of Theorem \ref{thm:convergence_mirror_steps}]

Let   $V^n_\tau := V^{\pi^n}_\tau$ for $n\in \mathbb N$.
We begin with an application of Bregman proximal inequality, see Lemma~\ref{lem three point}.
Fix $s\in S$ and $\pi^n \in \Pi_\mu$ and define $G:M_\mu\to \mathbb R$ by 
\begin{equation*}
G(m) = \frac1\lambda \int_A \frac{\delta V^n_\tau}{\delta \pi}(s,a)(m(da)-\pi^n(da|s))\,.
\end{equation*}
It is linear and thus clearly convex and hence due to the mirror descent update~\eqref{eq:pointwise_min} we have, for all $\pi \in \Pi_\mu$, $s\in S$ and $n\in \mathbb N$ that 
\begin{equation*}
\begin{split}
& \frac1\lambda\int_A \frac{\delta V^n_\tau}{\delta \pi}(s,a)(\pi-\pi^n)(da|s) +  \operatorname{KL}(\pi|\pi^n)(s)\\
& \geq \frac1\lambda\int_A \frac{\delta V^n_\tau}{\delta \pi}(s,a)(\pi^{n+1}-\pi^n)(da|s) + \operatorname{KL}(\pi|\pi^{n+1})(s) + \operatorname{KL}(\pi^{n+1}|\pi^n)(s)\,.	
\end{split}
\end{equation*} 
Re-arranging this leads to 
\begin{equation}
\label{eq:proof_of_convergence_after_3_point}
\begin{split}
& \operatorname{KL}(\pi|\pi^{n+1})(s) - \operatorname{KL}(\pi|\pi^n)(s)\\
& \leq \frac1\lambda \int_A \frac{\delta V^n_\tau}{\delta \pi}(s,a)(\pi-\pi^n)(da|s) - \frac1\lambda \int_A \frac{\delta V^n_\tau}{\delta \pi}(s,a)(\pi^{n+1}-\pi^n)(da|s)  - \operatorname{KL}(\pi^{n+1}|\pi^n)(s)\,.\\
\end{split}
\end{equation} 
From this with $\pi = \pi^\ast_\tau$ and from Lemma~\ref{lemma:pointwise_estimate} {
applied to $\pi^{n+1}$ and $\pi^n$ which satisfy~\eqref{eq:imp_proof_2}} we thus have, for all $s\in S$, that 
\begin{equation}
\label{eq:proof_of_convergence_after_3_mdp_step}
\begin{split}
\operatorname{KL}(\pi^\ast_\tau|\pi^{n+1})(s) - \operatorname{KL}(\pi^\ast_\tau|\pi^n)(s)
& \leq \frac1\lambda \int_A \frac{\delta V^n_\tau}{\delta \pi}(s,a)(\pi^\ast_\tau-\pi^n)(da|s) - \frac1\lambda (V^{n+1}_\tau - V^n_\tau)(s) \\
& \qquad  - \frac1\lambda(\lambda - \tau)\operatorname{KL}(\pi^{n+1}|\pi^n)(s)\,.\\
\end{split}
\end{equation} 
As $\tau \leq \lambda$ the last KL term above is negative and we can drop it.
Summing up over $n=0,1,\ldots, N-1$ we see (spotting the telescoping sums) that for all $s\in S$,
\begin{equation}
\label{eq:convergence_proof_1}
\begin{split}
& \operatorname{KL}(\pi^\ast_\tau|\pi^{N})(s) - \operatorname{KL}(\pi^\ast_\tau|\pi^0)(s)
\leq \sum_{n=0}^{N-1}\frac1\lambda \int_A \frac{\delta V^n_\tau}{\delta \pi}(s,a)(\pi^\ast_\tau-\pi^n)(da|s) - \frac1\lambda (V^N_\tau - V^0_\tau)(s)\,.\\
\end{split}
\end{equation} 
Notice that $V^N_\tau(s) \geq V^\ast_\tau(s)$ and so $(V^N_\tau-V^0_\tau)(s) \geq (V^\ast_\tau-V^0_\tau)(s)$ for all $N\in \mathbb N$.
Let 
\begin{equation*}
y^n := \int_S \operatorname{KL}(\pi^\ast_\tau|\pi^n)(s) d^{\pi^\ast_\tau}_\rho(ds)\,\,\,\text{and}\,\,\, \alpha := - \int_S (V^\ast_\tau-V^0)(s)	d^{\pi^\ast_\tau}_\rho(ds)
\end{equation*}
so that, after integrating~\eqref{eq:convergence_proof_1} over $d^{\pi^\ast_\tau}_\rho$ we have
\begin{equation*}
\label{eq:convergence_proof_2}
y^N - y^0
\leq \sum_{n=0}^{N-1}\frac1\lambda \int_S \int_A \frac{\delta V^n_\tau}{\delta \pi}(s,a)(\pi^\ast_\tau-\pi^n)(da|s)d^{\pi^\ast_\tau}_\rho(ds) + \frac\alpha\lambda\,.
\end{equation*} 
Using the performance difference lemma, see Lemma~\ref{lem:performance_diff}, we get
\begin{equation*}
\label{eq:convergence_proof_3}
y^N - y^0
\leq \sum_{n=0}^{N-1}\bigg[\frac{1-\gamma}{\lambda}(V^{\pi^\ast_\tau} - V^{\pi^n})(\rho) - \frac\tau\lambda \int_S \operatorname{KL}(\pi^\ast_\tau|\pi^n)(s) d^{\pi^\ast_\tau}_\rho(ds) \bigg]   + \frac\alpha\lambda\,.
\end{equation*} 
Since $V^{\pi^N}_\tau(\rho) \leq V^{\pi^n}_\tau(\rho)$ for all $n = 0,1,\ldots,N$ and since $\operatorname{KL}(\cdot | \cdot) \geq 0$ we get that  
\begin{equation*}
y^N - y^0
\leq N\frac{1-\gamma}{\lambda}(V^{\pi^\ast_\tau}_\tau - V^{\pi^{N}}_\tau)(\rho) + \frac\alpha\lambda\,.	
\end{equation*}
Hence
\begin{equation*}
N\frac{1-\gamma}{\lambda}(V^{\pi^{N}}_\tau-V^{\pi^\ast_\tau}_\tau)(\rho) 
\leq  \frac\alpha\lambda + y^0	
\end{equation*}
and so
\begin{equation*}
0 \leq (V^{\pi^{N}}_\tau - V^{\pi^\ast_\tau}_\tau)(\rho) \leq (1-\gamma)^{-1}(\alpha + \lambda y^0)N^{-1}\,.
\end{equation*}
This completes the proof of~\eqref{eq:linear_convergence}. 

We will now prove that~\eqref{eq:mirror_exp_conv_with_dist_mismatch_coeff} holds. 
Define  $\xi := (1-\gamma)^{-1}\big\|\frac{\mathrm d d^{\pi^\ast_\tau}_\rho}{\mathrm d \rho}\big\|_{B_b(S)}$.
Note that $\frac{\mathrm d d^{\pi}_\rho}{\mathrm d \rho}(s)\geq 1-\gamma$ for any $\pi \in \mathcal P(A|S)$ and $s\in S$. 
Hence 
$
\frac{\mathrm d d^{\pi^\ast_\tau}_\rho}{\mathrm d \rho}(s) = \frac{\mathrm d d^{\pi^\ast_\tau}_\rho}{\mathrm d d^\pi_\rho}(s) \frac{\mathrm d d^{\pi}_\rho}{\mathrm d \rho}(s) \geq \frac{\mathrm d d^{\pi^\ast_\tau}_\rho}{\mathrm d d^\pi_\rho}(s)(1-\gamma)
$ for any $\pi \in \mathcal P(A|S)$ and $s\in S$.
We thus have $\vartheta_n := \big\|\frac{\mathrm d d^{\pi^\ast_\tau}_\rho}{\mathrm d d^{\pi^n}_\rho}\big\|_{B_b(S)} \le  \xi$ for all $n\in \mathbb N$.
Observe that
\[
\begin{split}
& \int_S \bigg(\int_A \frac{\delta V^n_\tau}{\delta \pi}(s,a)(\pi^{n+1}-\pi^n)(da|s) +\lambda  \operatorname{KL}(\pi^{n+1}|\pi^n)(s)\bigg)\,d^{\pi^\ast_\tau}_\rho(ds)\\
& = \int_S \bigg(\int_A \frac{\delta V^n_\tau}{\delta \pi}(s,a)(\pi^{n+1}-\pi^n)(da|s) +\lambda  \operatorname{KL}(\pi^{n+1}|\pi^n)(s)\bigg)\frac{\mathrm d d^{\pi^\ast_\tau}_\rho}{\mathrm d d^{\pi^{n+1}}_\rho}(s)\,d^{\pi^{n+1}}_\rho(ds)\\
& \geq \bigg\|\frac{\mathrm d d^{\pi^\ast_\tau}_\rho}{\mathrm d d^{\pi^{n+1}}_\rho}\bigg\|_{B_b(S)}\int_S \bigg(\int_A \frac{\delta V^n_\tau}{\delta \pi}(s,a)(\pi^{n+1}-\pi^n)(da|s) +\lambda  \operatorname{KL}(\pi^{n+1}|\pi^n)(s)\bigg)\,d^{\pi^{n+1}}_\rho(ds)	\\
& \geq \vartheta^{n+1}\int_S \bigg(\int_A \frac{\delta V^n_\tau}{\delta \pi}(s,a)(\pi^{n+1}-\pi^n)(da|s) +\tau  \operatorname{KL}(\pi^{n+1}|\pi^n)(s)\bigg)\,d^{\pi^{n+1}}_\rho(ds)\\
& = \vartheta^{n+1}(1-\gamma)\Big(V^{n+1}-V^n\Big)(\rho)\,,
\end{split}
\]
where the first inequality is due the integrand in the $s\in S$ variable being non-positive, see~\eqref{eq:imp_proof_2}, the second inequality is just that $\lambda \geq \tau$ and $\operatorname{KL}(\cdot|\cdot)\geq 0$ and the final equality is performance difference, i.e., Lemma~\ref{lem:performance_diff}. 
Due to this and~\eqref{eq:proof_of_convergence_after_3_point} we have  
\begin{equation*}
\begin{split}
& \lambda\int_S\operatorname{KL}(\pi^\ast|\pi^{n+1})(s)d^{\pi^\ast_\tau}_\rho(ds) - \lambda\int_S\operatorname{KL}(\pi^\ast|\pi^n)(s)d^{\pi^\ast_\tau}_\rho(ds)\\
& \leq \int_S \int_A \frac{\delta V^n_\tau}{\delta \pi}(s,a)(\pi^\ast-\pi^n)(da|s)\,d^{\pi^\ast_\tau}_\rho(ds) - (1-\gamma)\vartheta_{n+1}\Big(V^{n+1}_\tau-V^n_\tau\Big)(\rho)\,.\\
\end{split}
\end{equation*}
One more application of Lemma~\ref{lem:performance_diff} leads to
\begin{equation*}
\begin{split}
& \lambda\int_S\operatorname{KL}(\pi^\ast_\tau|\pi^{n+1})(s)d^{\pi^\ast_\tau}_\rho(ds) - \lambda\int_S\operatorname{KL}(\pi^\ast_\tau,\pi^n)(s)d^{\pi^\ast_\tau}_\rho(ds)\\
& \leq (1-\gamma)(V^\ast_\tau - V^n_\tau)(\rho) - \tau \int_S \operatorname{KL}(\pi^\ast_\tau|\pi^n)\,d^{\pi^\ast_\tau}_\rho(ds) - (1-\gamma)\vartheta_{n+1}(V^{n+1}_\tau-V^n_\tau)(\rho)\,.\\
\end{split}
\end{equation*}
Writing $y^n := \int_S \operatorname{KL}(\pi^\ast_\tau|\pi^n)(s) d^{\pi^\ast_\tau}_\rho(ds)$ and $\delta^n := (V^n_\tau - V^\ast_\tau)(\rho)$ we thus have 
\[
\lambda y^{n+1} - \lambda y^n  \leq -(1-\gamma)\delta^n - \tau y^n - (1-\gamma)\vartheta^{n+1}\delta^{n+1} + (1-\gamma)\vartheta^{n+1}\delta^n\,.
\]
Hence
\[
\delta^n + \vartheta^{n+1}(\delta^{n+1} - \delta^n) \leq  \frac{\lambda-\tau}{1-\gamma} y^n - \frac{\lambda}{1-\gamma} y^{n+1} \,.
\]
Further note  that $\delta^{n+1} - \delta^n = (V^{n+1}_\tau - V^n_\tau)(\rho)\leq 0$ due to the policy improvement property of the scheme, Lemma~\ref{lemma:improve}, and hence 
\[
\delta^n + \xi(\delta^{n+1} - \delta^n) \leq  \frac{\lambda-\tau}{1-\gamma} y^n - \frac{\lambda}{1-\gamma} y^{n+1} \,.
\]
We can divide be $\xi>0$, re-arrange and get that
\[
\delta^{n+1} + \frac{\lambda}{(1-\gamma)\xi}y^{n+1} \leq \frac{\xi-1}{\xi}\bigg(\delta^n +\frac{\lambda-\tau}{(1-\gamma)(\xi-1)}y^n\bigg)\,. 
\]
Recall our assumption that $\lambda \leq \xi \tau$ which implies that $\lambda-\tau \leq \lambda \frac{\xi-1}{\xi}$.
Then 
\[
\delta^{n+1} + \frac{\lambda}{(1-\gamma)\xi}y^{n+1} \leq \frac{\xi-1}{\xi}\bigg(\delta^n +\frac{\lambda}{(1-\gamma)\xi}y^n\bigg)\,. 
\]
We can iterate this recursive inequality and get~\eqref{eq:mirror_exp_conv_with_dist_mismatch_coeff}.
\end{proof}

\appendix

\section{Policy gradient formula for general state and action spaces}
\label{sec:policy_gradient}

This section proves the policy gradient formula \eqref{eq:delta_V_delta_pi-intro} for Polish state and action spaces and compares it with existing policy gradient formulas for tabular MDPs.

\begin{proposition}
\label{prop:pg}
Let  $\tau\ge 0$ and  $\rho\in \clP(S)$. 
For all
$\pi,\pi'\in \Pi_\mu \subset \clP(A|S)$ (cf.~Definition \ref{def:pi_mu}),
\begin{align}
\label{eq:pg_normalized}
\begin{split}
& \lim_{\varepsilon\searrow 0}\frac{V^{(1-\varepsilon)\pi+\varepsilon\pi'}_\tau(\rho) - V^\pi_\tau(\rho)}{\varepsilon}
\\
&\quad = \frac{1}{1-\gamma}\int_S \int_A\left(Q^{\pi}_{\tau}(s,a)+\tau \ln \frac{\mathrm d \pi}{\mathrm d\mu}(a|s)-V^\pi_\tau(s)\right)(\pi'-\pi)(da|s)d^{\pi}_{\rho}(ds)\,.    
\end{split}
\end{align}
Consequently, given  $\nu\in \clP(S)$
satisfying $d^{\pi}_{\rho}\ll \nu$,
$
\lim_{\varepsilon\searrow 0}\frac{V^{(1-\varepsilon)\pi+\varepsilon\pi'}_\tau(\rho) - V^\pi_\tau(\rho)}{\varepsilon}
= \left\langle \frac{\delta V^{\pi}_{\tau}(\rho)}{\delta \pi}\Big|_{\nu}, \pi' - \pi  \right\rangle_\nu
$,
with  
the dual pair 
$\langle \cdot, \cdot \rangle_\nu$
defined in \eqref{eq:dual_pair}, 
and  
$\frac{\delta V^{\pi}_{\tau}(\rho)}{\delta \pi}\big|_{\nu}\in B_b(S\times A)$
defined in \eqref{eq:delta_V_delta_pi-intro}.
\end{proposition}
\begin{proof}
Fix $\pi,\pi'\in \Pi_\mu \subset P(A|S)$. For each  $\varepsilon>0$ let $\pi^\varepsilon = \pi + \varepsilon(\pi'-\pi) \in \Pi_\mu$.
By  the performance difference lemma (Lemma \ref{lem:performance_diff}),  
\[
\begin{split}
\frac1\varepsilon (V^{\pi}_\tau(\rho) - V^{\pi^\varepsilon}_\tau(\rho)) 
& = \frac{1}{1-\gamma}\int_S \int_A\left(Q^{\pi^\varepsilon}_{\tau}(s,a)\right)(\pi-\pi')(da|s)d^{\pi}_{\rho}(ds)\\	
&\quad +  \frac{\tau}{1-\gamma}\int_S \frac1\varepsilon\Big(\operatorname{KL}(\pi(\cdot|s)|\mu(\cdot|s)) - \operatorname{KL}(\pi^\varepsilon(\cdot|s)|\mu(\cdot|s))\Big)d^{\pi}_{\rho}(ds)\,,
\end{split}
\]
which implies that 
\[
\begin{split}
\lim_{\varepsilon\to 0} \frac1\varepsilon (V^{\pi^\varepsilon}_\tau - V^{\pi}_\tau)(\rho) 
& = \lim_{\varepsilon\to 0}  \bigg[ \frac{1}{1-\gamma}\int_S \int_A\left(Q^{\pi^\varepsilon}_{\tau}(s,a)\right)(\pi'-\pi)(da|s)d^{\pi}_{\rho}(ds)\\	
&\quad \qquad +  \frac{\tau}{1-\gamma}\int_S \frac1\varepsilon\Big(\operatorname{KL}(\pi^\varepsilon(\cdot|s)|\mu(\cdot|s)) - \operatorname{KL}(\pi(\cdot|s)|\mu(\cdot|s))\Big)d^{\pi}_{\rho}(ds)\bigg]\,.
\end{split}
\]
By Proposition \ref{prop:Q_continuity},
$\lim_{\varepsilon\to 0}
\|
Q^{\pi^\varepsilon}_{\tau}
-Q^{\pi}_{\tau}\|_{B_b(S\times A)}=0$. 
Moreover, 
as $\pi,\pi'\in \Pi_\mu$,
for all $s\in S $, 
by \cite[Lemma 3.8]{kerimkulov2024mirror},
\[
\lim_{\varepsilon\searrow 0}\frac1\varepsilon\Big(\operatorname{KL}(\pi^\varepsilon(\cdot|s)|\mu(\cdot|s)) - \operatorname{KL}(\pi(\cdot|s)|\mu(\cdot|s))\Big)
=\int_A 
\ln \frac{\mathrm d \pi}{\mathrm d\mu}(a|s)  (\pi'-\pi)(d a|s)\,,
\]
which along with Proposition \ref{prop:boundedness_Q} and the  dominated   convergence theorem shows 
\begin{align*}
\label{eq:pg_unnormalized}
& \lim_{\varepsilon\searrow 0}\frac{V^{(1-\varepsilon)\pi+\varepsilon\pi'}_\tau(\rho) - V^\pi_\tau(\rho)}{\varepsilon}
=\frac{1}{1-\gamma}\int_S \int_A\left(Q^{\pi}_{\tau}(s,a)+\tau \ln \frac{\mathrm d \pi}{\mathrm d\mu}(a|s) \right)(\pi'-\pi)(da|s)d^{\pi}_{\rho}(ds).
\end{align*}
This along with 
the facts that for all $s\in S$, $V^\pi_\tau(s)$ is independently of $a$, and $\pi'(\cdot|s), \pi(\cdot|s)\in \clP(A)$
yields 
the identity 
\eqref{eq:pg_normalized}. 
\end{proof}

Proposition~\ref{prop:pg} 
holds for MDPs with arbitrary Polish state and action spaces.
To see its connection with  existing policy gradient formulas for discrete state and action spaces, consider a tabular MDP
whose state space $S$ has cardinality $|S|\in \bbN$,
and action space $A$ has cardinality $|A|\in \bbN$. 
Take the reference measure $\nu$
to be the uniform distribution on $A$, i.e., $\mu(a) = 1/|A|$ for all $a\in A$.
As $A$ is a finite set, 
any $\mu\in \clP(A)$ is absolutely continuous with respect to $\mu$, and hence $\Pi_\mu= \clP(A|S)$,
which  can be further  identified as  $\Delta (A)^{|S|}\subset \bbR^{|S|\times |A|}$, 
with $\Delta(A)$
being the probability simplex over $A$,
through  the direct parameterization $ \pi_\theta(a|s)\coloneqq \theta_{s,a} $, 
for all $s\in S$ and $a\in A$.
In this case, 
\eqref{eq:pg_normalized} implies  for all $\theta,\theta'\in \Delta(A)^{|S|}$,
\begin{equation}
\label{eq:pg_tabular}
\begin{split}
& \lim_{\varepsilon\searrow 0}\frac{V^{\pi_{\theta+ \varepsilon (\theta'-\theta)}}  - V^{\pi_{\theta}}_\tau(\rho)}{\varepsilon}
= \lim_{\varepsilon\searrow 0}\frac{V^{(1-\varepsilon) 
\pi_{\theta}+\varepsilon \pi_{\theta'}}  - V^{\pi_{\theta}}_\tau(\rho)}{\varepsilon}
\\
&\quad = \frac{1}{1-\gamma}\sum_{s\in S} \sum_{a\in A}\left(Q^{\pi_\theta}_{\tau}(s,a)+\tau \ln \theta_{s,a}-V^{\pi_\theta}_\tau(s)\right)(\theta'_{s,a}-\theta_{s,a})d^{\pi_\theta}_{\rho,s}\,,
\end{split}
\end{equation}
with $d^{\pi_\theta}_{\rho,s} \coloneqq d^{\pi_\theta}_{\rho}(\{s\}) $
for all $s\in S$, and  the flat derivative 
$\frac{\delta V^{\pi_\theta}_{\tau}(\rho)}{\delta \pi}\big|_{d^{\pi_\theta}_{\rho}}$
(taking $\nu =d^{\pi_\theta}_{\rho}$  in \eqref{eq:delta_V_delta_pi-intro}) 
is given by
\begin{align}
\label{eq:pg_tabular_flat}
\begin{split}
& \frac{\delta V^{\pi_\theta}_{\tau}(\rho) }{\delta \pi} (s,a)
\coloneqq \frac{\delta V^{\pi_\theta}_{\tau}(\rho)}{\delta \pi}\bigg|_{d^{\pi_\theta}_{\rho}}(s,a)
=  Q^{\pi_\theta}_{\tau}(s,a)+\tau \ln \theta_{s,a}-V^{\pi_\theta}_\tau(s)\,,
\end{split}
\end{align}
which is the gradient direction
\eqref{eq:delta_V_delta_pi-intro-state-dep-pairing} used in the Fisher--Rao gradient flow \eqref{eq:gradient_flow_introduction}.  
The right-hand side of 
\eqref{eq:pg_tabular_flat}
is often referred to as the regularized   advantage function
(see e.g., \cite{neu2017unified,mei2020global,lan2022policy}).

Note that  
the gradient expression
\eqref{eq:pg_tabular_flat}
is derived by constraining the perturbed directions
in \eqref{eq:pg_tabular}
to the form 
$\theta' - \theta$, with $\theta, \theta' \in \Delta(A)^{|S|}$.
This stands in contrast to the Euclidean gradient used in classical policy gradient formulas (see e.g., \cite{sutton1999policy}),
which is required to satisfy 
\eqref{eq:pg_tabular} for arbitrary directions
$\theta' - \theta \in \mathbb{R}^{|S| \times |A|}$. 
Relaxing the gradient requirement by only allowing perturbations from the probability simplex allows the policy gradient for the direct parameterization to be expressed in terms of advantage functions as in \eqref{eq:pg_tabular_flat}.

Indeed, 
standard policy gradient formulas typically require the policy parameterization   $\pi_\theta$ 
to  take  values within the simplex,
i.e., 
$\pi_\theta\in \Delta(A)^{|S|}$ for all $\theta\in \bbR^d$,
in order to express the policy gradient in terms of the advantage function; 
see 
\cite[Equation (6)]{agarwal2020optimality}
for unregularized MDPs
and 
\cite[Lemma 10]{mei2020global}
for regularized MDPs. 
Consequently,
those gradient formulas do not apply to  the direct parameterization $\pi_{\theta}=\theta$.

Here, by restricting the perturbed directions, 
the flat derivative \eqref{eq:pg_tabular_flat} represents a gradient constrained within the policy space $\Delta(A)^{|S|}$, rather than the full Euclidean gradient. Such a constrained gradient enables the design of a gradient flow that remains within the manifold $\clP(A|S)$ and optimizes the value function over it.

\section{Bellman equations for entropy-regularised MDPs}
\label{sec:Bellman}

This section establishes a dynamical programming principle for the regularised MDP (see Theorem \ref{thm:DPP}) and proves that for each $\pi\in \Pi_\mu$, the value function $V^{\pi}_\tau$ satisfies a Bellman equation (see Lemma \ref{lem:on_policy}).

\begin{theorem}
[Dynamic programming principle]
\label{thm:DPP}
Let $\tau>0$.
The optimal value function $V^*_{\tau}$ 
is the unique bounded solution of the following Bellman equation:
\[
V^*_{\tau}(s)=\inf_{m \in \clP(A)}\int_{A}\left(c(s,a)+\tau \ln \frac{\mathrm{d} m}{\mathrm{d} \mu}(a)+\gamma \int_{S}V^*_{\tau}(s')P(ds'|s,a)\right)m(da),\quad \forall s\in S,.
\]
Consequently, 
for all $s\in S$, 
\[
V^{\ast}_{\tau}(s)=-\tau\ln\int_{A}\exp\left(-
\frac{1}{\tau}Q^{\ast}_{\tau}(s,a)\right)\mu(da),
\]
where
$Q^*\in B_b(S\times A)$ is defined by  
\[
Q^{*}_{\tau}(s,a)=c(s,a)+\gamma\int_S V_{\tau}^{*}(s')P(ds'|s,a)\,,
\quad \forall (s,a)\in S\times A.
\]
Moreover, 
there is an optimal policy $\pi^*_{\tau} \in \clP_{\mu}(A|S)$  given by
\[
\pi^*_{\tau}(da|s) = \exp\left(-(Q^{\ast}_{\tau}(s,a)-V^{\ast}_{\tau}(s))/\tau\right)\mu(da)\,,
\quad \forall s\in S.
\]
\end{theorem}

\begin{lemma}\label{lem:on_policy} 
Let $\tau>0$ and $\pi \in \Pi_{\mu}.$ The value function $V^{\pi}_{\tau}$ 
is the unique bounded solution of the following Bellman equation:
\[
V^{\pi}_{\tau}(s)=\int_{A}\left(c(s,a)+\tau \ln \frac{\mathrm{d} \pi}{\mathrm{d} \mu}(a|s)+\gamma \int_{S}V^{\pi}_{\tau}(s')P(ds'|s,a)\right)\pi(da|s),\quad \forall s\in S\,.
\]
\end{lemma}

The proofs of Theorem \ref{thm:DPP}
and Lemma \ref{lem:on_policy} are given below. 

\begin{proof}[Proof of Theorem \ref{thm:DPP}]
This proof can mostly be seen as a special case of the proof of the DPP for generic Borel state and action spaces (e.g., \cite[Theorem   4.2.3]{hernandez2012discrete}) once one enriches the action space to $\clP(A)$ and understands the entropy/KL as an additional cost. Here, we present a self-contained proof for the reader's convenience.

Let $\tau>0$ be fixed. For each $u\in B_b(S)$ and each $s\in S$, define
\begin{align*}
T_{\tau} u(s) &= \inf_{m \in \clP(A)}\int_{A}\left[c(s,a) + \tau \ln \frac{\mathrm{d} m}{\mathrm{d} \mu}(a) + \gamma \int_{S} u(s')P(ds'|s,a)\right]m(da)\\
&=\tau \inf_{m \in \clP(A)}\left[\tau^{-1}\int_{A}Q_u(s,a) m(da)  + \operatorname{KL}(m|\mu)\right]\,,
\end{align*}
where
$
Q_u(s,a):=c(s,a) + \gamma \int_{S} u(s') P(ds'|s,a)\,.
$
Since 
$
\|Q_u\|_{B_b(S\times A)}\le \|c\|_{B_b(S\times A)} + \gamma \|u\|_{B_b(S)}\,,
$
by \cite[Proposition  1.4.2]{dupuis1997weak}, for each $s\in s$, we have
\[
T_\tau u(s) = - \tau \ln \int_A \exp\left(\tau^{-1}Q_u(s,a)\right)\,\mu(da)\,,
\]
where the infimum is uniquely attained at $\pi_{u}\in \clP_{\mu}(A|S)$ given by
\[
\pi_u(da|s)= \frac{\exp\left(-\tau^{-1} Q_u(s,a)\right)}{\int_{A}\exp\left(-\tau^{-1}Q_u(s,a) \right)\mu(da)} \mu(da).
\]
It is clear that $T_\tau u:S\to \mathbb{R}$ is measurable by Fubini's theorem. Moreover, 
since the natural logarithm is increasing, for all $s\in S$,  we have
\[
|T_\tau u(s)| 
\le \tau \left\vert \ln \int_A \exp\left(\tau^{-1}\|Q_u\|_{B_b(S\times A)}\right)\,\mu(da)\right\vert \le  \|c\|_{B_b(S\times A)} + \gamma \|u\|_{B_b(S)}\,.
\]
Thus, the Bellman operator $T_{\tau}:B_b(S)\rightarrow B_b(S)$ is well defined.

We will now show that $T_{\tau}$ is a contraction on the Banach space $B_b(S)$, following the proof in \cite{haarnoja2017reinforcement}. Let $u,v\in B_b(S)$ be fixed. Note that for all $(s,a)\in  S\times A$, we have
\[
Q_v(s,a) - Q_u(s,a) 
= \gamma \int_S \left(v(s') - u(s')\right) P(ds'|s,a) \leq \gamma \|u-v\|_{B_b(S)}\,.
\]
Using that the natural logarithm is increasing, for all $s\in S$, we get
\[
\begin{split}
-T_\tau u(s) 
&= \tau \ln \int_A \exp (\tau^{-1}Q_v(s,a) - \tau^{-1}Q_u(s,a) - \tau^{-1}Q_v(s,a)) \mu(da)\\    
& \leq \tau \ln \left(\exp \left(\frac\gamma\tau \|u-v\|_{B_b(S)}\right)\int_A \exp \left( -\tau^{-1} Q_v(s,a)\right) \mu(da)\right)\\
&=  \gamma \|u-v\|_{B_b(S)} - T_\tau v(s)\,,
\end{split}
\]
and hence  $T_\tau v(s) - T_\tau u(s) \leq  \gamma \|u-v\|_{B_b(S)}\,.$
Swapping the roles of $u$ and $v$ in the above, we find 
$T_\tau u - T_\tau v \leq  \gamma \|u-v\|_{B_b(S)}$, and thus 
\[
\|T_\tau u - T_\tau v\|_{B_b(S)} \leq \gamma \|u-v\|_{B_b(S)} \,. 
\]
Since $\gamma\in [0,1)$, $T_\tau:B_b(S) \to B_b(S)$ is a contraction, and there is a unique fixed point $\bar{V}\in B_b(S)$ such that
$T_\tau \bar{V}=  \bar{V}$. In particular,
for all $s\in S$,
\begin{align}
\bar{V}(s) &=  \inf_{m \in \clP(A)}\int_{A}\left[c(s,a) + \tau \ln \frac{\mathrm{d} m}{\mathrm{d} \mu}(a) + \gamma \int_{S} \bar{V}(s') P(ds'|s,a)\right]m(da)\label{eq:inf_Bellman}\\
&=\int_{A}\left[c(s,a) + \tau \ln \frac{\mathrm{d} \bar{\pi}}{\mathrm{d} \mu}(a|s) + \gamma \int_{S} \bar{V}(s') P(ds'|s,a)\right]\bar{\pi}(da|s),\label{eq:op_Bellman}
\end{align}
where the the unique infimum is attained at $\bar{\pi}\in \clP_{\mu}(A|S)$ given by 
\[
\bar{\pi}(da|s)= \frac{\exp\left(- \tau^{-1} Q_{\bar{V}}(s,a)\right)}{\int_{A}\exp\left(-\tau^{-1} Q_{\bar{V}}(s,a) \right)\mu(da)} \mu(da)\,.
\]
Thus, we have  proved that $\bar{V}$ is the unique bounded solution of the Bellman equation \eqref{eq:op_Bellman}.

It remains to  show  that $\bar{V}(s)=V^*_\tau(s)$ for all $s\in S$. We will first show $\bar{V}(s)\ge V^*_{\tau}(s)$ for all $s\in S$. Iterating \eqref{eq:op_Bellman} and using  \eqref{eq:markov_like}, we get that for all $N\in \bbN$,
\[
\bar{V}(s)  = \bbE^{\bar{\pi}}_{s}\sum_{n=0}^{N-1}\gamma^n\left(c(s_n, a_n) + \tau \ln \frac{\mathrm{d} \bar{\pi} }{\mathrm{d} \mu}(a_n|s_n)\right) + 
\gamma^{N}
\int_{A}P^{(N)}\bar{V}(s,a) \bar{\pi}(da|s),
\]
where $P^{(N)}\in \clL(B_b(S),B_b(S\times A))
$ is the operator induced by the $N$-step transition kernel.
Since $P^{(N)}$  has operator norm less than one, we have 
$
\int_{A}P^{(N)}\bar{V}(s,a) \bar{\pi}(da|s)  \le  \|\bar{V}\|_{B_b(S)}\,,
$
and hence by Lebesgue's dominated convergence theorem,
for all $s\in S$,
\[
\bar{V}(s)  = \bbE^{\bar{\pi}}_{s}\sum_{n=0}^{\infty}\gamma^n\left(c(s_n,a_n) + \tau \ln \frac{\mathrm{d} \bar{\pi} }{\mathrm{d} \mu}(a_n|s_n)\right)  \ge V^*_{\tau}(s).
\]
We will now show that $\bar{V}(s) \le V_{\tau}^{\pi}(s)$ for all $\pi\in \Pi$ and $s\in S$, which then implies that $\bar{V}(s) \le V^*(s)$ for all $s\in S$. 
Let $\pi=\{\pi_n\}_{n\in \mathbb{N}_0}\in \Pi$, so that for each $n\in \bbN_0$,   $\pi_n \in \clP(A|H_n)$. Let $s\in S$ denote an arbitrary fixed initial state. 
Without loss of generality, we assume  
$\pi_n\in \clP_{\mu}(A|H_n)$ for all $n\in \bbN_0$, since otherwise 
$V^{\pi}(s)=\infty$. For each $n\in \bbN$, applying \eqref{eq:markov_like} and adding and subtracting, we find 
\begin{equation*}
\begin{split}
& \gamma^{n+1}\bbE^{\pi}_{s}\left[ \bar{V}(s_{n+1})| h_n, a_n\right]=\gamma^{n+1}\int_{S}\bar{V}(s')P(ds'|s_n,a_n)\\
& \quad = \gamma^{n}\left[c(s_n,a_n) + \tau \ln \frac{\mathrm{d} {\pi} }{\mathrm{d} \mu}(a_n|h_n) + \gamma \int_{S} \bar{V}(s')P(ds'|s_n,a_n)\right]
-\gamma^{n} \left(c(s_n,a_n) + \tau \ln \frac{\mathrm{d} {\pi} }{\mathrm{d} \mu}(a_n|h_n) \right)\,.
\end{split}
\end{equation*}
By the tower property of conditional expectations,
\begin{equation*}
\begin{split}
& \gamma^{n+1}\bbE^{\pi}_{s}\left[ \bar{V}(s_{n+1})| h_n\right]
=\gamma^{n+1}\bbE^{\pi}_{s}\left[\bbE^{\pi}_{s}\left[ \bar{V}(s_{n+1})| h_n, a_n\right]|h_n\right]\\
&\quad
=\gamma^{n}\bbE^{\pi}_{s}\bigg[
c(s_n,a_n) + \tau \ln \frac{\mathrm{d} {\pi} }{\mathrm{d} \mu}(a_n|h_n) + \gamma \int_{S} \bar{V}(s')P(ds'|s_n,a_n)\bigg|h_n\bigg]
\\
&\qquad -\gamma^{n}  \bbE^{\pi}_{s}\left[
c(s_n,a_n) + \tau \ln \frac{\mathrm{d} {\pi} }{\mathrm{d} \mu}(a_n|h_n) \bigg|h_n\right]
\\
&\quad
= \gamma^{n}
\bbE^{\pi}_{s}\bigg[ 
\int_A
\left(c(s_n,a) + \tau \ln \frac{\mathrm{d} {\pi} }{\mathrm{d} \mu}(a|h_n) + \gamma \int_{S} \bar{V}(s')P(ds'|s_n,a)\right)\pi_n(da|h_n)
\bigg|h_n\bigg]
\\
&\qquad 
-  \gamma^{n}  \bbE^{\pi}_{s}\left[  c(s_n,a_n) + \tau \ln \frac{\mathrm{d} {\pi} }{\mathrm{d} \mu}(a_n|h_n)  
\bigg|h_n\right]\,,
\end{split}
\end{equation*}
where we have used \eqref{eq:markov_like} in the last identity.  

Applying \eqref{eq:inf_Bellman} (with 
$m= \pi_n(da|h_n)$),
\begin{align*}
\int_A
\left(c(s_n,a) + \tau \ln \frac{\mathrm{d} {\pi} }{\mathrm{d} \mu}(a|h_n) + \gamma \int_{S} \bar{V}(s')P(ds'|s_n,a)\right)\pi_n(da|h_n)
\ge \bar{V}(s_n),
\end{align*}
and hence 
\begin{equation*}
\begin{split}
& \gamma^{n+1}\bbE^{\pi}_{s}\left[ \bar{V}(s_{n+1})| h_n\right]
\ge
\gamma^{n}
\bbE^{\pi}_{s}[\bar{V}(s_n)|h_n]
-  \gamma^{n}  \bbE^{\pi}_{s}\left[  c(s_n,a_n) + \tau \ln \frac{\mathrm{d} {\pi} }{\mathrm{d} \mu}(a_n|h_n) 
\bigg|h_n\right]\,.
\end{split}
\end{equation*}
Rearranging the inequality, applying the expectation operator $\bbE^{\pi}$, and using a telescoping sum argument, we get
\[
\bbE^{\pi}_s\left[\sum_{n=0}^{N-1}\gamma^n\left(c(s_n,a_n) + \tau  \ln \frac{\mathrm{d} \pi }{\mathrm{d} \mu}(a_n|s_n)\right)\right]\ge \bar{V}(s) - \gamma^N\bbE^{\pi}_{s}\left[\bar{V}(s_{N})\right] \,.
\]
Letting $N\rightarrow \infty$ and using that $\bar{V}\in B_b(S)$, we find $V^{\pi}(s)\ge  \bar{V}(s)$ for all $s\in S$, which gives $\bar{V}(s)\le V^*_{\tau}(s)$ for all $s\in S$, and finally $\bar{V}\equiv V^*_{\tau}$. This completes the proof.
\end{proof}

\begin{proof} [Proof of Lemma \ref{lem:on_policy}]
For  each $u\in B_b(S)$, $\pi\in \Pi_{\mu}$, and $s\in S$, define
\[
L_{\tau}^{\pi}u(s)=\int_{A}\left(c(s,a) + \tau \ln \frac{\mathrm{d}\pi}{\mathrm{d}\mu}(a|s) + \gamma \int_{S} u(s')P(ds'|s,a)\right)\pi(da|s)\,,
\]
which is well-defined as 
$\pi\in \Pi_\mu$
and
$
\left\|\int_{S} u(s')P(ds'|s,\cdot)\right\|_{B_b(A)}\le \|u\|_{B_b(S)}$.
Recalling that $\pi=\boldsymbol{\pi}(f)$ for some $f\in B_b(S\times A)$, and thus  by Proposition \ref{prop:boundedness_Q},  
\[
\left\|c + \tau \ln \frac{\mathrm{d}\pi}{\mathrm{d}\mu}\right\|_{B_b(S\times A)} \le \|c\|_{B_b(S\times A)} + \tau \left\|\ln \frac{\mathrm{d}\pi}{\mathrm{d}\mu}\right\|_{B_b(S\times A)} \le \|c\|_{B_b(S\times A)} + 2\tau \|f\|_{B_b(S\times A)}\,.
\]
Thus for all $u\in B_b(S)$, $L_{\tau}u\in B_b(S)$ and 
\[
\|L_{\tau}^{\pi}u\|_{B_b(S)} \le \|c\|_{B_b(S\times A)} + 2 \tau \|f\|_{B_b(S\times A)} + \gamma \|u\|_{B_b(S)}\,.
\]
Moreover, for all $u,v\in B_b(S)$, we have 
\[
\left\|L_{\tau}^{\pi} u - L_{\tau}^{\pi}v\right\|_{B_b(S)}=  \gamma \left\|\int_{A}\int_{S}(u(s')-v(s'))P(ds'|\cdot,a) \pi(da|\cdot)\right\|_{B_b(s)} \le \gamma \|u-v\|_{B_b(S)}\,.
\]
Since $\gamma\in [0,1)$, the map $L_\tau: B_b(S) \to B_b(S)$ is a contraction, and thus there is a unique  $V\in B_b(S)$ such that for all $s\in S$,
\begin{equation}\label{eq:lin_bellman_pf}
V(s)=\int_{A}\left(c(s,a) + \tau \ln \frac{\mathrm{d}\pi}{\mathrm{d}\mu}(a|s) + \gamma \int_{S} V(s')P(ds'|s,a)\right)\pi(da|s)\,.
\end{equation}
To verify $V=V^{\pi}_{\tau}$, 
iterating \eqref{eq:lin_bellman_pf} and using  \eqref{eq:markov_like}, we get that for all $N\in \bbN$,
\[
V(s)  = \bbE^{\pi}_{s}\sum_{n=0}^{N-1}\gamma^n\left(c(s_n, a_n) + \tau \ln \frac{\mathrm{d}\pi}{\mathrm{d}\mu}(a_n|s_n)\right)+ \gamma^{N}\int_{A}P^{(N)}V(s,a)\pi(da|s),
\]
where $P^{(N)}\in \clL(B_b(S),B_b(S\times A))
$ is the operator induced by the $N$-step transition kernel.
Since $P^{(N)}$  has an operator norm less than one, we have 
$
\int_{A}P^{(N)}V(s,a) {\pi}(da|s)  \le  \|V\|_{B_b(S)}\,,
$
and hence by Lebesgue's dominated convergence theorem,
for all $s\in S$,
\[
V(s)  = \bbE^{\pi}_{s}\sum_{n=0}^{\infty}\gamma^n\left(c(s_n,a_n) + \tau \ln \frac{\mathrm{d}\pi}{\mathrm{d}\mu}(a_n|s_n)\right)=V^{\pi}_{\tau}(s),
\]
where the last identity used the definition of $V^{\pi}_\tau$ 
in \eqref{eq:V_pi_tau}.
This proves the desired identity. 
\end{proof}

\section{Natural gradient flows}
\label{sec:natural_gradient_flows}

In this appendix, we provide more details on how to interpret \eqref{eq:gradient_flow_introduction} and \eqref{eq:mirror_descent_intro} as natural gradient flows with respect to suitable Riemannian metrics.
These arguments utilize the directional derivatives established in Section \ref{sec:derivatives}.

\subsection{Convex conjugate of negative entropy and the Fisher--Rao metric}

Let $\nu\in \clP(S)$ denote an arbitrary measure and define the state-integrated negative entropy $h_{\nu}: \clP(A|S)\rightarrow \overline{\bbR}$ by
\begin{equation}\label{def:state_integrated_entropy}
h_{\nu}(\pi)=\frac{1}{1-\gamma}\int_{S}\operatorname{KL}(\pi(\cdot|s)|\mu)\nu(ds)\,.
\end{equation}
An extension of Lemma 1.4.3.\ in \cite{dupuis1997weak} shows that the convex conjugate of integrated negative entropy $h_{\nu}$, understood as a mapping $h^*_{\nu}: B_b(S\times A)\rightarrow \overline{\bbR}$, is given by
\begin{align}
\label{eq:KL_fenchel-legendre}
h^*_{\nu}(Z)= \max_{\pi\in \clP(A|S)}\left[\langle Z, \pi\rangle_{\nu}  - h_{\nu}(\pi) \right]=\int_{S}\ln \left(\int_A e^{Z(s,a)} \mu(da)\right)\nu(ds) \,.
\end{align}
By Proposition \ref{prop:derivative_integral_logexp},  for every $Z,Z'\in B_b(S\times A)$, 
\[
\lim_{\varepsilon\searrow 0}\frac1\varepsilon(h^\ast_\nu(Z+\varepsilon Z') - h^\ast_\nu(Z)) = [\mathfrak{d} h^\ast_\nu(Z)]Z' = \left\langle Z', \frac{\delta h^\ast_\nu}{\delta Z}\big|_{\nu}(Z)  \right\rangle_\nu\,,
\]
where  the first variation $\frac{\delta h^*_{\nu}}{\delta Z}\big|_{\nu}: B_b(S\times A)\rightarrow \Pi_{\mu}\subset  b\clM(A|S)$ is given by
\begin{equation}
\label{eq:delta_Phi_delta_Z}
\frac{\delta h^\ast_\nu}{\delta Z}\bigg|_{\nu}(Z)(da|s) = \frac{\delta h^*_{\nu}}{\delta Z}(Z)(da|s)=\boldsymbol{\pi}(Z)(da|s)\,.
\end{equation}
The relation~\eqref{eq:delta_Phi_delta_Z} is expected since by Lemma 1.4.3.\ in \cite{dupuis1997weak},
\begin{align*}
\boldsymbol{\pi}(Z) 
\in \underset{m \in \mathcal P(A|S)}{\operatorname{arg\,max}}\left[\langle Z, m\rangle_{\nu}  - h_{\nu}(m) \right]\,,
\end{align*}
and the `gradient' of the convex conjugate should be equal to the $\operatorname{arg \,max}$ in the Legendre--Fenchel transformation~\eqref{eq:KL_fenchel-legendre}.

It follows from Proposition \ref{prop:differentiability_Pi} that the Hessian $\frac{\delta^2 h^*_{\nu}}{\delta Z^2}\big|_{\nu}: B_b(S\times A)\rightarrow \clL(B_b(S\times A), b\clM(A|S))$ relative to the pairing $\langle \cdot, \cdot\rangle_{\nu}$ is given by
\begin{equation}\label{eq:hessian}
\frac{\delta^2 h^*_{\nu}}{\delta Z^2}\bigg|_\nu(Z)(f)(da|s) = \frac{\delta^2 h^*_{\nu}}{\delta Z^2}(Z)(f)(da|s) = f_{\boldsymbol{\pi}(Z)}(s,a)\boldsymbol{\pi}(Z)(da|s)\,,
\end{equation}
where 
\[
f_{\boldsymbol{\pi}(Z)}(s,a) = f(s,a) - \int_{A}f(
s, a')\boldsymbol{\pi}(Z)(da'|s)\,.
\]
In particular, for all $f,g\in B_b(S\times A)$,
\begin{align*}
[\mathfrak d [\mathfrak d h_{\nu}^*(Z)](f)](g) &= \left\langle g,\frac{\delta^2 h^*_{\nu}}{\delta Z^2}\bigg|_\nu(Z)(f)\right\rangle_{\nu} = \frac{1}{1-\gamma}\int_{S}\int_{A}g(s,a)f_{\boldsymbol{\pi}(Z)}(s,a)\boldsymbol{\pi}(Z)(da|s)\nu(ds)\\
&=\frac{1}{1-\gamma}\int_{S}\int_{A}g_{\boldsymbol{\pi}(Z)}(s,a)f_{\boldsymbol{\pi}(Z)}(s,a)\boldsymbol{\pi}(Z)(da|s)\nu(ds)\,,
\end{align*}
where  the last line   used   
\[
\int_{A}\int_{A}g(s,a')\boldsymbol{\pi}(Z)(da'|s)f_{\boldsymbol{\pi}(Z)}(s,a)\boldsymbol{\pi}(Z)(da|s)=0\,,\quad   \forall s\in S\,.
\]
Two crucial properties of the Hessian are that i) for any $v \in B_b(S)$ and $f\in B_b(S\times A)$,  
\begin{equation}\label{eq:hessian_diagonal}
\frac{\delta^2 h^*_{\nu}}{\delta Z^2}(Z)(vf )(da|s) = v(s)\frac{\delta^2 h^*_{\nu}}{\delta Z^2}(Z)(f )(da|s)
\end{equation}
and ii) for all $\nu,\nu'\in \clP(S)$,
\begin{equation}\label{eq:hessian_change_of_measure}
\frac{\delta^2 h^*_{\nu}}{\delta Z^2}\bigg|_{\nu}=\frac{\delta^2 h^*_{\nu'}}{\delta Z^2}\bigg|_{\nu'}\,.
\end{equation}
Proposition~\ref{prop:differentiablity_log_density} implies that the first variation of the map $\ln \frac{d\boldsymbol{\pi}}{d\mu}:B_b(S\times A)\rightarrow B_b(S\times A)$, understood as a mapping  $\frac{\delta \ln \frac{d\boldsymbol{\pi}}{d\mu}}{\delta Z}:B_b(S\times A)\rightarrow \clL(B_b(S\times A), B_b(S\times A))$, is given by
\[
\frac{\delta \ln \frac{d\boldsymbol{\pi}}{d\mu}}{\delta Z}(Z)(g)(s,a)= g_{\boldsymbol{\pi}(Z)}(s,a)\,.
\]
Thus,  for all $f,g\in B_b(S\times A)$,
\[
[\mathfrak d [\mathfrak d h_{\nu}^*(Z)](f)](g) =\frac{1}{1-\gamma}\int_{S}\int_{A}\frac{\delta \ln \boldsymbol{\pi}}{\delta Z}(Z)(g)(s,a)\frac{\delta \ln \boldsymbol{\pi}}{\delta Z}(Z)(f)(s,a)\boldsymbol{\pi}(Z)(da|s)\nu(ds)\,,
\]
which can be interpreted as a Fisher--Rao  metric on the dual space $B_b(S\times A)$. It is the Fisher information matrix originally derived in \cite{kakade2001natural} that corresponds to the entire policy class $\{\boldsymbol{\pi}(Z)\mid Z\in B_b(S\times A)\}$.

\subsection{Mirror descent flow as natural gradient flow}
Propositions~\ref{prop:differentiability_Pi} and~\ref{prop:differentiability_dV_df} can be used to argue that the first variation of $V^{\boldsymbol{\pi}(\cdot)}_{\tau}(\rho): B_b(S\times A)\rightarrow \bbR$, relative to duality pairing $\langle \cdot, \cdot\rangle_{\nu}$, understood as a mapping $\frac{\delta V_{\tau}^{\boldsymbol{\pi}(\cdot)(
\rho)}}{\delta Z}\big|_{\nu}: B_b(S\times A)\rightarrow  b\clM(A|S)$, is given by
\begin{equation}
\label{eq:delta_V_delta_Z_intro}
\frac{\delta V^{\boldsymbol \pi(Z)}_\tau(\rho)}{\delta Z}\bigg|_{\nu}(da|s) = \frac{\delta V^{\boldsymbol{\pi}(Z)}_{\tau}(\rho)}{\delta \pi}\bigg|_{\nu}(s,a)\boldsymbol{\pi}(Z)(da|s)= \frac{\delta^2 h^*_{\nu}}{\delta Z^2}(Z)\left(\frac{\delta V^{\boldsymbol{\pi}(Z)}_{\tau}(\rho)}{\delta \pi}\bigg|_{\nu}\right)(da|s)\,,
\end{equation}
where we have used  $\left\langle \frac{\delta V^{\pi}_{\tau}(\rho)}{\delta \pi}\big|_{\nu}, \pi  \right \rangle_\nu = 0$ in the first and second equality and~\eqref{eq:hessian} in the second equality. Next,  we observe that
\begin{align*}
\partial_t Z_t 
&= -\frac{\delta V^{\boldsymbol \pi(Z_t)}_\tau(\rho)}{\delta \pi}\bigg|_{d^{\boldsymbol{\pi}(Z_t)}_{\rho}}
=-\frac{\mathrm d\nu}{\mathrm d d^{\boldsymbol \pi(Z_t)}_{\rho}}\frac{\delta V^{\boldsymbol \pi(Z_t)}_\tau(\rho)}{\delta \pi}\bigg|_{\nu}
=\frac{\mathrm d\nu}{\mathrm d d^{\boldsymbol \pi(Z_t)}_{\rho}}\left(\frac{\delta^2 h^*_{\nu}}{\delta Z^2}(Z_t)\right)^{-1}\left(\frac{\delta V^{\boldsymbol \pi(Z_t)}_\tau(\rho)}{\delta Z}\bigg|_{\nu}\right)\,,
\end{align*}
where the first equality uses the dynamics~\eqref{eq:mirror_descent_intro}, the second is obtained by changing the duality pairing of the flat derivative, and the third follows from applying the inverse of the Hessian to~\eqref{eq:delta_V_delta_Z_intro}. 
By \eqref{eq:hessian_diagonal}, we have
\begin{align}
\label{eq:Z_flow_with_nu_on_Hessian_pi_on_der}
\partial_t Z_t  & =-\left(\frac{\delta^2 h^*_{\nu}}{\delta Z^2}(Z_t)\right)^{-1}\left(\frac{\delta V^{\boldsymbol \pi(Z_t)}_\tau(\rho)}{\delta Z}\bigg|_{d^{\boldsymbol{\pi}(Z_t)}_{\rho}}\right)\,,
\end{align}
which by \eqref{eq:hessian_change_of_measure} implies  
\begin{equation}\label{eq:natural_gradient_mirror_Z}
\partial_t Z_t 
=-\left(\frac{\delta^2 h^*_{d_{\rho}^{\boldsymbol{\pi}(Z_t)}}}{\delta Z^2}(Z_t)\right)^{-1}\left(\frac{\delta V^{\boldsymbol \pi(Z_t)}_\tau(\rho)}{\delta Z}\bigg|_{d^{\boldsymbol{\pi}(Z_t)}_{\rho}}\right)\,.
\end{equation}
Therefore, ~\eqref{eq:mirror_descent_intro} is a natural gradient flow on the dual space $B_b(S\times A)$ in the sense that the derivative of the objective function in the dual variable is pre-conditioned by the inverse of the Hessian of the conjugate of negative entropy, which is itself induces the Fisher--Rao metric on the dual $B_b(S\times A)$.
In particular, we see that ~\eqref{eq:mirror_descent_intro} is an infinite-dimensional version of the natural policy gradient update originally proposed in~\cite{kakade2001natural}.

\subsection{Fisher--Rao flow as   natural gradient flow}
Formally, by convex duality, we expect the inverse of the Hessian of negative entropy $h_{\nu}$ to be equal to the Hessian of its conjugate; that is,
$\left(\frac{\delta^2 h_\nu}{\delta \pi^2}\right)^{-1} \circ \boldsymbol{\pi}: B_b(S\times A)\rightarrow \clL(B_b(S\times A); b\clM(A|S))$ is given by
\begin{equation}\label{eq:hessian_neg_ent}
\left(\frac{\delta^2 h_\nu}{\delta \pi^2}(\boldsymbol{\pi}(Z))\right)^{-1}(f)=f_{\boldsymbol{\pi}(Z)}(s,a)\boldsymbol{\pi}(Z)(da|s)\,.
\end{equation}
To obtain the primal flow, we formally compute the time derivative of the mirror map $\boldsymbol{\pi}$ applied to the mirror flow \eqref{eq:natural_gradient_mirror_Z}:
\begin{equation*}
\begin{split}
\partial_t \boldsymbol{\pi}(Z_t) & =\partial_t\left[\frac{\delta h^\ast_\nu}{\delta Z}(Z_t)\right]= \frac{\delta^2 h^\ast_\nu}{\delta Z^2}(Z_t) \partial_t Z_t=\frac{\delta V^{\boldsymbol \pi(Z_t)}_\tau(\rho)}{\delta Z}\bigg|_{d^{\boldsymbol{\pi}(Z_t)}_{\rho}}\\
&=\left(\frac{\delta^2 h^*_{d^{\boldsymbol{\pi}(Z_t)}_{\rho}}}{\delta Z^2}(Z_t) \right)\frac{\delta V^{\boldsymbol{\pi}(Z_t)}_{\tau}}{\delta \pi}\bigg|_{d^{\boldsymbol{\pi}(Z_t)}_{\rho}}  =  -\left(\frac{\delta^2 h_{d^{\boldsymbol{\pi}(Z_t)}_{\rho}}}{\delta \pi^2}(\boldsymbol{\pi}(Z_t))\right)^{-1}\left(\frac{\delta V^{\boldsymbol{\pi}(Z_t)}_{\tau}}{\delta \pi}\bigg|_{d^{\boldsymbol{\pi}(Z_t)}_{\rho}}\right)\,,
\end{split}
\end{equation*}
where   the first equality used \eqref{eq:delta_Phi_delta_Z},  the second equality   used the chain rule,   the third equality used~\eqref{eq:Z_flow_with_nu_on_Hessian_pi_on_der},   the fourth equality  used~\eqref{eq:delta_V_delta_Z_intro}, and   the fifth equality used \eqref{eq:hessian_neg_ent}. In particular, $\pi_t=\boldsymbol{\pi}(Z_t)$ satisfies
\begin{equation}\label{eq:natural_gradient_pi_flow}
\partial_t \pi_t
= -\left(\frac{\delta^2 h_{d^{\pi_t}_{\rho}}}{\delta \pi^2}(\pi_t)\right)^{-1}\left(\frac{\delta V^{\pi_t}_{\tau}}{\delta \pi}\bigg|_{d^{\pi_t}_{\rho}}\right)\,,
\end{equation}
which by  $\left\langle \frac{\delta V^{\pi}_{\tau}(\rho)}{\delta \pi}\big|_{\nu}, \pi  \right \rangle_\nu = 0$ and \eqref{eq:hessian_neg_ent} is~\eqref{eq:gradient_flow_introduction}. 
Therefore,  ~\eqref{eq:gradient_flow_introduction} is a state-dependent Fisher--Rao gradient flow in the primal space $\Pi_{\mu}$.

\section{Numerical experiments}
\label{sec:numerical}

This section  examines   the performance of various time discretizations of the mirror flow \eqref{eq:mirror_descent} through a simple 
Gridworld environment  (see e.g.~\cite[Example 3.5]{sutton2018reinforcement}).\footnote{
Code available at \href{https://github.com/deterministicdavid/mirror_descent_for_gworld_mdp}{\texttt{https://github.com/deterministicdavid/mirror\_descent\_for\_gworld\_mdp}}.
} 
The results demonstrate
that various Runge-Kutta (RK) discretizations of \eqref{eq:mirror_descent} achieve exponential convergence, and that suitable higher-order discretizations can yield smaller errors compared to the Euler discretization
(equivalently the
mirror descent analyzed in Section~\ref{sec:unregularised})
when using relatively small steps.

More precisely, consider the environment with the action space $A$ containing four actions (up, down, left, and right)  and
the state space 
$S$ being an $n\times n$ grid (we take $n=11$) consisting of  several ``trap'' states which incur high cost, one reward state with negative cost and the remaining states leading to small cost. 
The goal is to navigate from a given starting position to the reward state while avoiding trap states along the path.
At each state, the agent selects an action corresponding to the desired moving direction. The agent then moves in the chosen direction with probability $p \in (0,1)$ (we take $p=0.7$) or moves to one of the orthogonal directions with equal probability.
We set the discount factor $\gamma=0.8$,
the reference measure $\mu$ to be the uniform distribution,
and 
the entropy regularization parameter 
$\tau =0.1$.

For clarity, we assume all coefficients are known and focus on analyzing the impact of different time discretizations of the flow~\eqref{eq:mirror_descent}, excluding the effects of gradient estimation errors.
To quantify algorithm convergence, a reference optimal value function $V^*_\tau$ is computed  by solving the corresponding Bellman equation using the policy iteration algorithm (see, e.g., \cite{bertsekas2004stochastic}).

To derive the discrete-time policy updates, 
we write the mirror flow \eqref{eq:mirror_descent} 
as the ODE
$\partial_t Z_t =f(Z_t)$, 
where $f:B_b(S\times A)\to B_b(S\times A)$ is defined by
\[
f(Z)= -Q^{\boldsymbol{\pi}(Z)}_\tau -\tau Z+ V^{\boldsymbol{\pi}(Z)}_\tau,
\]
with $\boldsymbol{\pi}$ defined in \eqref{eq:pi_f_mu}.
Given an initial guess $Z^0\in B_b(S\times A)$, consider the following explicit discretizations of the flow  \eqref{eq:mirror_descent}
with a time step size $h>0$: 
\begin{enumerate}
\item The 1st-order RK  method (known as the Euler method) given by 
\[
\label{eq:mirro_euler}
Z^{n+1}  = Z^n + h f(Z^n), \quad n\in \bbN_0\,,
\]
which
corresponds to the policy mirror descent update 
\eqref{eq:mirror_descent_euler}
with $\lambda =1/h$.

\item The 2nd-order  RK method (known as the midpoint method) given by  
\[
\label{eq:midpoint}
Z^{n+1}  = Z^n + h f\left(Z^n + \tfrac12 h f(Z^n)\right),\quad n\in \bbN_0\,.
\]
\item The  4th-order RK method:
\[
\begin{split}
\label{eq:RK4}
Z^{n+1} & = Z^n + \tfrac16 (k_1 + 2 k_2 + 2 k_3 + k_4)\,, 
\quad 
n\in \bbN_0\,,
\quad \text{where}\\
k_1 & = h f(Z^n)\,,\,\,\, 
k_2 = h f(Z^n + \tfrac12 k_1)\,,\,\,\, 
k_3 = h f(Z^n + \tfrac12 k_2)\,,\,\,\,
k_4 = h f(Z^n + k_3)\,.	
\end{split}
\]
\end{enumerate}
In the sequel, we take the initial guess $Z^0 = 0$ and evaluate the convergence of the error term  
$\sup_{s \in S} \left| V^{\pi^n_h}_\tau(s) - V^*_\tau(s) \right|$,
where $\pi^n_h \coloneqq \boldsymbol{\pi}(Z^N_h)$,   $Z^N_h$ is obtained using one of the aforementioned discretizations with a  time step size $h>0$ and iteration number $N\in \bbN$,
and $V^*_\tau$ is the refernce optimal solution.

Figure~\ref{fig:plots1} illustrates the convergence of the aforementioned 
RK  discretizations with various choices of time step sizes and iteration numbers, where all three methods use the same step size $h$.
Figure~\ref{subfig:plots1_left}
shows that    all three methods (with a fixed time stepsize) converge exponentially to the optimal value function as the iteration number tends to infinity. 
This exponential convergence aligns with the convergence result of the continuous-time flow presented in Theorem~\ref{thm:linear_convergence}. 
The empirical performance of the mirror descent update exceeds the theoretical convergence result presented in Theorem~\ref{thm:convergence_mirror_steps}, which only guarantees convergence when averaging over a suitable initial distribution, rather than uniform convergence across all initial states.
Notably, schemes based on higher-order discretizations achieve greater accuracy than the mirror descent obtained through the Euler discretization, although the 4th-order method performs similarly to the 2nd-order method.
A similar conclusion holds when the running time horizon $T$ is fixed, and the algorithms are executed with $h = T/N$ for increasing $N \in \mathbb{N}$, as illustrated in Figure~\ref{subfig:plots1_right}.

\begin{figure}
\centering
\begin{subfigure}[t]{0.48\textwidth}
\centering
\includegraphics[width=\textwidth]{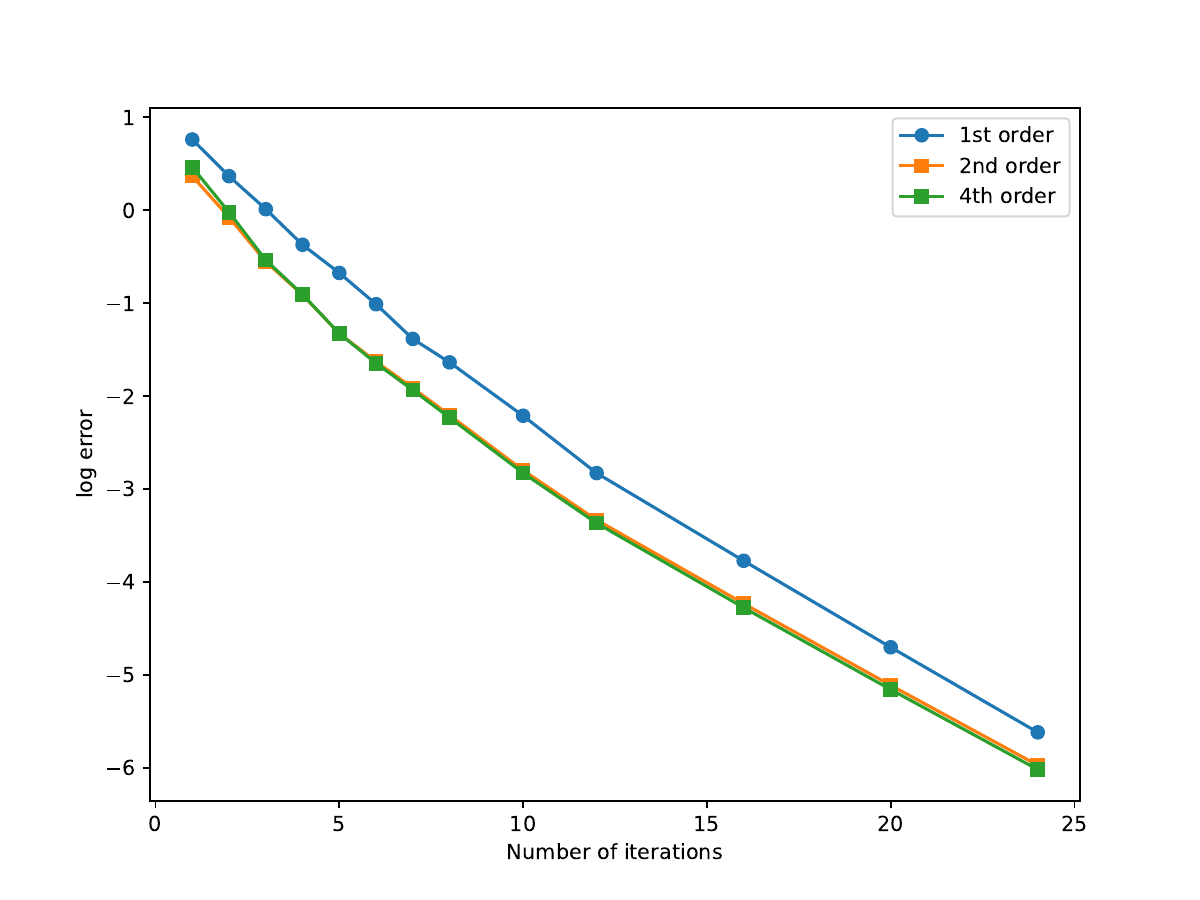} 
\caption{
Increasing the number of iterations $N$ while fixing the time step size $h = 1$, resulting in an increasing time horizon $T = hN$.}
\label{subfig:plots1_left}
\end{subfigure}
\hfill
\begin{subfigure}[t]{0.48\textwidth}
\centering
\includegraphics[width=\textwidth]{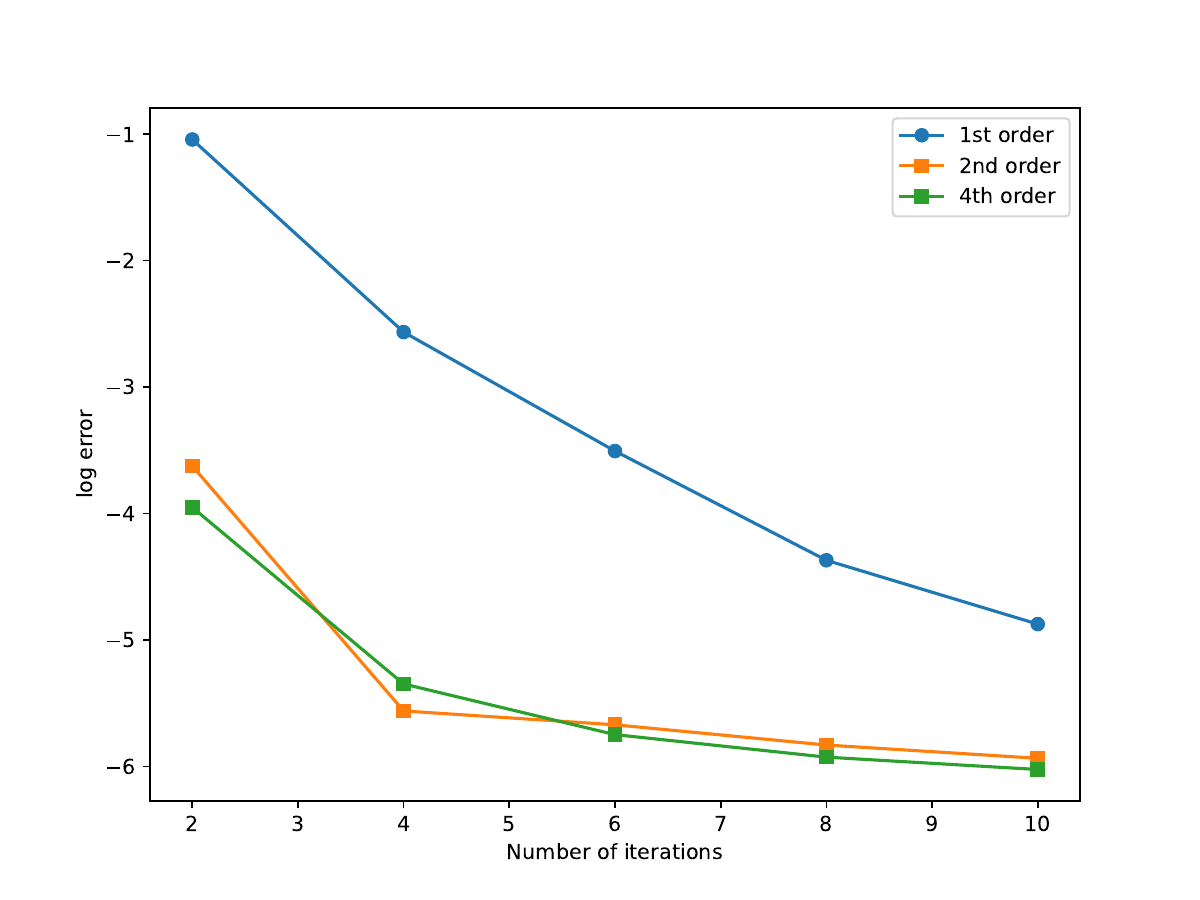}
\caption{
Increasing the number of iterations $N$ while fixing the  time horizon $T = 25$, resulting in a decreasing time step size $h = T/N$.
}
\label{subfig:plots1_right}
\end{subfigure}

\caption{Convergence of various Runge–Kutta discretizations of the mirror flow, with all methods using the same time step size $h$.}
\label{fig:plots1}
\end{figure}

We further compare the performance of the above discretization methods by fixing the computational cost for each policy update, represented by the number of required policy evaluations.
Observe that for each policy update, the 2nd-order and 4th-order RK methods require twice and four times as many policy evaluations as the 1st-order method (i.e.,   mirror descent), respectively.
To equalize the computational costs across all methods, we run the 2nd-order method with twice the step size used in the 1st-order  method and the 4th-order method with four times the step size of the 1st-order  method. The results are presented in Figure \ref{fig:plots2}.

The results show that all three methods still converge exponentially as  the number of iterations tends to infinity, but the performance differences among them are   reduced (Figure~\ref{subfig:plots2_left}).
Moreover, the 2nd-order method achieves greater accuracy than both the 1st-order and 4th-order methods for the same computational cost, as illustrated in Figure~\ref{subfig:plots2_right}.
This suggests that in some cases, especially when taking small step sizes,  higher-order discretization of the continuous-time flow can produce more efficient algorithms than the standard mirror descent update.
Further analysis would be needed to determine whether this insight can be utilised when the updates are estimated from samples which is a situation where small updates are natural. 
We leave this for future work.

\begin{figure}
\centering
\begin{subfigure}[t]{0.48\textwidth}
\centering
\includegraphics[width=\textwidth]{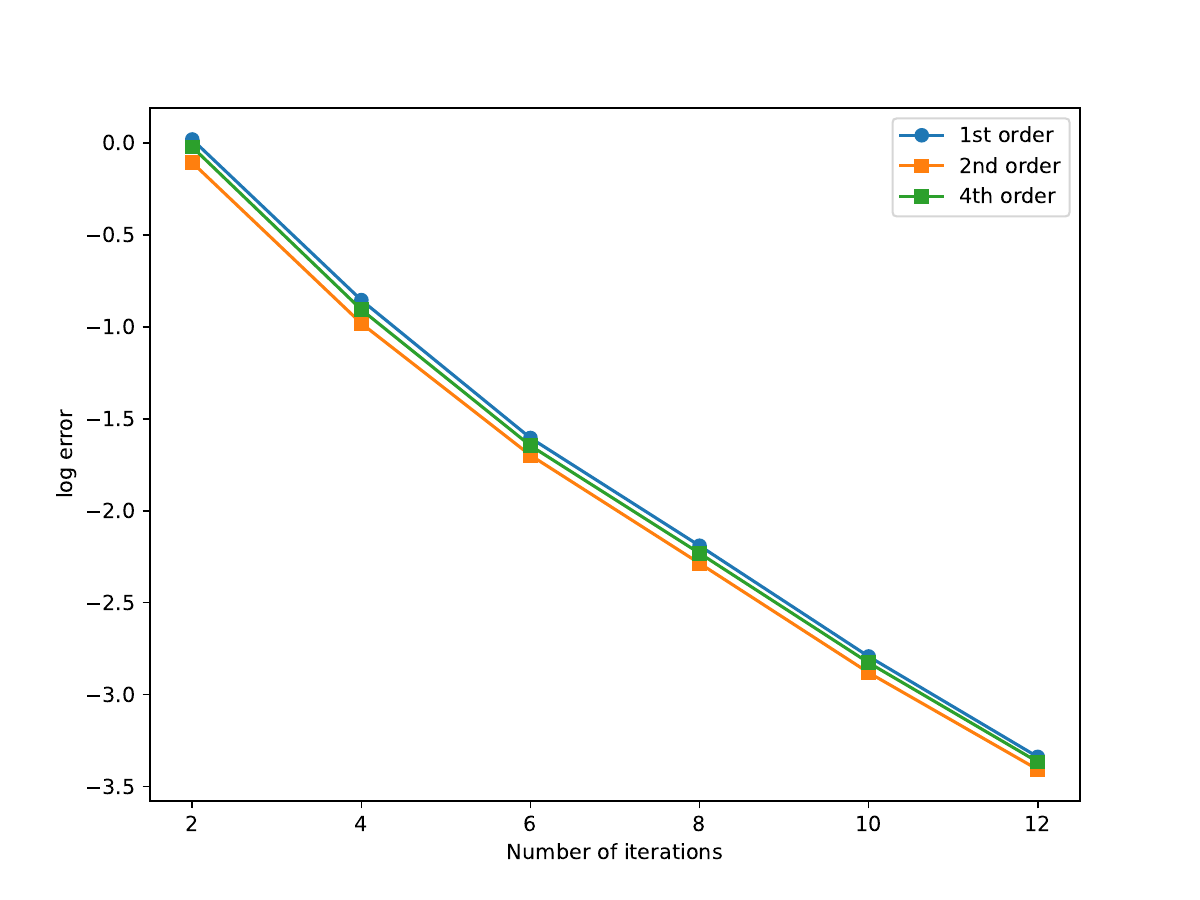}
\caption{
Increasing the time horizon while fixing the time step sizes proportional to the number of policy evaluations required, ensuring
equal total policy
evaluations across all methods. The 1st-order, 2nd-order, and 4th-order methods use step sizes $h$ of $0.25$, $0.5$, and $1$, respectively. The x-axis represents the number of iterations for the 4th-order method. }
\label{subfig:plots2_left}
\end{subfigure}
\hfill
\begin{subfigure}[t]{0.48\textwidth}
\centering
\includegraphics[width=\textwidth]{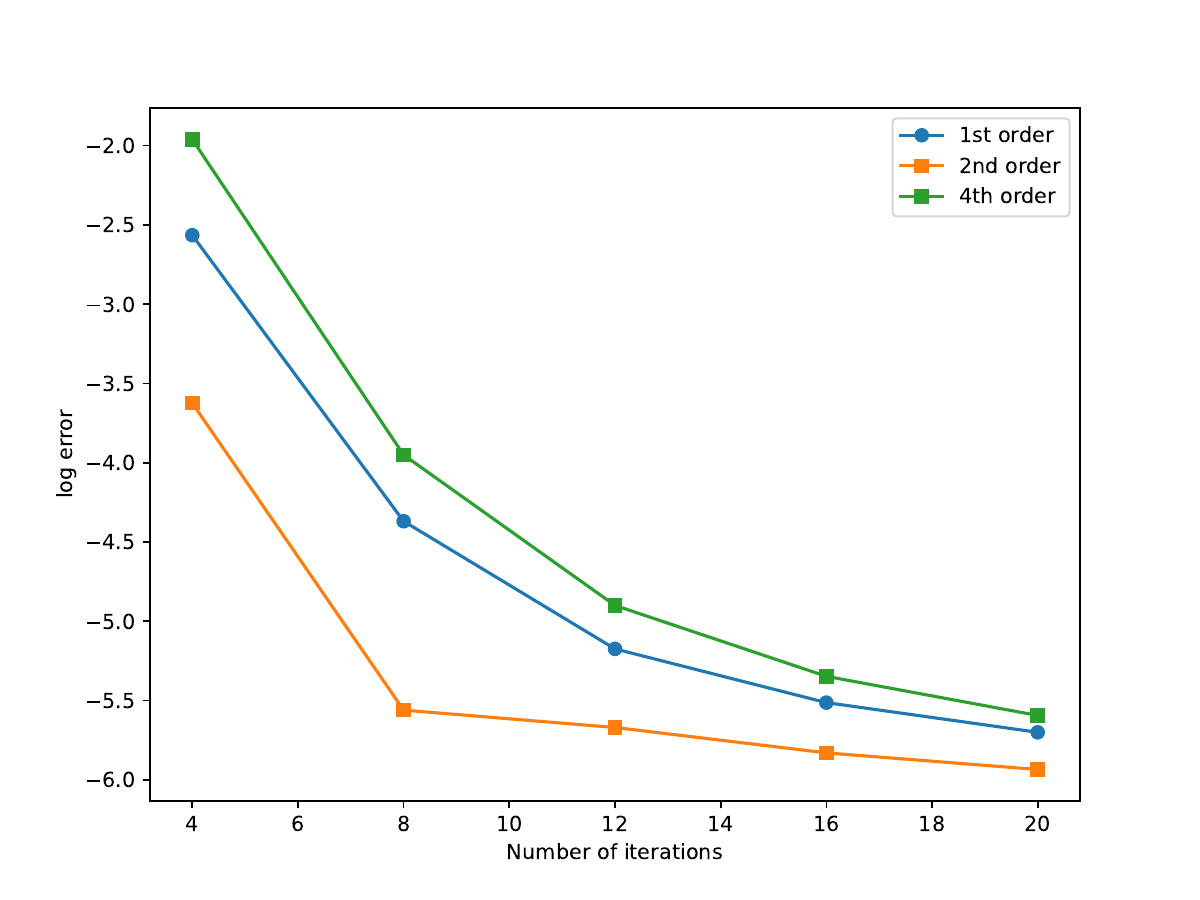}
\caption{
Increasing the number of iterations by reducing the time step size while fixing the time horizon    $T = 25$.
The x-axis represents the number of iterations for the 1st-order method, while the 2nd-order and 4th-order methods use half and one-fourth as many iterations, respectively, ensuring equal total policy evaluations across all methods. }
\label{subfig:plots2_right}
\end{subfigure}
\caption{
Convergence of various Runge–Kutta discretizations of the mirror flow, with all methods sharing the same budget for policy evaluations.}
\label{fig:plots2}
\end{figure}

\bibliographystyle{siam}
\bibliography{bibliography}
\end{document}